\documentclass[11pt]{amsart}
\usepackage{amsmath,amsxtra,amssymb,amsthm,amsfonts,etoolbox,setspace}
\usepackage{color,hyperref}

\numberwithin{equation}{section}
\newtheorem{lemma}{Lemma}[section]
\newtheorem{prop}[lemma]{Proposition}
\newtheorem{theorem}[lemma]{Theorem}

\newtheorem{cor}[lemma]{Corollary}

\newtheorem{prob}[lemma]{Problem}
\newtheorem{rem}[lemma]{Remark}

\newtheorem{conc}[lemma]{Conclusion}

\newcommand{\re}{\begin{rem}\rm}
  \newcommand{\mar}{\end{rem}}
\newtheorem{exam}[lemma]{Example}

\newtheorem{defi}[lemma]{Definition}

\newcommand{\kla}{\left ( }
\newcommand{\mer}{\right ) }

\renewcommand{\for}{\begin{eqnarray*}}

\newcommand{\mel}{\end{eqnarray*}}

\newcommand{\kl}{\pl \le \pl}
\newcommand{\gl}{\pl \ge \pl}

\newcommand{\lel}{\pl = \pl}

\newcommand{\ez}{{\mathbb E}}

\renewcommand{\L}{\mathcal{L}}
\newcommand{\eiz}{{\rm 1}\!\!{\rm 1}}

\newcommand{\nz}{{\mathbb N}}
\newcommand{\nen}{n \in \nz}

\newcommand{\rz}{{\mathbb R}}
\newcommand{\zz}{{\mathbb Z}}

\newcommand{\Mz}{{\mathbb M}}

\newcommand{\cz}{{\mathbb C}}

\newcommand{\ten}{\otimes}

\DeclareMathOperator{\dom}{dom}

\DeclareMathOperator{\Exp}{Exp}

\DeclareMathOperator{\diam}{diam}

\DeclareMathOperator{\Tan}{Tan}

\DeclareMathOperator{\grad}{grad}

\DeclareMathOperator{\Ent}{Ent}

\newcommand{\p}{\hspace{.05cm}}
\newcommand{\pl}{\hspace{.1cm}}
\newcommand{\pll}{\hspace{.3cm}}

\newcommand{\qd}{\end{proof}\vspace{0.5ex}}

\newcommand{\Om}{\Omega}
\newcommand{\om}{\omega}
\renewcommand{\a}{\alpha}
\newcommand{\al}{\alpha}

\newcommand{\si}{\sigma}

\newcommand{\la}{\lambda}
\newcommand{\eps}{\varepsilon}

\newcommand{\E}{{\mathcal E}}
\newcommand{\A}{{\mathcal A}}

\newcommand{\Ha}{{\mathcal H}}

\newcommand{\M}{{\mathcal M}}

\newcommand{\mm}{{\mathbb M}}

\newcommand{\N}{{\mathcal N}}

\newcommand{\I}{{\mathcal I}}

\newcommand{\pf}{\begin{proof}}

\newcommand{\xspace}{\hbox{\kern-2.5pt}}
\newcommand{\xyspace}{\hbox{\kern-1.1pt}}

\newcommand\bra[1]{\langle  #1|}
\newcommand\ket[1]{| #1\rangle}

\newcommand{\ff}{{\mathbb  F}}

\newcommand{\norm}[2]{\parallel \! #1 \! \parallel_{#2}}

\newcommand{\ssubset} {\!\!\subset\! \!}

\allowdisplaybreaks

\definecolor{LightGray}{rgb}{0.94,0.94,0.94}
\definecolor{VeryLightBlue}{rgb}{0.9,0.9,1}
\definecolor{LightBlue}{rgb}{0.8,0.8,1}
\definecolor{DarkBlue}{rgb}{0,0,0.6}
\definecolor{LightGreen}{rgb}{0.88,1,0.88}
\definecolor{MidGreen}{rgb}{0.6,1,0.6}
\definecolor{DarkGreen}{rgb}{0,0.6,0}
\definecolor{DarkGrreen}{rgb}{0,0.8,0}

\definecolor{VeryLightYellow}{rgb}{1,1,0.9}
\definecolor{LightYellow}{rgb}{1,1,0.6}
\definecolor{MidYellow}{rgb}{1,1,0.5}
\definecolor{DarkYellow}{rgb}{0.8,1,0.3}
\definecolor{VeryLightRed}{rgb}{1,0.9,0.9}
\definecolor{LightRed}{rgb}{1,0.8,0.8}
\definecolor{DarkRed}{rgb}{0.8,0.2,0}
\definecolor{DarkRedb}{rgb}{0.6,0.2,0}
\definecolor{DarkLila}{rgb}{0.8,0,1}
\definecolor{Beige}{rgb}{0.96,0.96,0.86}
\definecolor{Gold}{rgb}{1.,0.84,0.}
\definecolor{Goldb}{rgb}{0.7,0.3,0.5}
\definecolor{MyYellow}{rgb}{1.,0.84,0.8}

\newcommand{\lan}{\langle}
\newcommand{\ran}{\rangle}

\begin{document}
\title[CLSI]{Fisher information and Logarithmic Sobolev inequality for matrix valued functions}
\thanks{LG and NL acknowledge support from NSF grant DMS-1700168. NL is supported by Graduate Research Fellowship Program DMS-1144245.  MJ is partially supported by NSF grants DMS-1501103 and DMS 1800872.}

\author[L. Gao]{Li Gao}
\address{Department of Mathematics\\
University of Illinois, Urbana, IL 61801, USA} \email[Li Gao]{ligao3@illinois.edu}

\author[M. Junge]{M. Junge}
\address{Department of Mathematics\\
University of Illinois, Urbana, IL 61801, USA} \email[Marius
Junge]{mjunge@illinois.edu}

\author[N. LaRacuente]{Nicolas LaRacuente}
\address{Department of Physics\\
University of Illinois, Urbana, IL 61801, USA} \email[Nick LaRacuente]{laracue2@illinois.edu}

\maketitle
\begin{abstract} We prove a version of Talagrand's concentration inequality for subordinated sub-Laplacian on a compact Riemannian manifold using tools from noncommutative geometry. As an application, motivated by quantum information theory, we show that on a finite dimensional matrix algebra the set of self-adjoint generators satisfying a tensor stable modified logarithmic Sobolev inequality is dense.
\end{abstract}

\section{Introduction}

Isoperimetric inequalities play an important role in geometry and analysis. In the last decades the deep and beautiful connection between isoperimetric inequalities and functional inequalities has been discovered. This discovery started with the work of Meyer, Bakry and Emry on the famous `carr\'ee du champs' or gradient form, and was brought to perfection by Varopoulos, Saloff-Coste \cite{sfc,sfc2}, Coulhon \cite{VSCC}, Diaconis \cite{DSF},
 Bobkov, G\"otze \cite{BG,BG2}, Barthe and his coauthors \cite{BaK,BMil,BO1,BaK,BCE}, and Ledoux \cite{L1,L2,L3,L4,L5,L6}. It appears that the right framework of this analysis is given by abstract semigroup theory, i.e.\! starting with a semigroup of measure preserving maps on a measure space.

A crucial application of isoperimetric inequalities on compact manifolds is the famous \emph{concentration of measure} phenomenon, used fundamentally  in \cite{FLM}, and analyzed systematically  by   Milman and  Schechtmann, see \cite{MS}. Thanks to the work of Gross \cite{G1,G2,G3,G4}, it is know well-known that concentration of measure can occur in noncommutative spaces and in  infinite dimension in form of a logarithmic Sobolev inequality. Indeed, let $T_t=e^{-tA}$ be a measure preserving  semigroup, acting on $L_{\infty}(\Om,\mu)$ with energy form $\E(f)=(f,Af)$. Then $(T_t)$ (or $A$) satisfies a \emph{logarithmic Sobolev inequality, in short $\la$-LSI}, if
\begin{equation}\label{LSI}  \la \int f^2\log f^2d\mu \kl \E(f)  \end{equation}
holds for all $f$ in the domain of $A^{1/2}$. We will use the notation $\Ent(f)=\int f\log fd\mu$ for  the entropy. To simplify the exposition, we will assume  throughout this paper  that $\mathcal{A}\subset {\rm dom}(A)\cap L_{\infty}$ is a not necessarily closed $^*$-algebra in the domain and invariant under  the semigroup. Semigroup techniques have been very successfully combined with the notion of hypercontractivity
 \[ \|T_t:L_2\to L_{q(t)}\|\kl 1 \quad \mbox{for} \quad q(t)\gl 1+e^{ct}\pl.\]
Indeed, the standard procedure to show that the Laplace Beltrami operator on a compact Riemannian manifold satisfies $\la$-LSI is to derive hypercontractivity from  heat kernel estimates, and then use the Rothaus Lemma to derive LSI from hypercontractivity. In this argument ergodicity of the underlying semigroup appears to be crucial.

A major breakthrough in this development is Talagrand's inequality which connects entropic quantities with a given distance. A triple $(\Om,\mu,d)$ given by a measure and a metric satisfies Talagrand's inequality if
 \[ W_1(f\mu,\mu) \kl \sqrt{  \frac{2\Ent(f)}{\la}} \pl .\]
Here
  \begin{equation}\label{wass} W_1(\nu,\mu)\lel \inf_{\pi} \int d(x,y)d\pi(x,y)
  \lel \sup_{\|g\|_{Lip}\le 1} \big|\int g(x)d\nu(x)-\int g(x)d\mu(x)\big|
  \end{equation}
is the Wassertein 1-distance, and the last equality is a famous duality result by Kantorovitch and Rubenstein, see  \cite{KR,Vi}. The infimum is taken over all probability measures $\pi$ on $\Om\times \Om$ with marginals $\nu$ and $\mu$. Using the triangle inequality for the Wasserstein distance, it is easy to derive the \emph{geometric Talagrand inequality}
  \[ d(A,B)\gl h \quad \Longrightarrow\quad \mu(A)\mu(B)\kl e^{-\frac{h^2}{C}} \pl. \]
If in addition $\mu(A)\gl 1/2$ and  $B_h=\{x| d(x,A)\gl h\}$, this inequality implies exponential decay in $h$, i.e. concentration of measure  usually proved via isoperimetric inequalities. We refer to Tao's blog for applications of  Talalgrand's inequality \cite{TA} in particular to eigenvalues of random matrices \cite{TV1,TV2}.

As pointed out by Otto and Villani \cite{OV}, Talagrand proved a much stronger inequality ($\la$-TA$_2$), namely
 \begin{equation}
 \label{Ta}\la W_2(f\mu,\mu) \kl \sqrt{2\Ent(f)}
 \end{equation}
for $\Om=[0,1]^n$ and the Euclidean distance and  for $\{0,1\}^n$ and the Hamming distance, with a constant $\la$ not depending on $n$. Here the $p$-Wasserstein distance is obtained by replacing the $1$-norm by the $p$-norm in $L_p(\pi)$ in the middle term of \eqref{wass}. Indeed, in the very insightful paper by Otto and Villani \cite{OV}, they point out that the correct way to understand Talagrand's inequality consists in pushing the semigroup into the state space of the underlying commutative $C^*$-algebra. Then Talagrand's concentration inequality can be reformulated as a convexity condition for a suitably defined sub-Riemannian metric. In that sense Otto and Villani reconnect to the geometric aspect of concentration inequalities. The key idea in the Otto-Villani approach is to define a sub-Riemannian matric such that  the function given by \emph{relative entropy}
 \[ D(\nu||\mu) \lel \int \log \frac{d\nu}{d\mu} d\nu \]
admits $T_t(\nu)$ as a path of steepest descent. Here $\frac{d\nu}{d\mu}$ is the Radon-Nikodym  derivative.
A key tool in their analysis was to consider the modified version of the logarithmic Sobolev inequality, (in short $\la$-FLSI)
\begin{equation} \label{FLSI}
  \la \Ent(f) \kl \int A(f)\log f d\mu \pl =:\pl \I_A(f)   \end{equation}
and show that it implies $\la$-TA$_2$. The right hand side is known as Fisher information and  turns out to be the energy functional for the relative entropy with respect to the sub-Riemannian metric.

In this paper we extend the theory of logarithmic Sobolev inequalities in \emph{two directions}, by including matrix valued functions and non-ergodic semigroups. The main road block, discovered in the quantum information theory literature, is that the Rothaus Lemma (\cite{Ro})
 \begin{equation}\label{rot} \Big(\exists_{\la>0} \forall_{E(f)=0}  \pl:\pl \la D(f^2|E(f^2))\kl  \E(f)\Big)
    \quad \stackrel{?}{\Longrightarrow} \quad \Big(\exists_{\tilde{\la}>0} \forall_f \pl:\pl   \tilde{\la} D(f^2|E(f^2))  \kl \E(f)\Big)
    \end{equation}
may fail for matrix-valued functions. The relevance of variations of the Rothaus Lemma is, however, very well-known in the commutative setting, see \cite{BaK}.  We are not aware of any investigation of Talagrand's inequality in the commutative, non-ergodic setting. For selfadjoint semigroups the fixpoint algebra $N=\{x: \forall_t T_t(x)=x\}$ admits a normal conditional expectation $\E_{fix}(M)=N$. This remains true in the noncommutative setting, i.e.\!  for a semigroup $(T_t)$ of (sub-)unital completely positive maps on a finite von Neumann algebra $M$ provided each $T_t$   is selfadjont with respect to the inner product $(x,y)=\tau(x^*y)$ of a normal, faithful tracial state $\tau$. The reader less familiar with von Neumann algebras is welcome to think of $M=L_{\infty}(\Om,\mu;\Mz_m)$, the space of bounded random matrices equipped with $\tau(f)=\int_{\Om} \frac{tr(f(\om))}{m} d\mu(\om)$ such  that $\tau(T_t(x^*)y)=\tau(x^*T_t(y))$.

The failure of \eqref{rot} forces us to introduce new tools, and  study a new gradient condition. This allows us to show that Talagrand's inequality is still valid in the context of group representations. Our main examples are of the form $T_t=S_t\ten id_{\Mz_m}$, where $S_t$ is a nice ergodic semigroup. These examples are natural in the context of operator spaces (see \cite{Po} and \cite{ER} for more background), despite being obviously not ergodic. We say that a semigroup satisfies $\la$-FLSI (also called $\la$-MLSI but we want to stress the Fisher information aspect)  if
 \[ \la D(\rho||\E_{fix}(\rho))\kl \I_A(\rho) \lel \tau(A(\rho)\ln \rho) \pl .\]
The noncommutative Fisher information, introduced under the name entropy production by Spohn \cite{Spo},
is very well-known in the quantum information theory.
At the time of this writing it is not known whether $\la$-FLSI is stable under tensorization. However, tensorization is an important feature and allows to deduce gaussian estimates from an elementary $2$-point inequality, see e.g. \cite{BCL}. Therefore, we introduce the \emph{complete logarithmic Sobolov inequality} (in short $\la$-CLSI) by requiring that $A\ten id_{\Mz_m}$ satisfies $\la$-FLSI for all $m\in \nz$. Using the data processing inequality it is easy to show that the CLSI is stable under tensorization, see section 7.1. Before this paper the list of examples which satisfy good tensorization properties could all be deduced from one key example (see however \cite{DR}), due to  \cite{BaRo}:

\begin{lemma}(Bardet/Rouz\'{e}) Let $E:M\to N$ be a conditional expectation. Then $T_t=e^{-t(I-E)}$ satisfies $1$-CLSI.
\end{lemma}

Indeed, for  conditional expectations we have
 \[ \I_{I-E}(\rho) \lel D(\rho ||E(\rho))+D(E(\rho)||\rho) \gl D(\rho||E(\rho)) \pl .\]
In this case stabilization is trivial. The middle term is the original symmetrized divergence introduced by Kullback and Leibler \cite{KL}, which is interesting from a historical point of view. Using the tensorization, one can now deduce that gaussian systems (and certain depolarizing channels) also satisfy CLSI, see also \cite{DR}. Our new tool to prove CLSI is based on the gradient form
 \[ 2\Gamma(f,g) \lel A(f^*)g+f^*A(g)-A(f^*g)  \pl .\]
We say that the generator $A$ satisfies $\la$-$\Gamma\E$ if
 \[ \sum_{jk} \bar{\al}_j\la \Gamma_{I-\E_{fix}}(f_j,f_k)\al_k \kl  \sum_{j,k}\bar{\al}_j \Gamma_{A}(f_j,f_k)\al_k  \]
holds for all finite families $(f_k)$ and scalars $(\al_k)$. 
The next lemma states the two new basic facts used in this paper.

\begin{lemma}\label{keyel} i) $\la$-$\Gamma\E$ implies $\la$-CLSI. ii) If $A$ and $B$ satisfy $\la$-CLSI, then $A\ten id+id\ten B$ satisfies $\la$-CLSI.
\end{lemma}

Our main contribution is to identify large classes of examples from representation theory satisfying $\Gamma\E$. For this we have to recall the definition of a H\"ormander system on a Riemannian manifold $(\M,g)$ given by vector fields $X=\{X_1,...,X_r\}$ such that  the iterated commutators $[X_{i_1},[X_{i_2},\cdots]]$ generate the tangent $T_p\M$ space at every point $p\in \M$. Building on the famous heat kernel estimates from \cite{SR}, see also \cite{LZ}, and the work of Saloff-Coste on return time, we find entropic concentration inequalities subordinated sub-Laplacians. Quite surprisingly, qualitative forms of non-locality turn out to be helpful in this context.

\begin{theorem} Let $X$ be a H\"ormander system on a compact Riemannian manifold, and the selfadjoint generator  $\Delta_X=\sum_j \nabla_{X_j}^*\nabla_{X_j}$ the corresponding sub-Laplacian. Then $\Delta_X^{\theta}$ satisfies $\la$-$\Gamma\E$, and $\la$-CLSI for some constant $\la=\la(X,\theta)$.
\end{theorem}

It is widely open wether the $\Delta_X$ itself satisfies CLSI, even when $\Delta_X$ is replaced by the Laplace-Beltrami operator on a compact Riemannian manifold.  For $\Om=S^1$ the standard semigroup given by $A=-\frac{d^2}{dx^2}$ satisfies $1$-CLSI, however fails $\la$-$\Gamma\E$. For more information on Barkry-Emery theory for sub-Laplacians, we refer to the deep work work of Baudoin and his coauthors \cite{B0,B1,B2,B3,B4,BOV}. Subordinated semigroups(in a slightly different meaning) have also been investigated in the gaussian setting, see \cite{Lind1,Lind2}. From a rough kernel perspective (see \cite{GF1,FP1,FP2}) it may appear lear surprising that subordinated semigroups outperform their smooth counterparts.

In the context of group actions, we can transfer logarithmic Sobolev inequalities. Indeed, let $\al:G\to {\rm Aut}(M)$ be a trace preserving action on a finite von Neumann algebra $(M,\tau)$, i.e. \! $\al$ is strongly continuous group homomorphism with values in the set of trace preserving automorphism on $M$.
A semigroup $S_t:L_{\infty}(G)\to L_{\infty}(G)$ which is invariant under right translations is given by an integral operator of the form
 \[ S_t(f)(g) \lel \int k_t(gh^{-1}) f(h) d\mu(g) \]
 where $\mu$ is the Haar measure. We will assume that $G$ is compact and $\mu$ is a probability measure. Then we may define the \emph{transferred semigroup}
  \[ T_t(x) \lel \int k_t(g) \al_g(x) d\mu(x) \pl .\]
For ergodic $S_t$ the fixpoint algebra of the transferred  semigroup $T_t$ is then given by fixpoint algebra of the action $N^G=\{x|\forall_{g} \al_g(x)=x\}$, which is generally not trivial.

\begin{theorem} Let $G$ be a compact group acting on a finite von Neumann algebra $(M,\tau)$. It the generator $S_t=e^{-tA}$ satisfies $\la$-CLSI ($\la$-$\Gamma\E$), then $T_t$ satisfies $\la$-CLSI ($\la$-$\Gamma\E$).
\end{theorem}

For a compact Lie group $G$
a generating set $X=\{X_1,...,X_r\}$ of the Lie-algebra $\mathfrak{g}$ defines
a H\"ormander system $X=\{X_1,...,X_r\}$ given by the corresponding right translation invariant vector fields. Then we conclude that for \emph{any} group representation the semigroup $T_t^{\theta}$ transferred from $S_t=e^{-t\Delta_X^{\theta}}$ satisfies $\la$-$\Gamma\E$. Our investigation is motivated by quantum information theory and previous results. Starting with the seminal papers \cite{DaLi,OZ}, Temme and his coauthors
\cite{T1,T2,T3,T4,T5} made hypercontractivity in matrix algebras available in the ergodic setting, see \cite{JPPP,JPPP2,JPPPR,RiX} for results in group von Neumann algebras.  Using the concrete description, the so-called Lindblad generators,  we can prove the following density result:

\begin{theorem} The set of selfadjoint generators of semigroups in $\Mz_m$  satisfying $\Gamma\E$ and CLSI is dense.
\end{theorem}

Indeed, combining all the results from above, we can  show that  for such generators $A$ the subordinated  $A^{\theta}$ satisfies $\la(\theta)$-$\Gamma\E$ for all $0<\theta<1$. Let us mention the deep work of Carlen-Maas \cite{CM}. They were able to tanslate the work of Otto-Villani \cite{OV} to the state space of matrices and identify a truly noncommutative Wasserstein $2$-distance  $d_{A,2}(\rho,\si)$. They also showed (in the ergodic setting) that $\la$-FLSI implies
 \[ d_{A,2}(\rho,E(\rho))\kl 2\sqrt{ \frac{D(\rho||E(\rho))}{\la}} \pl .\]
An analogue of an intrinsic Wasserstein distance has already been introduced in \cite{JZ,JRZ}:
 \[ d_{\Gamma}(\rho,\si)
 \lel \sup_{f=f^*, \Gamma_A(f,f)\le 1} |\tau(\rho f)-\tau(\si f)| \pl. \]
Based on \cite{JRS} we show that $d_{\Gamma}\kl 2\sqrt{2}d_{A,2}$, and hence we see that $\la$-FLSI does indeed imply a noncommutative geometric Talagrand inequality: Let $e_1$ and $e_1$ be projections in $M$ such for some test function $f$ with $\Gamma_A(f,f)\le 1$ we have
 \[ |\frac{\tau(e_1f)}{\tau(e_1)}-\frac{\tau(e_2f)}{\tau(e_2)}|\gl h \quad \Longrightarrow \quad \tau(e_1)\tau(e_2)\kl e^{-h^2/C}\pl ,\]
where $C$ only depends on the $\la$-FLSI constant of the generator $A$. Thus we have identified large classes of new examples with satisfy the Talagrand's concentration inequality, not only for $T_t=e^{-tA}$, but also for the $n$-fold tensor product $(T_t^{\ten_n})$.

The paper is organized at follows: We discuss gradient forms, derivations and Fisher information in section 2. In section 3 we first consider kernel and decay time estimates in the ergodic, and  then in the  non-ergodic case. The latter analysis relies on the theory of mixed $L_p(L_q)$ spaces from \cite{JPar}, which has been recently rediscovered in \cite{BaRo}.
In section 4 we discuss group representations and the density result in section 5.  Section 6 is devoted to geometric applications and deviation inequalities. In Section 7 we discuss examples and counterexamples. A chart of the different properties considered in this paper is given in the following diagram
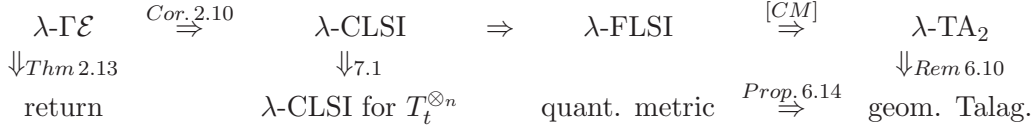
\[ \begin{array}{ccccccccccc}
 \mbox{$\la$-$\Gamma \E$} &\stackrel{Cor. \p 2.10}{\Rightarrow}&
 \mbox{$\la$-CLSI}&\stackrel{}{\Rightarrow}& \mbox{$\la$-FLSI}&\stackrel{[CM]}{\Rightarrow}& \mbox{$\la$-TA$_2$} \\
 \Downarrow_{Thm \p 2.13} & &  \Downarrow_{7.1} & &  & & \Downarrow_{Rem \p 6.10} \\
 \mbox{return}
  & &\mbox{$\la$-CLSI for $T_t^{\ten_n}$} & &
 \mbox{quant. metric} &\stackrel{Prop. \p 6.14}{\Rightarrow}& \mbox{geom. Talag.}
 \end{array} \]
Here the return time estimate of the form
$\|T_t(\rho)-E(\rho)\|_1\le e^{-\la t}\|\rho-E(\rho)\|_1$ is inspired by the work of Saloff-Coste \cite{sfc}.
Open problems will be mentioned at the end of section 7. In fact, we expect CLSI to hold for  matrix valued functions with respect to smooth generators of semigroups. Due to space restrictions we ignore the deep and interesting connection to free Fischer information.

{\bf Aknowledgment:} The starting point of this paper is a conversation by I. Bardet and the second named author during the IHP program ``Analysis in Quantum Information'', Fall 2017. The authors are thankful for disccusions with Ivan Bardet, David Stilk Fran\c{c}a, and Nilanjana Datta.


\section{Gradient forms and Fisher information}
\subsection{Modules and Gradient forms}
Let $(M,\tau)$ be a finite von Neumann algebra $M$ equipped with a normal faithful tracial state $\tau$. We denote the noncommutative $L_p$-spaces by $L_p(M,\tau)$ or $L_p(M)$ if the trace is clear from the context. Throughout the paper we consider that $\displaystyle T_t=e^{-tA}:M\to M$ is a strongly continuous semigroup of completely positive, unital and self-adjoint maps.  Then $\tau(x^*T_t(y))=\tau(T_t(x)^*y)$ for all $x,y\in M$ and hence $T_t$ is also trace preserving. The generator $A$ is the (possibly unbounded) positive operator on $L_2(M,\tau)$ given by $Ax=\lim_{t\to 0} \frac{1}{t}(T_t(x)-x)$ (see \cite{CS} for more background.) We will assume that there exists a weakly dense $^*$-subalgebra $\mathcal{A}\subset M$ such that
 \begin{enumerate}
 \item[i)] $\A\subset dom(A)\cap \{ x | A(x)\in M\}$;
 \item[ii)] $T_t(\A)\subset \A$ for all $t>0$.
 \end{enumerate}
 In most cases, it is enough to assume that $\A\subset {\rm dom}(A^{1/2})$ and the $\Gamma$-regularity from \cite{JRS}. The gradient form of $A$ is defined as
\[ \Gamma(x,y)(z) :\lel \frac{1}{2}\Big(\tau(A(x)^*yz)+ \tau(x^*A(y)z)-\tau(x^*yA(z))\Big) \pl\]
We say the generator $A$ satisfies $\Gamma$-regularity if its gradient form $\Gamma(x,y)\in L_1(M,\tau)$ for all $x,y\in {\rm dom}(A^{1/2})$.
\begin{theorem}\label{JRS1} Suppose $\Gamma(x,x)\in L_1(M,\tau)$ for all $x\in {\rm dom}(A^{1/2})$. Then there exists a finite von Neumann algebra $(\hat{M},\tau)$ containing $M$, and a self-adjoint derivation $\delta:{\rm dom}(A^{1/2})\to L_2(\hat{M})$ such that \[ \tau(\Gamma(x,y)z) \lel \tau(\delta(x)^*\delta(y)z) \pl.\]
Equivalently, $E_M(\delta(x)^*\delta(y))=\Gamma_A(x,y)$ where $E_M:\hat{M} \to M$ is the conditional expectation.
\end{theorem}

The $\Gamma$-regularity allows us to define the Hilbert $W^*$-module $\Om_{\Gamma}$ as the completion of ${{\rm dom}(A^{1/2})\ten M}$ with $\Gamma$-inner product
 \[ \lan x_1\ten x_2,y_1\ten y_2\ran_{\Gamma}
 \lel x_2^*\Gamma(x_1,y_1)y_2 \pl .\]
Note that in Theorem \ref{JRS1} the completion $\hat{M}$ with respect to $M$-valued inner product
 \[ \langle x,y\rangle_{E_M} :\lel E_M(x^*y)\pl, \pl x,y\in \hat{M}\pl\]
also gives a $W^*$-module, which we denote as $L_\infty^c(M\subset \hat{M})$.
Recall that for a $W^*$-module $\Ha$ of $M$, the norm is given by its $M$-valued inner product $\lan \cdot, \cdot\ran_\Ha$ as follows,
 \[ \|\xi\|_{\Ha}\lel \|\langle \xi,\xi\rangle_\Ha\|_{M}^{1/2} \pl .\]
Then it is readily to verify that the map
 \[ \phi:\Om_{\Gamma}\to L_\infty^c(M\subset \hat{M})\pl, \pl \phi(x\ten y) \lel \delta(x)y \]
is an isometric right $M$-module map. Moreover, thanks to Theorem  of \cite{Pasch} (see also \cite{JS}) the range $\phi (\Om_\Gamma)$ is $1$-complemented in $L_\infty^c(M\subset \hat{M})$.
\begin{rem}{\rm Our notation $\Om_\Gamma$ is motivated by the universal $1$-forms bimodule \[  \Omega^1 \A \lel \{x\ten y-1\ten xy\pl |\pl x,y \in \A\} \subset \A\ten \A \pl \]
(see \cite{NCG}). Indeed the map $\phi$ induces a representation of the universal derivation $\delta(x)=x\ten 1-1\ten x$.}
\end{rem}

We refer to \cite{JRS} for the proof of Theorem \ref{JRS1}. Here we discuss the following special case. Let $T:M\to M$ be a unital completely positive self-adjoint map. Recall that the $W^*$-module $M\ten_{T}M$ is given by GNS-construction
 \[ \langle x_1\ten x_2,y_1\ten y_2\rangle_T
  \lel x_2^*T(x_1^*y_1)y_2 \pl .\]
It is shown in \cite{JRS}(see also \cite{Speich}) that there exists a finite von Neumann algebra $\hat{M}$ containing $M$ and a self-adjoint element $\xi\in L_2(\hat{M})$ of norm $1$ such that
 \[  T(x)\lel E_M(\xi x\xi) \pl .\]
This gives an isometric $M$-module map $\phi_1:M\ten_{T}M \to L_\infty^c(M\subset \hat{M})$
 \[ \phi_1(x\ten y) \lel x\xi y \pl \pl.\]
On the other hand, let $I:M\to M$ be the identity map. The map $A=I-T$ is a generator of a semigroup of completely positive maps (see \cite{Jesse} for a characterization of generators of trace preserving self-adjoint semigroups). Its gradient form is
\[\Gamma_{I-T}(x,y)(z)=\frac{1}{2}(x^*y-T(x)^*y-x^*T(y)+T(x^*y))\pl.\]
One has another isometric module map
\[ \phi_2:\Om_{\Gamma_{I-T}}\to M\ten_{T}M \pl ,\pl \phi_2(x\ten y) \lel x\ten y-1\ten T(x)y \pl .\]
Then for the special case $A=I-T$, one can choose the derivation as follows,
\[\delta:M\to \hat{M}\pl, \pl \delta(x)=\phi_1\circ \phi_2 (x\ten 1)= x\xi-\xi T(x)\pl.\]

In the following we will use the complete positive order in two ways. For two completely positive maps $T$ and $S$, we write $T\le_{cp}S$ if $S-T$ is completely positive. For two gradient forms $\Gamma,\Gamma'$,
we write $\Gamma\le_{cp}\Gamma'$ if
\[ [\sum_k \Gamma(x_{ik},x_{kj})]_{i,j}\kl [\sum_k \Gamma'(x_{ik},x_{kj})]_{i,j} \]
holds for all $M$-valued matrices $(x_{jk})$ in the domain of $\Gamma$.
\begin{lemma}\label{grad1}
i) Let $T:M \to M $ be a completely positive unital map. Then for any state $\rho $,
     \[ \rho(\lan \xi,\xi\ran_{\Gamma_{I-T}})
     \lel \inf_{c\in M} \rho(\lan \xi-1\ten c,\xi-1\ten c\rangle_{T}) \pl . \]
ii) Let $T_1,T_2: M \to M$ be two completely positive unital maps. Then $\la T_1\le_{cp} T_2$ implies $\la \Gamma_{I-T_1}\le_{cp}  \Gamma_{I-T_2}$.
\end{lemma}
\begin{proof} We choose a module basis $\{\xi_i\}_{i\in I}$ of $M\ten_TM$ (see \cite{Pasch,JS}) and let $\xi_0=1\ten 1$. Then \[ \langle \xi_i,\xi_j \rangle_T \lel \delta_{ij}e_i \pl,\]
where $(e_i)_{i\in I} \subset M$ is a family of projections and $e_0=1$. Let $P: M\ten_{T}M\to 1\ten M$ be the orthogonal projection given by
 \[ P(\xi) \lel P(\sum_i \al_i\xi_i) \lel \al_0\xi_0 \pl, \pl\pl \al_i\in M\]
Note that
 \[ \langle 1\ten z,x\ten y\rangle_T
 \lel z^*T(x)y \lel \langle 1\ten z,1\ten T(x)y\rangle  \pl.\]
This implies that $P(x\ten y)= 1\ten T(x)y$ and hence for $\xi=\sum_i \al_i\xi_i$, we find that
  \[\lan \xi,\xi\ran_{\Gamma_{I-T}}=\lan \phi_2(\xi),\phi_2(\xi)\ran_{T}=\lan\xi-P(\xi),\xi-P(\xi)\ran_{T}=\sum_{i\neq 0} \al_i^*e_i\al_i  \pl .\]
On the other hand
 \[ \lan\xi-1\ten c, \xi-1\ten c\ran_T
   \lel (\al_0-c)^*(\al_0-c)+\sum_{i\neq 0} \al_i^*e_i\al_i\pl .\]
This exactly implies that for any $c\in M$,
\[\lan\xi,\xi\ran_{\Gamma_{I-T}}\le \lan\xi-1\ten c, \xi-1\ten c\ran_T\]
and clearly for any state $\rho$,
 \[ \rho(\langle \xi,\xi\rangle_{\Gamma_{I-T}})
  \lel \inf_{c\in N} \rho(\langle \xi-1\ten c, \xi-1\ten c\rangle_T) \pl \pl.\]
For b), we find a $c\in M$ such that for any state $\rho$,
 \begin{align*}
  \rho(\langle \xi,\xi\rangle_{\Gamma_{I-T_2}})
 &=  \rho(\langle \xi-1\ten c,\xi-1\ten c\rangle_{T_2}) \\
 &\gl  \la \rho(\langle \xi-1\ten c,\xi-1\ten c\rangle_{T_1})
 \gl \la  \rho(\langle \xi,\xi\rangle_{\Gamma_{I-T_1}})  \pl.
 \end{align*}
Here we used a) twice. The argument for arbitrary matrix levels is the same. \qd

Our next observation is based on operator integral calculus. (see \cite{Su} and references therein for more information). Let $F:\rz \to \rz$ be a continously differentiable function and $\delta$ be a derivation in Theorem \ref{JRS}. Then for a positive $\rho$, the functional calculus for $\delta$ is given by the following operator integral,
\begin{equation}\label{derv}
 \delta(F(\rho)) \lel \int_{\rz_+}\int_{\rz_+}
 \frac{F(s)-F(t)}{s-t} dE_s^{\rho}\delta(\rho)dE_t^{\rho}  \pl .
 \end{equation}
 where $E^{\rho}((s,t])=1_{(s,t]}(\rho)$ is the spectral projection of $\rho$. Indeed, this is obvious for monomials
 \[ \delta(\rho^n)
 \lel \sum_{j=0}^{n-1} \rho^j\delta(\rho)\rho^{n-j-1}
 \lel \int_{\rz_+}\int_{\rz_+} \frac{s^n-t^n}{s-t} dE_s^{\rho}\delta(\rho)dE_t^{\rho} \pl .\]
The convergence of \eqref{derv} in $L_2$ follows from the boundedness of the derivative $F'$ (and in $L_p$ from the theory of singular integrals, see again \cite{Su}). Let us introduce the \emph{double operator integral}
 \[ J_F^{\rho}(y) \pl:=\pl  \int_{\rz_+}\int_{\rz_+} \frac{F(s)-F(t)}{s-t} dE_s^{\rho}ydE_t^{\rho} \pl .\]
For $\rho=\sum_k \rho_k e_k$ with discrete spectrum, this  simplifies to a Schur multiplier
\[J_F^{\rho}(y)\lel \sum_{k,l} \frac{F(\rho_k)-F(\rho_l)}{\rho_k-\rho_l} e_kye_l\pl.\]

\begin{lemma}\label{mm}   Let $F:\rz \to \rz$ be a continuously differentiable monotone function and $\rho\in M$ be positive. Let $A$ and $B$ be two generators of semigroups on $M$ with corresponding deviation $\delta_A$ and $\delta_B$. Suppose their gradient forms satisfy
 \[ \la\Gamma_A \pl \le_{cp} \pl  \Gamma_B \pl .\]
Then ${\rm dom}(B^{\frac{1}{2}})\subset {\rm dom}(A^{\frac{1}{2}})$ and for any $x\in {\rm dom}(B^{\frac{1}{2}})$,
 \[ E_M(\delta_A(x)^*J_F^{\rho}(\delta_A(x)))
  \kl \la  \pl E_M(\delta_B(x)^*J_F^{\rho}(\delta_B(x))) \pl .\]
\end{lemma}

\begin{proof} Let us first assume that $\rho=\sum_k \la_k e_k$ has discrete spectrum.  Then for any $d\in M$,
 \begin{align*}
 \tau(E_M(\delta_A(x)^*J_F^{\rho}(\delta_A(x)))dd^*)
 &= \sum_{k,l}  \frac{F(\la_k)-F(\la_l)}{\la_k-\la_l}
 \tau(\delta_A(x)^*e_k\delta_A(x)e_ldd^*)\\
 &= \sum_{k,l} \frac{F(\la_k)-F(\la_l)}{\la_k-\la_l}
  \|e_k\delta_A(x)e_ld\|_{L_2(\hat{M}_A)}^2 \pl .
  \end{align*}
Recall that $\Om_{\Gamma_A}$ (resp. $\Om_{\Gamma_B}$) is a submodule of $L_\infty^c(M\subset \hat{M}_A)$ (resp. $L_\infty^c(M\subset \hat{M}_B)$) and hence there is a $M$-module projection $P_B$ onto $\Om_{\Gamma_B}$. Our assumption implies that the map
\[ \Phi: \Om_{\Gamma_B}\to L_\infty^c(M\subset \hat{M}_A)\pl, \pl  \Phi(x\ten y)=\delta_A(x)y \]
is of norm less than $\sqrt{\la}$. Thus $\norm{\Phi \circ P_B:L_2(\hat{M}_B)\to L_\infty(\hat{M}_A)}{}$ is also less than $\sqrt{\la}$. Moreover, $\Phi$ is also a left $\A$-module map:
 \begin{align*}
  \Phi(a \delta_B(x) y)
 &=  \Phi(\delta_B(ax)y-\delta_B(a)xy)
 \lel\delta_{A}(ax)y-\delta_A(a)xy \lel
  a\delta_A(x)y \pl .
  \end{align*}
Using strongly converging bounded nets from the weak$^*$-dense algebra $\A\subset \dom(A^{1/2})$, we deduce that $\Phi$ extends to a $M$-bimodule map, and hence for all $k,l$,
 \[ \|e_k\delta_A(x)e_ld\|_{L_2(M_A)} \kl \sqrt{\la} \|e_k\delta_B(x)e_ld\|_{L_2(M_B)} \pl. \]
 Since $F$ is increasing, $\frac{F(\la_k)-F(\la_l)}{\la_k-\la_l}$ is positive. Therefore we obtain
 \[\tau(E_M(\delta_A(x)^*J_F^{\rho}(\delta_A(x)))dd^*)
 \kl \la \tau(E_M(\delta_B(x)^*J_F^{\rho}(\delta_B(x)))dd^*) \]
for all $dd^*\gl 0$. This implies the assertion for $\rho$ with discrete spectrum.

Let $\rho\in \M$ be a general positive element. Then we can approximate $F(s,t)=\frac{F(s)-F(t)}{s-t}$ by the sequence
 \[  F_n(s,t) \lel \frac{F(n\lfloor
 \frac{s}{n}\rfloor) -F(n\lfloor \frac{t}{n}\rfloor)}{n(\lfloor \frac{s}{n}\rfloor)-n(\lfloor \frac{t}{n}\rfloor)} \pl,\]
 and find $\displaystyle\lim_{n\to \infty} J_{F_n}^{\rho}(x)=J_{F}^{\rho}(x)$. \qd

 \subsection{Fischer Information}
Recall that the Fisher information of a generator $A$ is defined as
 \[ {\I}_A(x) \lel \tau(A(x)\ln(x)) \pl, \pl x\in \A \cap M_+\pl, \]
provided  $A(x)\ln x\in L_1(M)$. Equivalently, one can define $\displaystyle {\I}_A(x)=\lim_{\eps\to 0}\tau(A(x)\ln(x+\eps 1))$.
In the quantum information theory literature $\mathcal{I}_A$ is also called \emph{entropy production} (see \cite{Spo}).
\begin{cor}\label{Fisher} Let $\Gamma_A,\Gamma_B$ be the gradient forms of two
semigroups on $M$. Suppose $\Gamma_A\le_{cp} \la \Gamma_B$. Then
\[ \la \mathcal{I}_A(x) \kl  \mathcal{I}_B(x) \pl .\]
\end{cor}

\begin{proof} Let $x\in {\rm dom}(B^{1/2})\cap M$. Then by differentiation formula , \[\|B^{1/2}(x^*x)\|_2=\|\delta_B(x^*x)\|_2=\|\delta_B(x^*)x+x^*\delta_B(x)\|_2\pl,\] Hence $x^*x$ is also the domain of $B^{1/2}$. Thus we have enough positive elements in ${\rm dom}(B^{1/2})\cap M$. Take the function $F(t)=\ln t$. Then
 \begin{align*}
     \tau(B(x)\ln(x+\eps 1))
 &=  \tau(\delta_B(x)\delta_B(\ln(x+\eps 1))) \lel  \tau(\delta_B(x)I_F^{x+\eps 1}(\delta_B(x))) \\
 &\gl \la \tau(\delta_A(x)I_F^{x+\eps 1}(\delta_A(x))) \lel  \la \tau(A(x)\ln(x+\eps 1)) \pl .
 \end{align*}
The assertion following from sending $\eps\to 0$.\qd

Let $N\subset M$ be a von Neumann subalgebra and $E_N$ be the conditional expectation. (Or shortly $E$ for $E_N$ if no ambiguity.) We define the Fisher information for the subalgebra $N$ with help of the generator $I-E$,
 \[ \I_N(\rho) \lel \I_{I-E}(\rho)= \tau\Big((\rho-E(\rho) ) \ln \rho \Big)\pl .\]
Recall that for two positive elements $\rho,\si\in M$, the relative entropy is
 \[ D(\rho||\si) :\lel \begin{cases}
                        \tau(\rho \ln \rho)-\tau(\rho \ln \si), & \mbox{if } \rho\ll \si \\
                        +\infty, & \mbox{otherwise}.
                      \end{cases} ,\]
Equivalently one can define $D(\rho||\si)=\lim_{\delta\to 0}D(\rho||\si+\delta 1)$.
When $\tau(\rho)=\tau(\si)$, $D(\rho||\si)$ is always positive. The relative entropy with respect to $N$ is defined as
 \[ D_N(\rho)\lel D(\rho||E(\rho))\lel \inf_{\si \in N, \tau(\si)=\tau(\rho)} D(\rho||\si) \pl .\]
See \cite{Marv,GJLR2} for more information on $D_N$ as an asymmetry measure. The following result is well-known (see \cite{Spo,BaRo}), but the simple proof is crucial for this paper.

\begin{lemma}\label{fis} The Fisher information satisfies
 \[ \I_N(\rho) \lel D(\rho||E(\rho))+D(E(\rho)||\rho) \]
and hence $\mu_N\le I_N$.
\end{lemma}
\begin{proof} We first note that
 \begin{align*}
  \I_{N}(\rho)&=\tau(\rho \ln \rho)-\tau(E(\rho)\ln \rho) \\
  &=\tau(\rho \ln \rho)-\tau(\rho \ln E(\rho))+ \tau(E(\rho)\ln E(\rho))-\tau(E(\rho)\ln \rho) \\
  &= D(\rho||E(\rho))+D(E(\rho)||\rho) \pl .
  \end{align*}
The positivity of the relative entropy implies the assertion.\qd

Now let $T_t=e^{-At}:M \to M$ be a semigroup of completely positive unital maps and
$N\subset M$ be the fixed-point algebra of $T_t$. It is easy to see that \[E\circ T_t=T_t\circ E=E\pl.\] The Fisher information $I_A$ appears as the negative derivative of the asymmetry measure $D_N$ under the semigroup $T_t$.

\begin{prop}\label{dmu} Suppose that
 \[ \la D_N(\rho)\kl \I_A(\rho) \pl, \pl \forall \pl \rho\ge 0\pl.\]
Then
 \[ D_N(T_t(\rho)) \kl e^{-\la t}D_N(\rho)\pl, \pl \forall \pl \rho\ge 0.\]
\end{prop}

\begin{proof} Take  $f(t) =  D_N(T_t(\rho))$. The idea is to differentiate
 \begin{align*}
  f(t)  &=  D_N(T_t(\rho))
 =  \tau(T_t(\rho)\ln T_t(\rho))-
 \tau(E(T_t(\rho))\ln E(T_t(\rho))) \\
 &=
 \tau(T_t(\rho)\ln T_t(\rho))-\tau(E(\rho)\ln E(\rho)) \pl .
 \end{align*}
 It was proved in \cite{Su} that for a function $F:\rz_+ \to \rz$ that
 \[ \lim_{s\to 0} \frac{F(\rho+s \si)-F(\rho)}{s}
 \lel J^{\rho}_{F}(\si) \pl,\]
and hence for the trace
 \[ \lim_{s\to 0} \frac{\tau(F(\rho+s \si)-\tau(F(\rho))}{s} \lel \tau(J^{\rho}_{F}(\si))\lel  \tau(F'(\rho)\si) \pl . \]
 Note that \[\lim_{s\to 0}\frac{1}{s}(T_{t+s}(\rho)-T_{t}(\rho))=-A(T_t(\rho))\pl.
\] Now we use the chain rule for $F(s)=s\ln s$  and $F'(s)=1+\ln s$ and deduce that
 \begin{align*}
   f'(t) &=  \tau\Big(-A(T_t(\rho))\Big)+\tau\Big(-A(T_t(\rho))\ln\big(T_t(\rho)\big)\Big) \lel
    -\mathcal{I}_A(T_t(\rho)) \pl .
  \end{align*}
Here the first term vanishes because $A$ is self-adjoint and $A(1)=0$. Thus the assumption implies that \[f'(t)=-\mathcal{I}_A(T_t(\rho))\kl -\la D_N(T_t(\rho))=-\la f(t).\]
Then the assertion follows from Gr\"onwall's Lemma.
\qd

\begin{defi} The semigroup $T_t$ or the generator $A$ with fixed-point algebra $N$ is said to satisfy:
\begin{enumerate}
\item[a)]the gradient condition $\la$-$\Gamma \E$ if there exists a $\la>0$ such that
  \[ \la \Gamma_{I-E_{N}} \pl \le_{cp}\pl  \Gamma_A \pl .\]
\item[b)] the generalized logarithmic Sobolev inequality $\la$-FLSI if
 \[ \la D(\rho||E_N(\rho))\kl \I_A(\rho) \pl .\]
\item[c)] satisfies the \emph{complete} logarithmic Sobolev inequality  $\la$-CLSI if $A\ten id_{\Mz_m}$ satisfies $\la$-FLSI for all $m\in \nz$
\end{enumerate}
We also say that $T_t$ or $A$ has $\Gamma \E$ (resp. FLSI, CLSI) if it satisfies $\la$-$\Gamma\E$ (resp. $\la$-FLSI, $\la$-CLSI) for some $\la>0$.
\end{defi}
An immediate Corollary is that $\la$-$\Gamma \E$ implies $\la$-CLSI.
\begin{cor}\label{kkeeyy} If the generator $A$ satisfies $\la$-$\Gamma \E$, then
$A\ten id_{M}$ has $\la$-$\Gamma\E$ for any finite von Neumann algerba $M$.
In particular, $\la$-$\Gamma \E$ implies $\la$-CLSI.
\end{cor}
\begin{proof} For the first, easy, assertion see
\cite{JZ}, for the second Proposition \ref{dmu}.
\qd

In the rest of this section we discuss some interesting consequences of $\Gamma E$ condition. The first result is motivated by the original definition of relative entropy as a symmetric metric.

\begin{prop}\label{dIA} Let $(T_t):M\to M$ be a semigroup of completely positive  unital self-adjoint maps with fixed-point algebra $N\subset M$. Then
 \[ \la \mathcal{I}_N(\rho)\kl \mathcal{I}_A(\rho) \]
implies
 \[ \mathcal{I}_N(T_t(\rho)) \kl e^{-\la t}\mathcal{I}_N(\rho) \pl .\]
\end{prop}

\begin{proof} Let us consider the function
\begin{align*} f(t) \lel I_N(T_t(\rho)) &\lel D(T_t(\rho)||E(T_t(\rho))+D(E(T_t(\rho))||T_t(\rho))\\ &\lel D(T_t(\rho)||E(\rho))+D(E(\rho)||T_t(\rho))\pl. \end{align*}
We have seen in Proposition \ref{dmu} that the derivative of the first term is $-\mathcal{I}_N(T_t(\rho))$.
This leads us to differentiate $h(t)=\tau(E(\rho)\ln T_t(\rho))$. Let us recall that for $F(x)=\ln x$ the double operator integral $J^{\rho}_{\ln}$ has a specific form as follows, (see \cite{CM})
\[ J^{\rho}_{\ln}(y) \lel \int_0^{\infty}(s+\rho)^{-1}y(s+\rho)^{-1} ds \pl .\]
Write $\rho_t=T_t(\rho)$. We find that
 \begin{align*}
  h'(t) &=  \tau(E(\rho)J^{T_t(\rho)}_{\ln}(-AT_t(\rho)))\\
  &\lel - \int_0^\infty \tau((s+\rho_t)^{-1}E(\rho)(s+\rho_t)^{-1}AT_t(\rho))ds \\
 &=  - \int_0^\infty \tau(\delta((s+\rho_t)^{-1}E(\rho)(s+\rho_t)^{-1})
 \delta(\rho_t)) ds \pl ,
 \end{align*}
 where $\delta$ is the deviation of $A$. We use the Leibniz rule and deduce from $\delta(E(\rho))=0$ that
\begin{align*}
\delta\Big((s+\rho)^{-1}E(\rho)(s+\rho)^{-1}\Big)
 &= \delta((s+\rho)^{-1})E(\rho)(s+\rho)^{-1}+ (s+\rho)^{-1}E(\rho)\delta((s+\rho)^{-1})
 \end{align*}
For a commutator,
 \[ [a^{-1},\xi] \lel - a^{-1}[a,\xi]a^{-1} \pl .\]
It is proved in \cite{JRS} that a deviation in Theorem \cite{JRS} is a ultra-limit of commutator and hence
 \[ \delta((s+\rho)^{-1}) \lel -(s+\rho)^{-1}\delta(\rho)(s+\rho)^{-1} \pl .\]
Then the tracial property implies that \begin{align*}
 &-\tau\Big(\delta\big((s+\rho_t)^{-1}E(\rho)(s+\rho_t)^{-1}\big)
 \delta(\rho_t)\Big)
 \\
 &= -\tau\Big(\delta\big((s+\rho_t)^{-1}\big)E(\rho)(s+\rho_t)^{-1}
 \delta(\rho_t)\Big) -\tau\Big((s+\rho_t)^{-1}E(\rho)\delta\big((s+\rho_t)^{-1}\big)
 \delta(\rho_t)\Big)
 \\
  & =
  \tau\Big((s+\rho_t)^{-1}\delta(\rho_t)(s+\rho_t)^{-1}
 E(\rho)(s+\rho_t)^{-1}\delta(\rho_t)\Big) \\
 &\quad + \tau\Big((s+\rho_t)^{-1}
 E(\rho) (s+\rho_t)^{-1}\delta(\rho_t)(s+\rho_t)^{-1}\delta(\rho_t)\Big)\\
   &=
  2 \tau\Big(E(\rho)(s+\rho_t)^{-1}\delta(\rho_t)(s+\rho_t)^{-1}\delta(\rho_t)(s+\rho_t)^{-1}\Big) \gl 0  \pl .
  \end{align*}
We deduce that $h'(0)\gl 0$ and hence
  \begin{align*}
   \frac{d}{dt}  D(E(\rho_t)||\rho_t)
   &= \frac{d}{dt}   \Big( \tau(E(\rho_t)\ln E(\rho_t))-\tau(E(\rho_t)\ln\rho_t)\Big)
   \lel  -h'(t) \kl 0 \pl
   \end{align*}
because $E(\rho_t)=E(\rho)$ is a constant function.  Therefore,
 \[ \frac{d}{dt}\mathcal{I}_N(\rho_t) \kl -\mathcal{I}_A(\rho_t) \kl -\la \mathcal{I}_N(\rho_t) \pl .\]
We conclude that $f$ satisfies $f'(t)\le -\la f(t)$ and hence $f(t)\le e^{-\la t}f(0)$.
\qd
Another application of deviation calculus is that $\Gamma \E$ gives exponential decay of $L_p$-distance for all $1\le p\le \infty$. Let us start with a lemma.

\begin{lemma}\label{pnm} Let $\la \Gamma_A\le \Gamma_B$ and $N\subset M$ be the fixed-point subalgebra of both semigroups $e^{-tA}$ and $e^{-tB}$. Let $1< p<\infty$.
Then for $x\in M$ self-adjoint, the functions
 \[ f_A(t) \lel \|e^{-tA}(x)-E(x)\|_p^p \quad,\quad
  f_B(t)\lel \|e^{-tB}(x)- E(x)\|_p^p \]
satisfy $-\la f_A'(t)\le -f_B'(t)$ for all $t\ge 0$.
\end{lemma}

\begin{proof} Let $x\in M$ be self-adjoint.  Then $a=x-E(x)$ is again self-adjoint. We use the notation $a_+$ and $a_-$ for the positive and negative part of $a$. Recall that the spectral projections of $a_+$ and $a_-$ are mutually disjoint and commute with $a$. Thus $|a|^p=a_+^p+a_-^p$. Note that
\[f_A=\|e^{-tA}(x)-E(x)\|_p^p=\|e^{-tA}(x-E(x))\|_p^p=\|e^{-tA}(a)\|_p^p\]
Differentiating $f_A$ at $t=0$, we obtain that
\begin{align*}
 f_A'(t) &= -p\tau(|a|^{p-1}A(|a|)) \\
 &= -p\Big(\tau(a_+^{p-1}A(a^+))+
 \tau(a_+^{p-1}A(a_-))+\tau(a_-^{p-1}A(a_-))
 +\tau(a_-^{p-1}A(a_+)\Big) \pl .
 \end{align*}
Let $\delta_A$ be the derivation of $A$. Write $a_-=(\sqrt{a_-})^2=b^2$. Then  \eqref{derv} implies that
 \begin{align*}
 \tau(\delta(a_+^{p-1})\delta(b^2))
 &= \int_{\rz_+\times \rz_+}\int_{\rz_+\times \rz_+}
 \frac{s^{p-1}-t^{p-1}}{s-t} \frac{r^2-v^2}{r-v} \tau(dE_s\delta(a_+)dE_tdF_r\delta(b)dF_v)
 \end{align*}
where $E_s$ (resp. $F_r$) are spectral projections of $a_+$ (resp. $\sqrt{a_-}$). Because $E_s$ and $F_r$ are disjoint, and hence we obtain $\tau(a_+^{p-1}A(a^-))=0$. The same argument applies to $\tau(a_-^{p-1}A(a_+))$.
By Lemma \ref{mm},
$\la E(\delta_A(a_+^{p-1})\delta_A(a_+))
 \kl E(\delta_B(a_+^{p-1})\delta_B(a_+))$, and \[ \la E(\delta_A(a_-^{p-1})\delta_A(a_-))
 \kl E(\delta_B(a_-^{p-1})\delta_B(a_-)) \]
This implies
 \begin{align*}
 -\la f_A'(0)
 &=  \la p\Big(\tau(\delta_A(a_+^{p-1})\delta_A(a_+))+
 \tau(\delta_A(a_-^{p-1})\delta_A(a_-))\Big) \\
 &\le  p\Big(\tau(\delta_B(a_+^{p-1})\delta_B(a_+))+
 \tau(\delta_B(a_-^{p-1})\delta_B(a_-)\Big)
 \lel - f_B'(0) \pl .
 \end{align*}
Replacing $x$ by $e^{-tA}(x)$, we obtain that $-\la f_A'(t)\le f_B'(t)$ for all $t\ge 0$.
\qd


\begin{theorem} Let $T_t=e^{-tA}$ be a generator satisfying $\la$-$\Gamma\E$. Then for all $1\le p\le \infty$
 \[ \|T_t(x)-E(x)\|_p \kl e^{-\la t}\|x-E(x)\|_p \pl .\]
\end{theorem}

\begin{proof} Let us first assume that $x=T_s(y)$ for some $y$ so that $x$ belongs to the domain of $A$ and $\sqrt{A}$. Then we note that
   \[ T_t(I-E)(x)) \lel T_t(x)-E(x) \pl. \]
Write $a=x-E(x)$ and $S_t=e^{-t(I-E)}$. According to Lemma \ref{pnm} we have that
 \[ - \frac{d}{dt}\|T_t(a)\|_p^p
 \kl - \la \frac{d}{dt}\|S_t(a)\|_p^p \pl \pl.\]  However $E_{N}(a)=0$ and hence $S_t(a)=e^{-t}a$. Then
 \[ - \frac{d}{dt}\|S_t(a)\|_p^p \lel pe^{-tp}\|a\|_p^p \pl .\]
We apply Lemma \ref{pnm} and deduce that
\[ \la p f_A(0) \lel
 \la p \|a\|_p^p \lel
-\la f_{I-E}'(0) \kl -\la f_A'(0) \pl .\]
Repeat the argument for $a_t=T_t(a)$ and deduce
\[ \la p f_A(t) \kl - f_A'(t) \pl .\]
This implies
$f_A(t)\kl e^{-\la p t}f_A(0)$. Taking
the $p$-root we obtain that
 \[ \|T_t(a)\|_p \kl e^{-\la t}\|a\|_p \pl .\]
For general self-adjoint $x$, the assertion follows from the approximation $x=\lim_s T_s(x)$. By considering the $2\times 2$ matrix $\kla \begin{array}{cc} 0&x\\x^*&0\end{array}\mer $, we deduce the assertion for all $x$. The cases $p=1$ and $p=\infty$ are obtained by passing to the limit for $p\to 1$ or $p\to \infty$.  \qd
The next corollary studies the $\Gamma \E$ condition under tensor product.
\begin{cor} Let $T_t^j:M_j\to M_j$ be a family of semigroups with fixed-point subalgebras $N_j\subset M_j$. Then the tensor product semigroup $T_t=T_t^1\ten  T_t^2\ten \cdots \ten T_t^n$ has the fixed-point algebra $N=\ten_{j=1}^n N_j$.
Suppose for each $j$, $T_j^t$ satisfies $\la_j$-$\Gamma$E.  Then
 \[ \|T_t-E_{N}(x)\|_p \kl 2(\sum_{j=1}^n e^{-t\la_j}) \|x\|_p \pl .\]
 \end{cor}

\begin{proof} Let us consider a two-folds tensor product $S_t=T_t^1\ten T_t^2$. Then
 \begin{align*}
  S_t-E_{N_1}\ten E_{N_2} &= S_t(id-E_{N_1}\ten E_{N_2})
  \lel S_t\Big((id-E_{N_1})\ten id_{M_2}+E_1\ten (id_{M_2}-E_{N_2})\Big) \\
  &= (T_{t}^1-E_{N_1})\ten T_t^2+ E_{N_1}\ten (T_{t}^2-E_{N_2}) \pl .
  \end{align*}
Since $T_t^2$ and $E_{N_1}$ are completely contractive on $L_p$ spaces, we deduce that
 \begin{align*}
  \|S_t-E_{N_1}\ten E_{N_2}(x)\|_p
 & \kl \|(T_t^1-E_{N_1})\ten T_t^2(x)\|_p+ \|id\ten (T_t^2-E_{N_2})(x)\|_p \\
 &\le
    2 e^{-\la_1t} \|x\|_p+ \|id\ten (T_t^2-E_{N_2})(x)\|_p
     \\
 &\le 2(e^{-\la_1t}+e^{-\la_2t})\|x\|_p \pl .
 \end{align*}
For the $n$-fold tensor product, we may use induction.  \qd
\section{Kernel estimates and module maps}

\subsection{Kernels on noncommutative spaces} In this part we derive kernel estimates for ergodic and non-ergodic semigroups and their matrix valued-extension. Let $N,M$ be two von Neumann algebra and $N_*$ be the predual of $N$. The kernel of maps between noncommutative measure spaces are given by the following theorem due to Effros and Ruan  \cite{ER}:
 \begin{equation}\label{ER}
  CB(N_*,M) \lel N\bar{\ten}M
  \end{equation}
which states that space of completely bounded maps $CB(N_*,M)$ is completely isometric isomorphic to the von Neumann algebra tensor product $N\bar{\ten}M$.
$CB(N_*,M)$  is valid for two von Neumann algebras $N,M$. Let us now assume that $N$, is semifinite with trace $tr$.  The linear duality bracket $\langle x,y\rangle \lel tr(xy)$
gives a completely isometric pairing between $L_1(N,\tau)$ and $N^{op}$. More precisely, for every kernel $K\in N^{op}\bar{\ten}M$ the linear map
 \[ T_K(x) \lel (tr\ten id)(K(x\ten 1)) \]
satisfies $\|K\|_{\min}=\|T_K:L_1(N,\tau)\to M\|_{cb}$. Let us pause for a moment and consider $M=\mm_d$ and $N=\mm_k$. For a linear map $T:S_1^k\to \mm_d$ the Choi matrix is given by
 \[ \chi_T \lel \sum_{rs} |r\ran \lan s| \ten T(|r\ran \lan s|) \in \mm_k\ten \mm_d \pl .\]
The map $\phi(a)=a^{t}$ given by the transposed is $^*$-homomorphism between $\mm_k$ and $\mm_k^{op}$. Therefore we should consider
 \[ K_T \lel \sum_{rs} |s\ran\lan r| \ten T(|r\ran\lan s|)\in  \mm_k^{op}\ten \mm_m \]
and find
 \[ T_{K_T}(|r\ran\lan s|)
 \lel tr(K_T(|r\ran\lan s|\ten 1))
 \lel \sum_{tv} tr(|t\ran\lan v||r\ran \lan s|)\pl  T(|v\ran\lan t|) \lel T(|r\ran\lan s|) \pl .\]
Equivalently this shows that our description via kernels in $N^{op}\bar{\ten}M$ is compatible with the standard choice of a Choi matrix, see also \cite{Pau}.

\subsection{Saloff-Coste's estimates}
After these preliminary observation, we now assume that $T_t$ is a semigroup of (sub-)unital completely  positive, selfadjoint  maps on a finite von Neumann algebra $M$ so that $T_t:L_1(M)\to M$ are completely bounded.  According to the previous section, the kernels of $T_t$ are given by positive element $K_t\in M^{op}\bar{\ten}M$. Let $N\subset M$ be the fixed-point subalgebra of the semigroup $T_t$ and $E:M\to N$ be the conditional expectation. Recall that $T_t\circ E\lel E\circ T_t \lel  E $ by the self-adjointness on $L_2(M)$.

\begin{lemma}\label{cbc} Let $(T_t):\M\to \M$ be a semigroup of self-adjoint $^*$-preserving maps. Then for any $t\ge 0$,
\begin{enumerate}
\item[i)] $\|T_{2t}:L_1(M)\to L_{\infty}(M)\|=\|T_t:L_2(M)\to L_{\infty}(M)\|^2$.
\item[ii)]$\|T_{2t}-E_\N:L_1(M)\to L_{\infty}(M)\|=\|(T_t-E):L_2(M)\to L_{\infty}(M)\|^2$.
\end{enumerate}
The same estimates also hold for $cb$-norm, instead of the operator norm.
\end{lemma}
\begin{proof} We start by recalling a general fact. Let $v:X\to H$ be a linear map from a Banach space $X$ to a Hilbert space $H$. By $H^*\cong \bar{H}$, we have for any $x,y \in X$,
\[ \lan v(y),v(x)\ran_H =\lan y, \bar{v}^*v(x)\ran_{(X,\bar{X}^*)} \pl.\]
Here $\bar{v}^*:\bar{H}\to \bar{X}^*$ is conjugate adjoint of $v$ and the right hand side is the sesquilinear bracket between $X$ and its conjugate dual $\bar{X}^*$. Then we have
\begin{equation}\label{vv}
  \|v:X\to H\|^2 \lel \sup_{\|x\|\le 1} \lan v(x),v(x)\ran
\lel \sup_{\|x\|,\|y\|\le 1} |\lan y, \bar{v}^*v(x)\ran|
 \lel \|\bar{v}^*v:X\to \bar{X}^*\| \pl .\end{equation}
In our situation, we use  $X=L_1(N)$ and $v=T_t:X\to L_2(N)$. Note that
 \[ \lan T_t(x),T_t(y)\ran  \lel \tau(T_t(x)^*T_t(y))
  \lel (x^*,T_{2t}(y))\pl. \]
and the anti-linear bracket \[ (x,y) \lel \tau(x^*y) \] gives a completely isometry between $M$ and $\overline{L_1(M)}^*$. Therefore,
 \[ \|T_{2t}:L_1(M)\to L_{\infty}(M)\|
 \lel \|T_t:L_1(M)\to L_2(M)\|^2 \pl .\]
Similarly, we deduce ii) from
 \[ (T_t-E)(T_t-E) \lel T_{2t}-ET_t-T_tE+E \lel T_{2t}-E \pl.\]
For the cb-norm estimate we take $X=S_2(L_1(N))$ from \cite{Po} (see also Section \ref{ner}). For the Schatten $2$-space $S_2$, the trace bracket $\lan x,y\ran =tr(x^*y)$ identifies $S_2^*$ with $\bar{S}_2$. Thus
 \[ \|id\ten T_t:S_2(L_1(M))\to S_2(L_2(M))\|^2 \lel \|id\ten T_{2t}:S_2(L_1(M))\to S_2(L_{\infty}(M))\| \]
follows again from the general principle. The same argument applies to $(T_t-E)^2=T_{2t}-E$. \qd

The following observation is essentially due to Saloff-Coste (\cite{sfc}).

\begin{prop}\label{sfc} Let $(T_t)$ be a semigroup of self-adjoint and $^*$-preserving maps such that
 \begin{enumerate}
 \item[i)] $\|T_t:L_1(M)\to L_{\infty}(M)\|_{cb}\kl c t^{-d/2}$ for $c>0$ and $0\le t\le 1$;
 \item[ii)] the self-adjoint positive generator $A$ satisfies
 \[ \|A^{-1}(I-E):L_2(M)\to L_2(M) \|_{cb}\kl \la^{-1} \pl .\]
 \end{enumerate}
Then
 \[ \|T_t-E:L_1\to L_{\infty}\|_{cb}
  \kl \begin{cases} 2c t^{-d/2} &0\le  t\le 1\\
  C(d,\la) e^{-\la t} & 1\le t< \infty \pl .
  \end{cases} \]
where $C(d,\la)$ is a constant depending only on $d$ and $\la$.
 \end{prop}
\begin{proof} First, note that $T_t(I-E) = T_t-E$ and $\|I-E: L_{\infty}(M)\to L_{\infty}(M)\|_{cb}\kl 2$. Then the estimate for $t\le 1$ follows from the assumption i). For $t\gl 1/2$, we use functional calculus that
 \[ \|T_{t-1/4}(I-E)\|_{cb}\kl e^{-\la(t-1/4)} \pl .\]
Thus we obtain
\begin{align*}
 \|(T_t-E):L_1\to L_2\|_{cb} &\kl
 \|T_{t-1/4}(I-E):L_2(M)\to L_2(M)\|_{cb} \|T_{1/4}:L_1(M)\to L_2(M)\|_{cb} \\
 & \kl e^{-\la(t-1/4)} \|T_{1/2}:L_1(M)\to L_\infty(M)\|^2_{cb} \kl 2^dc e^{\la/4} e^{-\la t} \pl .
 \end{align*}
Applying Lemma \ref{cbc} yields the assertion for $t'=2t\gl 1$.
\qd

\subsection{From ultracontractivity to gradient estimates}In this part we apply kernel estimate to functional calculus of the generator $A$ and the semigroup $T_t$. Recall that $(I-T_t)$ is a generator a of semigroup. For a positive function $F$, we can define a new generator \[ \Phi_F(A) \lel \int_0^{\infty} (I-T_t) F(t) \frac{dt}{t} \pl,\] provided the integral is well-defined. Then the gradient form of $\Phi_F(A)$ is given by
 \[ \Gamma_{\phi_F(A)}(x,y)  \lel \int_0^{\infty} \Gamma_{I-T_t}(x,y) F(t) \frac{dt}{t} \pl .\]
We also define the modified Laplace transform
 \[ \phi_F(\la) \lel \int_0^{\infty} (1-e^{-t\la}) F(t) \frac{dt}{t} \pl. \]
For positive $F$ we may use the integrability (I), quasi-monotonicity (QM), or the well-known $(\Delta_2)$ conditions:
\begin{enumerate}
\item[(I)] $C_F:=\int \min(1,t) F(t) \frac{dt}{t}<\infty$;
\item[(QM)] For some $0<\mu<1$, there exists $C_\mu>0$ such that $F(\mu t)\kl C_\mu F(t)$ for all $t>0$.
\item[($\Delta_2$)] There exists $0<\al<1$, $t_{\al}>0$, $c_\al>0$ such that $F(\mu t) \kl c_{\al} \mu^{-\al} F(t)$ for $t_\al\le \mu t\le t$.
\end{enumerate}
Since $1-e^{-\la t}\sim \min(1,\la t)\kl (1+\la) \min(1,t)$ we deduce that
 \[ \phi_F(\la) \kl C_F (1+\la)   \]
and hence
 \[ \Phi_F(A)\kl C_F(I+A) \pl.\]
Then $\Phi_F(A)$ is a closable operator well-defined on the domain of $A$, and hence according to our assumptions also defined on the dense subalgebra $\A$.

\begin{rem} {\rm Our calculus is closely related to the theory of symmetric positive definite functions on $\rz$, which can be represented as
 \[ \psi_G(\la) \lel \int_{\rz} (1-\cos(s\la)) G(s) \frac{ds}{s} \pl ,\]
where $G$ is a positive function such that \[\displaystyle \int_{\rz} \min(1,s^2)G(s)\frac{ds}{s}=4\int_0^{\infty} \min(1,t)G(\sqrt{t}) \frac{dt}{t}<\infty\pl.\] Let $g$ be a gaussian distribution. Then we obtain a randomized new positive definite function
 \[ \tilde{\psi}_G(\la) \lel \ez \psi(g\la) \lel \int_{\rz} (1-e^{-s^2\la^2}) G(s) \frac{ds}{s}
 \lel 2\int_0^\infty (1-e^{-t\la^2}) G(\sqrt{t}) \frac{dt}{t}   \pl .\]
Thus for any positive definite function $\psi_G$, $\phi(\la)=\tilde{\psi}(\sqrt{\la})$ gives the function in our generator calculus.
}\end{rem}
In particular all the subordinated generator are example of our calculus.
\begin{exam}\label{aal} {\rm Let $0<\al<1$ and $F_{\al}(t)=c(\al) t^{-\al}$. Then \[ \phi_{\al}(\la) \lel c(\al) \int_0^{\infty} (1-e^{-\la t}) t^{-\al} \frac{dt}{t} \lel  \la^{\al}  c(\al) \int_0^{\infty} (1-e^{-s}) s^{-\al} \frac{ds}{s} \lel \la^{\al}\]
holds for a suitable choice of the normalization $c(\al)$. Here $F_{\al}$ actually corresponds to the positive definite function $\psi(x)=x^{2\al}$. It is clear that $F_\al$ satisfies the condition (I) and (QM). We refer to \cite{TM} for monotonicity results overlapping with our approach.}
\end{exam}
Let us now fix an $F$ satisfy the condition (I) and (QM). A key technical tool for our estimates is the following family of unital completely positive map,
 \[ \Psi_F(r) \lel g(r)^{-1}\int_0^{\infty} e^{-r/t}T_t F(t) \frac{dt}{t} \pl, \pl r>0\pl,\]
where $g(r)$ is the normalization constant given by
 \[ g(r) \lel \int_0^{\infty} e^{-r/t} F(t) \frac{dt}{t}  \pl . \]
\begin{lemma}\label{limit} Let $T_t$ be a semigroup of unital completely positive, self-adjoint maps. Suppose the generator $A$ satisfy $\Gamma$-regularity and $F$ satisfies ${\rm (}I{\rm )}$. Then
\begin{enumerate}
\item[i)] For $r\gl s$, $g(r)\le g(s)$ and $g(r)\Psi_F(r)\le_{cp} g(s)\Psi_F(s)$.
\item[ii)] $\displaystyle \lim_{r\to 0} g(r)(I-\Psi_F(r))=\Phi_F(A)$.
\item[iii)] $\displaystyle \lim_{r\to 0} g(r)\Gamma_{I-\Psi_F(r)}\lel \Gamma_{\Phi_F(A)}$.
\end{enumerate}In particular,
 $\Phi_F(A)$ satisfies $\Gamma$-regularity.
\end{lemma}
\begin{proof} Let $r\gl s>0$. Then obviously $e^{-r/t}\le e^{-s/t}$ and hence
 \[g(r)\le g(s) \pl, \pl g(r)\Psi(r) \le_{cp}\pl  g(s) \Psi(s)\]
 Since $\Psi(r)$ is completely positive and self-adjoint, then $A_r:=g(r)(I-\Psi(r))$ is a generator of a semigroup of unital completely positive and self-adjoint maps.
 Let us consider the function
 \[ \psi_r(\la) \lel g(r)^{-1} \int e^{-r/t} (1-e^{-\la t}) F(t) \frac{dt}{t} \pl. \]
Using
 \[ (1-e^{-\la t})\kl \min(1,\la t) \kl (1+\la) \min(1,t) \]
we deduce from the Dominated Convergence Theorem that
 \[ \lim_{r\to 0} g(r)\psi_r(\la) \lel \Phi_F(\la)   \pl .\]
Note that $g(r)\phi_r\le g(s)\phi_s$ if $r\ge s$. Applying the monotone convergence theorem, we have that for any $x\in L_2(M)$
 \[ \lim_{r\to 0} \lan x,g(r)\psi_r(A)x\ran
 \lel \lim_{r\to 0} \int_0^{\infty} g(r)\psi(r)(\la) d\mu_x(\la) \lel \lan x,\Phi(A)x\ran \pl .\]
 Here $d\mu_x(\la)$ is the spectral measure $d \lan x, 1_{A\le \la} x\ran$. We split the function $\Phi_F$ into two pieces
 \[ \Phi_F(\la) \lel \Phi_F'(\la)+\Phi_{F}''(\la)
 \lel \int_0^1 \frac{1-e^{-\la t}}{t} F(t) dt
 + \int_1^\infty  (1-e^{-\la t}) F(t)\frac{dt}{t}  \pl .\]
According to our assumption $\displaystyle \int_0^1 F(t)dt, \int_1^\infty F(t)\frac{dt}{t}$ are finite. For any $x,y\in \A$, $\frac{1}{t}\Gamma_{1-T_t}(x,y)$ is uniformly bounded in $L_1(M)$. Therefore the gradient form of $\Phi_F'(A)$ is well-defined on $\A$ and has range in $L_1(M)$. On the other hand, the map $T_t:M\to M$ is complete positive and bounded from $L_{\infty}(M)$ to itself. Then $\Phi''_F(A)$ converges and the gradient forms $\Gamma_{\Phi''_F(A)}$ take range in $L_{\infty}(M)$ and in particular also in $L_1(M)$.\qd

A direct consequence of the above lemma is as follows:
\begin{cor}\label{ccc} For every $r>0$,
  \[g(r) \Gamma_{I-\Psi_F(r)} \kl \Gamma_{\Phi_F(A)} \pl ,\pl
  g(r)\mathcal{I}_{I-\Psi_F(r)}\kl \mathcal{I}_{\Phi_F(A)} \pl .\]
\end{cor}
\begin{proof}The previous lemma shows that $g(r)\Psi_F(r)$ is decreasing in $cp$-order as $r\to 0$. By Lemma \ref{grad1}, the gradient $g(r)\Gamma_{I-\Psi_F(r)}$ is also decreasing in $cp$-order. Then the assertion follows from the limit
\[\lim_{r\to 0} g(r)\Gamma_{I-\Psi_F(r)}=\Gamma_{\Phi_F(A)}\pl.\]
and Corollary \ref{Fisher}.
\end{proof}
We say a self-adjoint semigroup $T_t$ is \emph{ergodic} if its fixed-point subalgebra $\N=\mathbb{C}1$. In this situation, the conditional expectation is the trace $E_\tau(x)=\tau(x)1$ and the kernel is $K_E=1\ten 1$ in $M^{op}\overline{\ten} M$.
\begin{prop}\label{decay} Assume that a semigroup $(T_t)$ of unital completely positive and self-adjoint maps satisfy the conclusion of Proposition \ref{sfc} with respect to $E_\tau(x)=\tau(x)1$. Then there exists a $r_0>0$ such that for all $r\ge r_0$,
\begin{enumerate}
\item[i)] $\|\Psi_F(r)-E_{\tau}:L_1(M)\to L_{\infty}(M)\|_{cb}\kl 1/2$;
\item[ii)] $E_{\tau} \le_{cp} 2 \Psi_F(r)$;
\item[iii)] $g(r)\Gamma_{I-E_{\tau}}\le_{cp} 2 \Gamma_{\Phi_F(A)}$.
\end{enumerate}
\end{prop}
\begin{proof} Write $E_\tau=E$. For i) we use the triangle inequality
\begin{align*}
 \|\Psi_F(r)-E\|_{cb}
 &\le g(r)^{-1} \int_0^{\infty} e^{-r/t} \|T_t-E\|_{cb} F(t) \frac{dt}{t} \\
 &\le g(r)^{-1}(2c \int_0^r  e^{-r/t} t^{-d/2} F(t)\frac{dt}{t}+ C(d,\la)
 \int_r^{\infty} e^{-r/t} e^{-\la t} F(t) \frac{dt}{t})  \\
 &:=  g(r)^{-1} ( 2c \text{I}+ C(d,\la) \text{II}) \pl.
 \end{align*}
 Then the condition (QM) of $F$ implies that for some some $0<\mu<1$,
 \begin{align*}
 I &= r^{-d/2} \int_1^{\infty} F(r/u) u^{d/2} e^{-u} \frac{du}{u}\\ &\kl r^{-d/2}c(d,\mu)
 \int_1^{\infty} F(\mu r/\mu u)  e^{-\mu u} \frac{du}{u} \\
 &= r^{-d/2} c(d,\mu) C_\mu \int_{\mu}^{\infty}    F(r/w)    e^{-w} \frac{dw}{w} \kl
  r^{-d/2} c(d,\mu)  C_\mu  g(r) \pl .
 \end{align*}
Here we have used the change of variables $u=r/t$ and $w=\mu u$. Indeed, the same change of variable shows that
 \[ g(r) \lel \int_0^{\infty} F(r/u) e^{-u} \frac{du}{u} \pl .\]
For the second part we see that
 \begin{align*}
 II &= \int_0^1  e^{-\frac{2\la r}{u}}  F(r/u)  e^{-u} \frac{du}{u} \kl e^{-\la r} \int_0^1 F(r/u) e^{-u} \frac{du}{u} \kl
   e^{-\la r} g(r) \pl .
 \end{align*}
Hence there exists a $r_0$ which only depends to $c, d,C(d,\la)$ and $C_\mu$ so that for all $r \ge r_0$
 \[ \|\Psi_F(r)-E\|_{cb} \kl \frac{1}{2} \pl .\]
 Let $K_r$ be the kernel of $\Psi_F(r)$ in $\M^{op}\overline{\ten} \M$ and recall that $K_E=1\ten 1$ is the kernel of $E$. Then by \eqref{ER},
 \[ \|K_r-1\ten 1\|_{\M^{op}\overline{\ten} \M}\kl \frac{1}{2} \]
which implies $K_r\gl \frac{1}{2}1\ten 1= 2K_E$. The other assertions follow from Lemma \ref{grad1} and Corollary \ref{ccc}.
\qd

\begin{rem}{\rm The polynomial decay $t^{-d/2}$ is not really needed. Indeed, if
 \[ \|T_t:L_1\to L_{\infty}\|\kl C_{\al}  e^{c_{\al}t^{-\al}} \]
holds for some $\al<1$ and  ($\Delta_\mu$) for all $t>0$. We choose $\beta>0$ such that $C_{\al}C_F(\mu)e^{-\frac{1-\mu}{2}\beta}\le \frac14$ and choose $r$ large enough so that $c_{\al}r^{-\al}\le \frac{1-\mu}{2}\beta^{1-\al}$.
Then we find that
 \begin{align*}
  I_{\beta} &\le  C_F(\mu) C_{\al} \int_{\beta}^{\infty} e^{u^{\al}(c_{\al}r^{-\al}-(1-\mu)u^{1-\al})} F(r/\mu u) e^{-\mu u} \frac{du}{u} \\
  &\le C_F(\mu) C_{\al} e^{-\frac{1-\mu}{2}\beta}
 \int_\mu^{\infty} F(r/w) e^{-w} \frac{dw}{w} \pl .
 \end{align*}
Using a small modification in Proposition \ref{sfc}, the spectral gap allows us to estimate $g(r) II_{\beta}\kl C(\beta,\la)e^{-\frac{r\la}{\beta}}$.
}\end{rem}

Motivated by the discussion above, let us introduce
 \[ t_0 \lel \inf\{ t \pl |  \pl \|k_{T_t}-1\ten 1\|\kl 1/2\} \pl .\]
\begin{prop}\label{lowest} Assume that $F$ satisfies $(\Delta_2)$.  Let $r=\max\{t_0,t_{\al}\}$. Then
 \[ \Gamma_{\Phi_F(A)}\gl \frac{F(r)}{2\al c_{\al}}\Gamma_{I-E} \pl.  \]
\end{prop}

\begin{proof} Recall that $\Gamma_B$ is additive with respect to $B$. Therefore we deduce from $\|K_t-1\ten 1\|\kl 1/2$  and Lemma \ref{grad1}  that $\Gamma_{I-T_t} \gl \frac{1}{2} \Gamma_{I-E}$ holds for $t\gl t_0$. Then we note that the $(\Delta_2)$ condition implies
 \[ F(r)\lel F(\frac{r}{t}t)\kl c_{\al} (t/r)^{\al}  F(t)   \]
for all $r\le t$.  Therefore,
 \begin{align*}
 \Gamma_{\Phi_F(A)}
 &= \int_0^\infty \Gamma_{I-T_t} F(t) \frac{dt}{t}
  \ge \frac{\Gamma_{I-E}}{2}
 \int_{r}^{\infty} F(t) \frac{dt}{t}
   \ge
 \frac{\Gamma_{I-E}}{2c_{\al}} F(r) r^{\al} \int_{r}^{\infty} t^{-(1+\al)} dt \lel
   \frac{\Gamma_{I-E}}{2\al c_{\al}} F(r) \p ,
 \end{align*}
 which completes the proof.\qd

\begin{theorem}\label{erg} Let $T_t$ be an ergodic semigroup of completely positive self-adjoint maps such that
 \begin{enumerate}
 \item[i)] $\|T_t:L_1(M)\to L_{\infty}(M)\|_{cb}\le c t^{-d/2}$ for $0\le t\le 1$ and  $c,d>0$.
 \item[ii)] the generator $A$ has a spectral gap $\si_{\min}$.
 \end{enumerate}
Let $F$ be a function satisfying the condition (I)+(QM) or (I)+($\Delta_2$). Then $\Phi_F(A)$ satisfies $\la$-$\Gamma \E$ and hence $\la$-CLSI for some $\la$ depending on $c,d, F$ and $\si_{\min}$.
\end{theorem}
\begin{proof} Let $E=E_{\tau}$ be the conditional expectation on the trace.
According to Proposition \ref{decay}, we deduce that
 \[ \Gamma_{I-E_{\tau}} \kl 2g(r_0) \Gamma_{\Phi_F(A)}  \]
for some $r_0$ depending on $\si_{\min}$, $c$ and $d$. Then one can choose $\la=\frac{1}{2g(r_0)}$.
\qd

\subsection{Non-ergodic semigroups}\label{ner}
In this part, we want to adapt the kernel techniques for ergodic maps to the non-ergodic situation. This requires more operator space theory from the work in \cite{JP} on vector-valued noncommutative $L_p$-spaces associated with inclusion of von Neumann algebras. As usual we assume that $(T_t):M\to M$ is a semigroup of unital completely positive selfadjoint maps and $N\subset M$ is the fixed-point subalgebra. Whenever $N$ is infinite dimensional, we can longer hope for an ultracontractivity of $T_t:L_1(M)\to L_{\infty}(M)$ for operator norm or $cb$-norm, because the identity map \[ id:L_1(N)\to L_{\infty}(N) \]
is already unbounded. This leads us to use vector-valued $L_p$ norms.
Let $1\le p,q,r\le \infty$ and fix the relation $\frac{2}{r}=|\frac{1}{p}-\frac{1}{q}|$.
Recall that the $L_p(L_q)$ norms for the inclusion $N\subset\M$ is defined as,
 \[ \|x\|_{L_p^q(N\subset M)}
 \lel \begin{cases} \displaystyle\inf_{x=ayb} \|a\|_{L_{2r}(N)}\|y\|_{L_{q}(M)}\|b\|_{L_{2r}(N)} & p\le q  \pl ,\\
 \displaystyle\sup_{\|a\|_{L_{2r}(N)}=\|b\|_{L_{2r}(N)}\le 1}  \|axb\|_{L_{q}(M)} &q\le p \pl.
 \end{cases}  \]
Here for $p\le q$, the infimum takes over all factorization $x=ayb$ with $a,b\in L_{2r}(N), y\in L_q(M)$ and for $p\ge q$, the supremum runs over all $a,b\in L_{2r}(N)$ with $\|a\|_{L_{2r}(N)}=\|b\|_{L_{2r}(N)}\le 1$. The Banach space $L_p^q(N\subset M)$ is then the completion of $M$ with respect to the corresponding norm. It follows from the H\o lder inequality that for $p=q$, $L_p^p(N\subset M)\cong L_p(M)$. These norms have been extensively studied in the quantum information theory and operator space community \cite{GJLR2,BaRo,muller}. For the spacial cases $M=R\overline{\ten}N$ and $M=\mm_k(N) $,
 \[ L_p^q(N\subset R\overline{\ten}N) \lel L_p(R,L_q(N))\pl, L_p^q(N\subset \mm_k(N))=S_p^k(L_q(N)) \pl .\]
which are the vector-valued $L_p$ spaces introduced in \cite{pisier93}. In the following we mention the properties of $L_p^q(N\ssubset M)$ needed in our discussion and refer to \cite{JP} for a detail account of these $L_p$-spaces.  First, we will use a duality relation that the anti-linear trace bracket $(x,y)=\tau(x^*y)$ provides an isometric embedding
\begin{equation}\label{dual}
  L_p^q(N\ssubset M) \subset L_{p'}^{q'}(N\ssubset M)^* \pl,
  \end{equation}
for $1\le p,q\le \infty$, $\displaystyle \frac{1}{p}+\frac{1}{p'}=\frac{1}{q}+\frac{1}{q'}=1$ and it is indeed an equality when for $1<p,q<\infty$.  We will also need the following factorization property that
\begin{equation}\label{pp}
 L_p(M) \lel  L_{2p}(N) L_\infty^{p}(N\ssubset M) L_{2p}(N) \pl
 \end{equation}
 which reads as
 \[\norm{x}{p}=\inf_{x=ayb}\norm{a}{L_{2p}(N)}\norm{y}{L_\infty^{p}(N\subset M)}\norm{b}{L_{2p}(N)}\pl.\]
This can be verified by interpolation. Indeed, it is obvious for $p=\infty$. For $p=1$ let us assume that $x$ is positive and $\tau(x)=1$. Then $\tau(E(x))=1$, and we may write
 \[ x \lel E(x)^{1/2} (E(x)^{-1/2}xE(x)^{-1/2}) E(x)^{1/2} = E(x)^{1/2}yE(x)^{1/2} \pl .\]
Note that for every $a\in L_2(N)$,
 \begin{align*}
 \tau(a^*ya) &= \tau( E(x)^{-1/2}E(x)E(x)^{-1/2}aa^*) \lel \tau(aa^*) \lel \|a\|_2^2 \pl .
 \end{align*}
Using the positivity of $y$ and a Cauchy-Schwarz type argument, it follows that $\|y\|_{L_\infty^1(N\subset M)}\le 1$. This factorization property is closely related  to the following fact.

\begin{lemma}[Lemma 4.9 of \cite{JP}]\label{JP} Let  $x\in M$. Then \[ \|x\|_{L_{\infty}^1(N\subset M)}
 \lel \inf_{x=x_1x_2} \|E(x_1x^*_1)\|_{\infty}^{1/2}\|E(x_2^*x_2)\|_{\infty}^{1/2} \pl .\]
\end{lemma}

The following fact is an extension of Lemma 1.7 of \cite{pisier93}

\begin{lemma}\label{dfnorms} Let $T:M\to M$ be a $N$-bimodule map and $1\le p,q\le \infty$. Then for any $1\le s\le \infty$,
 \[ \|T:L_{\infty}^p(N\ssubset M)\to
  L_{\infty}^q(N\ssubset M)\| \lel \|T:L_s^p(N\ssubset M)\to L_s^q(N\ssubset M)\|
   \]
\end{lemma}

\begin{proof} Let us introduce the short notation
$|\mkern-1.5mu\|T\|\mkern-1.5mu|_s=\|T:L_s^p(N\ssubset M)\to L_s^q(N\ssubset M)\|$. We first prove ``$\ge$'' for $s=1$. For an element $x\in L_1^p(N\ssubset M)$ of norm less than $1$, we have a decomposition $x=ayb$ with $a,b\in L_{2p'}(N)$, and $y\in L_p(M)$ all of norm less than $1$. Using the factorization \eqref{pp}, we may further write $y=\al Y \beta$ with $\al,\beta\in L_{2p}(N)$ and $Y\in L_{\infty}^p(M)$ all of norm less than 1. Therefore we have shown that $x=(a\al)Y(\beta b)$ with $Y\in L_{\infty}^p(N\subset M)$, and
 \[  \|a\al\|_{L_2(N)}\le 1 \pl ,\pl \|\beta b\|_{L_2(N)}\le 1 \pl .\]
Thus we may write $a\al=a'\al'$ and $\beta b=\beta'b'$ such that
  \[ \max\{\|b\|_{2q'},\|\beta\|_{2q},\|a\|_{2q'}\|\beta\|_{2q}\}\le 1  \pl .\]
Then we deduce from the module property that
 \[T(x) \lel a'\al' T(Y)\beta' b'
 \lel a' (\al' T(Y) \beta') b'
  \]
and $\|T(Y)\|_{L_{\infty}^q(N\subset M)}\le |\mkern-1.5mu\|T\|\mkern-1.5mu|_{\infty}$. Using the other part of \eqref{pp}, $(\al' T(Y) \beta')\in L_q(M)$ of norm less than 1 and hence we have shown that $\norm{T(x)}{L_{1}^q(N\subset M)}\le 1$. By interpolation, we deduce that
 \[ |\mkern-1.5mu\|T\|\mkern-1.5mu|_s \kl |\mkern-1.5mu\|T\|\mkern-1.5mu|_{\infty} \]
for all $1\le s\le \infty$. We dualize this inequality by applying it to $T^*$ and obtain
 \[ |\mkern-1.5mu\|T\|\mkern-1.5mu|_s \lel \|T^*:L_{s'}^{q'}(N\ssubset M)\to L_{s'}^{p'}(N\ssubset M)\|\kl \|T^*: L_{\infty}^{q'}(N\ssubset M)\to L_{\infty}^{p'}(N\ssubset M)\| \lel |\mkern-1.5mu\|T\|\mkern-1.5mu|_1 \pl .\]
Then we dualize again to get
 \[ |\mkern-1.5mu\|T\|\mkern-1.5mu|_{\infty} \kl |\mkern-1.5mu\|T^*\|\mkern-1.5mu|_{s'}
\kl |\mkern-1.5mu\|T\|\mkern-1.5mu|_{s}\kl |\mkern-1.5mu\|T\|\mkern-1.5mu|_{\infty}
\pl  .\]
Hence all these norms coincide. \qd

Thanks to the independence of $s$, we may now introduce the short notation
 \[ \|T\|_{p\to q}=\|T:L_\infty^p(N\ssubset M)\to L_{\infty}^q(N\ssubset M)\| \pl . \]
and similarly, the $cb$-version
 \[ \|T\|_{p\to q,cb}=\sup_m \|id_{\mm_m}\ten T:L_\infty^p(\mm_m(N)\ssubset \mm_m(M))\to L_{\infty}^q(\mm_m(N)\ssubset \mm_m(M)\| \pl . \]
In particular, we understand $L_{\infty}^p(N\ssubset M)$ as an operator space with operator space structure
 \[ \mm_m(L_\infty^p(N\ssubset M)) \lel
 L_\infty^p(\mm_m(N)\subset \mm_m(M)) \pl .\]
The analogue of Lemma \ref{cbc} reads as follows:

\begin{lemma}\label{cbc2} Let $(T_t)$ be a semigroup of self-adjoint $^*$-preserving $N$-bimodule maps. Then
\begin{enumerate}
\item[i)] $\|T_{2t}\|_{1\to \infty}=\|T_t\|_{1\to 2}^2$.
\item[ii)]$\|T_{2t}-E\|_{1\to \infty}=\|(T_t-E)\|_{2\to \infty}^2$
\end{enumerate}
The same equality holds for $cb$-norms.
\end{lemma}
\begin{proof}Because $T_t$ are $N$-bimodule maps, we know by Lemma \ref{dfnorms} that
\begin{align*} &\|T_{2t} \|_{1\to \infty}=\|T_{2t}:L_2^1(N\subset \M ) \to L_2^\infty(N\subset M) \|\pl,\\ &\|T_t\|_{1\to 2}=\|T_t:L_2^1(N\subset M )\to L_2(M )\|
\end{align*}
Take $X=L_2^1(N\subset M)$ and $H=L_2(M)$. The anti-linear bracket $\lan x, y\ran=\tau(x^*y)$ gives a complete isometric embedding $L_2^{\infty}(N\subset M)\subset \bar{X}^*$. Then using the general principle in the proof Lemma \ref{cbc} implies the assertion because $T_t$ is $^*$-preserving and self-adjoint. The $cb$-norm case follows similarly with $X=S_2(L_2^1(N\subset M))$ and $H=S_2\ten_2L_2(M)$.
\end{proof}

We have seen in the last subsection that a complete positive order inequality $E_{\tau}\le_{cp} T$ can be deduced from kernel estimates. For non-ergodic cases, we have to modify the argument by introducing the appropriate Choi matrix for bimodule maps. Let us recall that the conditional expectation $E:M\to N$ generates a Hilbert $W^*$-module $\mathcal{H}_E=L_{\infty}^c(N\subset M)$ with $N$-valued inner product
 \[ \langle x,y\rangle_{\mathcal{H}_E} \lel E(x^*y) \pl .\]
As observed in \cite{JD,JS} it is easy to identify the completion of this module in $\mathbb{B}(L_2(M)$, namely the strong closure of $\bar{\mathcal{H}}_E=\overline{Mp_E}$, where $p_E=E:L_2(M)\to L_2(M)$ is the Hilbert space projection onto the subspace $L_2(N)\subset L_2(M)$. The advantage of a complete $W^*$-module is the existence of a module basis $(\xi_i)_{i\in I}$ such that
 \[ \langle \xi_i,\xi_j\rangle \lel \delta_{ij} p_i \]
where $p_i\in N$ are projections. Note that in our situation the inclusion $Mp_E\subset L_2(M)$ is faithful and hence, the basis elements $\xi_i$ (or more precisely $\hat{\xi}_i$ obtained from the GNS construction) are in $L_2(M)$. In particular, every element $x$ in $L_2(M)$ has a unique decomposition
\[ x \lel \sum_{i} \xi_ix_i \]
so that $x_i=p_ix_i\in N$. Indeed, we have $x_i=\langle \xi_i,x\rangle_E$. For a $N$-bimodule map $T:L_\infty^1(N\ssubset M)\to M$ we may therefore introduce the Choi matrix
 \[ \chi_T \lel \sum_{i,j} |i\ran\lan j| \ten T(\xi_i^*\xi_j) \pl . \]

\begin{lemma}\label{Choi} Let $T:M\to M$ be a $N$-bimodule map. Then
 \[ \|T\|_{1\to \infty,cb}\lel \|\chi_T\|_{\mathbb{B}(\ell_2(I))\bar{\ten}M} \pl .\]
\end{lemma}

\begin{proof} Let $q=\sum_{i,j}|i\ran\lan j|\ten \xi_i^*\xi_j\in \mathbb{B}(\ell_2(I))\bar{\ten}M$. Viewing $q$ as a kernel, the corresponding map $T_q: S_1(\ell_2(I))\to M$ is given by
\[T_q(|i\ran\lan j|)=\xi_i^*\xi_j\]
Let us show that $T_q:S_1(\ell_2(I))\to L_{\infty}^1(N\ssubset M)$ is completely contractive. Indeed, using operator space version of \eqref{ER}
\begin{align*} \norm{T_q}{cb}&=\norm{q}{\mathbb{B}(\ell_2(I))\ten_{\min} L_{\infty}^1(N\subset M)} \lel
 \norm{q}{ L_{\infty}^1(\mathbb{B}(\ell_2(I))\ten_{\min} N\subset \mathbb{B}(\ell_2(I))\ten_{\min} M)}\\
&=\norm{id\ten E(q)}{ \mathbb{B}(\ell_2(I))\bar{\ten} M} \lel
 \norm{\sum_{i}\ket{i}\bra{i}\ten p_i}{ \mathbb{B}(\ell_2(I))\bar{\ten} M}\le 1
\end{align*}
Here we have used the fact $q$ is positive and $p_i$ are projections. Note that the kernel of $T\circ T_q: S_1(\ell_2(I))\to M$ is exactly the Choi matrix $\chi_T$. Therefore, thanks to \eqref{ER} again, we deduce that \[ \norm{\chi_T}{}\lel \|Tq:S_1(\ell_2(I))\to M\|_{cb} \kl \|T\|_{1\to \infty,cb} \pl .\]
Now let $x\in L_{\infty}^1(N\ssubset M)$ of norm less than $1$. According Lemma \ref{JP}, we have a factorization $x=y_1y_2$ such that $E(y_1^*y_1)\le 1$ and $E(y_2^*y_2)\le 1$. This means we find coefficients $a_i,b_j$ such that
 \[ x \lel \sum_{i,j} a_i^*\xi_i^*\xi_j b_j \]
and $\sum_i a_ia_i^*\le 1$ and $\sum_j b_j^*b_j\le 1$. Therefore, we deduce that
  \begin{align*}
  \|T(x)\|_M \!&= \!\|\sum_{i,j} a_i^*T(\xi_i^*\xi_j)b_j\| = \|(\sum_{i}\bra{i}\ten a_i^*) \Big(\sum_{i,j}\ket{i}\bra{j}\ten T(\xi_i^*\xi_j)\Big(\sum_{j}\ket{j}\ten b_j)\\
  &\le
    \|\sum_i a_ia_i^*\|^{1/2} \|\chi_T\| \|\sum_j b_j^*b_j\|^{1/2}\! .
  \end{align*}
This implies
 \[ \|T(x)\|
 \kl \|\chi_T\| \inf_{x=y_1y_2} \|E(y_1y_1^*)\|^{1/2} \|E(y_2^*y_2)\|^{1/2} \pl ,\]
or equivalently
 \[ \|T\|_{1\to \infty} \kl \|\chi_T\| \pl .\]
 The same argument applies for $id_{\mm_m}\ten T$ and we have the equality  $\|T\|_{1\to \infty, cb}\lel \|\chi_T\|$. \qd

We are now in a position to prove a version of Proposition \ref{decay} ii) in the non-ergodic situation
\begin{lemma} Let $T:M\to M$ be a unital completely positive $N$-bimodule map such that
  \[ \|T-E_N:L_\infty^1(N\subset M)\to M\|_{cb} \kl \frac12 \pl .\]
Then $E_N\le_{cp} 2 T$.
\end{lemma}

\begin{proof} Let $\chi_T$ (resp. $\chi_E$) be the Choi matrix of $T$ (resp. $E_N$).
We known by Lemma \ref{Choi} that \[\|\chi_T-\chi_E\|_{B(\ell_2(I))\bar{\ten}N}\kl \frac12\pl.\] Since $T$ and $E$ are completely positive, $\chi_T$ and $\chi_E$ are positive. Thus we may write $\chi_E-\chi_T=\al-\beta$ with $\displaystyle 0\le \al,\beta \le \frac12 \pl$. Write $\al=\sum_{i,j}\bra{i}\ket{j}\ten \al_{i,j}$ and $\beta=\sum_{i,j}\bra{i}\ket{j}\ten \beta_{i,j}$. It is clear that $\al_{i,j}-\beta_{i,j}=E(\xi_i^*\xi_j)-T(\xi_i^*\xi_j)$. Let $x=y^*y$ be a positive element in $M$ and $y=\sum_j \xi_j y_j$ with coefficients $y_j\in N$ in the module basis. Then we deduce that
 \begin{align}\label{eqa}
 \sum_j y_j^*p_jy_j &=   E(y^*y) \lel (E-T)(y^*y)+T(y^*y) \nonumber\\
 &= \sum_{i,j} y_i^*p_i\al_{i,j}p_jy_j-
 \sum_{i,j} y_i^*p_i\beta_{i,j}p_jy_j + T(y^*y) \nonumber\\
 &\le \frac{1}{2} (\sum_j y_j^*p_jy_j) + T(y^*y) \pl .
 \end{align}
Indeed in the last step we use the fact that
\begin{align*}\sum_{i,j} y_i^*p_i\al_{i,j}p_jy_j&=\Big(\sum_{i}\bra{i}\ten y_i^*p_i \Big)\Big(\sum_{i,j}\ket{i}\bra{j}\ten \al_{i,j}\Big)\Big(\sum_{j}\ket{j}\ten p_jy_j\Big)\\&\le \frac{1}{2}\Big(\sum_{i}\bra{i}\ten y_i^*p_i \Big)\Big(\sum_{j}\ket{j}\ten p_jy_j\Big)\\&=\frac{1}{2} (\sum_j y_j^*p_jy_j)\pl.\end{align*}
Subtracting $\frac{1}{2}(\sum_j y_j^*p_jy_j)$ in \eqref{eqa}, we obtain
 \[ E(y^*y)=\sum_j y_j^*p_jy_j \kl 2 T(y^*y) \pl .\]
The same argument holds for matrix coefficients. Hence $E\le_{cp} 2 T$.\qd

Thus in the non-ergodic situation, we can now state the analogue of Theorem \ref{erg}.

\begin{theorem}\label{nerg} Let $T_t$ be a semigroup of completely positive self-adjoint  maps and $N$ be  contained in the  fixed-point subalgebra. Suppose that
 \begin{enumerate}
 \item[i)] $\|T_t:L_\infty^1(N\ssubset M)\to L_{\infty}(M)\|_{cb}\kl c t^{-d/2}$ for $0\le t\le 1$ and $c, d>0$.;
 \item[ii)] $\|T_t(I-E_N):L_2(M)\to L_2(M)\|\kl e^{-\la_{\min} t}$ for some $\la_{\min}>0$.
 \end{enumerate}
Let $F$ be a function satisfying (I)+(QM) or (I)+($\Delta_2$). Then
$\Phi_F(A)$ satisfies $\la$-$\Gamma \E$ and hence $\la$-CSLI for some $\la$ depending on $c,d,F$ and $\la_{\min}$.
\end{theorem}

\section{Riemannian manifolds and Representation theory}
In this section we find  heat kernel estimates, which allow us to apply Theorem \ref{nerg}.

\subsection{Riemannian manifolds}
Let $(\M,g)$ be a $d$-dimensional compact Riemannian manifold without boundary. A \emph{H\"{o}rmander system} is a finite family of vector fields  $X=\{X_1,...,X_r\}$ such that for some global constant $l_X$, the set of iterated commutators (no commutator if $k=1$)
 \[ \bigcup_{1\le k\le l_X} \{ [X_{j_1},[X_{j_2},...,[X_{j_{k-1}},X_{j_k}] \pl | \pl 1\le j_1,\cdots,j_k\le r\}  \]
spans the tangent space $T_x \M$ at every point $x\in \M$. We consider the
the sub-Laplacian
\[ \Delta_X \lel \sum_{j=1}^r X_j^*X_j \pl. \]
Here $X_j^*$ is the adjoint operator of $X_j$ with respect to $L_2(\M,\mu)$ and $d\mu$ is the measure given by the form, which in local coordinates is given by
\[  d\mu(x) \lel \sqrt{|g|}dx_1\wedge \cdots \wedge dx_n \pl , |g|(x)=|{\rm det}g_{ij}(x)|\pl.\]
For compact $\M$, $\Delta_X$ extends to a self-adjoint operator on $L_2(\M,\mu)$.
It follows from the famous Rothschild-Stein estimate \cite{SR} (see also \cite{OZ}) that $\Delta_X$ is hypo-elliptic. This leads to the estimate (see e.g. \cite{HelfferNier})
 \begin{equation}\label{hypo}
 \lan f,\Delta^{\frac{1}{l_X}}f\ran \le C (\lan f,\Delta_{X}f\ran + \norm{f}{2}^2 )   \pl,
  \end{equation}
  where $\Delta$ is the Laplace-Beltrami operator on $\M$. Using  the Hardy-Littlewood-Sobolev inequality in the Riemannian setting, we obtain the following Sobolev inequality (see e.g. \cite{LZ}):
\begin{lemma}\label{LZ}  Let $\M$ be a compact Riemannian manifold and $X$ be a H\"ormander system of $\M$. Let $q=\frac{2dl_X}{dl_X-1}$. Then
 \[ \|f\|_{q} \kl C ( \lan\Delta_X f,f\ran + \norm{f}{2}^2)^{1/2} \pl .\]
\end{lemma}
Now it is time to invoke the famous Varopoulos' Theorem about the dimension of semigroups.
\begin{theorem}[\cite{VSCC}]\label{Varo}Let $T_t:L_\infty(\Om,\mu) \to L_\infty(\Om,\mu)$ be a semigroup of positive measure preserving self-adjoint maps and $A$ is the generator.
The following conditions are equivalent: for $m\in \nz$,
 \begin{enumerate}
 \item[i)] $\|T_t:L_1(\Om,\mu)\to L_{\infty}(\Om,\mu)\|\kl C_1 t^{-m/2}$ for all $0\le t\le 1$ and some $C_1$;
 \item[ii)] $\|f\|^2_{\frac{2m}{m-2}} \kl C_2 (\lan Af,f\ran+ \norm{f}{2}^2)$;
 \item[iii)] $\|f\|_2^{2+4/m} \kl C_3 \lan Af,f\ran \|f\|_1^{4/m}$.
 \end{enumerate}
\end{theorem}
\begin{rem}{\rm Varopoulos' Theorem remains valid for
semi-finite von Neumann algebras. For the proof, the only part which requires modification is i)$\Rightarrow$ ii) (see \cite{JZ} and independently \cite{XX}).  The completely bounded norm analog is
significantly more involved \cite{JZh}, and it will be used later.}
\end{rem}

It is well-known that the Laplace-Beltrami operator $\Delta_{LB}$ on a compact Riemannian manifold has a spectral gap. Similarly, \eqref{hypo}
 shows that $\Delta_X$ also has a spectral gap. Combining  Lemma \ref{LZ} and Theorem \ref{Varo}, we obtain the kernel estimates in Proposition \ref{decay} for $m=dl_X$. As a consequence of Theorem \ref{erg}, we have:
\begin{theorem}\label{horm} Let $(\M,g)$ be a compact Riemannian manifold and $X=\{X_1,...,X_r\}$ be a H\"ormander system. Then there exists $m=dl_X\in \nz$ and $c>0$ such that $S_t=e^{-t\Delta_X}$ satisfies
 \[ \|S_t:L_1(\M,\mu)\to L_{\infty}(\M,\mu)\|_{cb}\kl c t^{-m/2} \pl .\]
Moreover, for every $0<\theta<1$, $S_t^{\theta}=e^{-t\Delta_X^{\theta}}$ satisfies $\la$-$\Gamma\E$ with $\la=c_0t_0^{-\theta}(1-\theta)\theta^2$. Here $t_0=t_0(\Delta_X)$ is the return time of $\Delta_X$ and $c_0$ is an absolute constant.
\end{theorem}

As mentioned in Corollary \ref{kkeeyy}, the $\Gamma\E$ condition automatically extends to the operator-valued setting for any finite Neumann algebra $M$. Here we note that the kernel estimates for H\"ormander systems also extend to $M$-valued functions.

\begin{cor} \label{horm'} Let $M$ be a finite von Neumann algebra with tracial state $\tau$. Let $(\M,\mu)$ be a compact Riemannian manifold and $X$ be a H\"ormander system as above.  Then
 \[ \|id\ten T_t:L_\infty(\M ; L_{1}(M))\to L_\infty(\M ; L_\infty(M))\|
 \kl c t^{-m/2} \]
holds for $0\le t\le 1$ and $c>0$. Moreover, $\Phi_F(\Delta_X\ten id_M)$ satisfies $\Gamma\E$.
\end{cor}

\begin{proof} Let $\ez(f)=\int_\M f(x) d\mu(x)$ be the conditional expectation onto $M$, i.e. $\ez=E_M$. Then  a positive element $f\in L_{\infty}^1(M\ssubset L_{\infty}(\M)\bar{\ten}M)$ has norm $\le 1$ if $\|\ez(f)\|_{M}\le 1$. Let $h\in L_2(M)$  be unit vector. Consider the scalar function $f_h(x)=\lan h,f(x)h\ran_{L_2(M)}$. We deduce that
 \[ \ez(f_h) \lel \int_{\M} f_h(x) d\mu(x) \lel \lan h,\ez(f)h\ran_{L_2(M)} \kl 1 \] and therefore $\|T_t(f_h)\|_{L_\infty(\M)}\le c t^{-m/2}$. This means
 \[ \sup_{\|h\|_2\le 1} \sup_{x\in \M}  \lan h,T_t(f)(x)h\ran_{L_2(M)} \kl ct^{-m/2} \pl . \]
Interchanging the double supremums, with the help of the  duality $L_1(\M,L_1(M))^*=L_{\infty}(\M)\bar{\ten}M$,
implies the assertion \qd

\subsection{Group representation}
Let $G$ be a compact group with Haar measure $\mu$.
 We consider a semigroup of positive measure preserving self-adjoint maps $S_t:L_{\infty}(G)\to L_{\infty}(G)$, which is also right translation invariant. Suppose that $S_t$ is given by the kernel \[  S_t(f)(g) \lel \int_G K_t(g,h)f(h) d\mu(h) \pl.\]
The right translation invariance means that for any $f\in L_\infty(G)$
 \begin{align*}
 \int K_t(gs,h)f(h) d\mu(h) &=
   S_t(f)(gs) \lel \int K_t(g,h) f(hs) d\mu(h)= \int K_t(g,hs^{-1}) f(h) d\mu(h) \pl .
   \end{align*}
Thus $K_t(gs,h)=K_t(g,hs^{-1})$ and hence $K_t(g,h)=k_t(gh^{-1})$ for some single variable function $k_t$. Conversely, $K_t(g,h)=k_t(gh^{-1})$ implies right invariance.

Now let $(M,\tau)$ be a finite von Neumann algebra and $\al:G \to Aut(M)$ be a action $G$ on $M$ of trace preserving automorphisms. Using the standard co-representation,
 \[ \pi: M \to L_{\infty}(G;M) \pl ,\pl \pi(x)(g) \lel \al_{g^{-1}}(x) \pl .\]
we define the \emph{transferred semigroup}
 \[ T_t(x) \lel \int_G k_t(g^{-1}) \al_{g^{-1}}(x) d\mu(g) \pl .\]
\begin{lemma}\label{comdiag} The semigroups $S_t$ and $T_t$ satisfy the following factorization property
 \[ \pi\circ T_t \lel (S_t\ten id_M)\circ \pi \pl .\]
\end{lemma}
\begin{proof} We include the proof for completeness. Indeed, for $x\in M$
 \begin{align*}
  \pi(T_t(x))(g) &= \al_g^{-1}\int_G k_t(h^{-1}) \al_{h^{-1}}(x) d\mu(h)
\lel
  \int_G k_t(gg^{-1}h^{-1}) \al_{(hg)^{-1}}(x) d\mu(h) \\
  &= \int_G k_t(gh^{-1}) \al_{h^{-1}}(x) d\mu(h)
\lel  (S_t\ten id)(\pi(x))(g) \pl . \qedhere
  \end{align*} \qd
Let us denote by $N=\{x\pl | \al_g(x)=x \pl \forall \pl g\in G\}$ the fixpoint subalgebra. Note that we have the following commuting diagram
\begin{equation}\label{ccss}
 \begin{array}{ccc} M\pl\pl &\overset{\pi}{\longrightarrow} & L_{\infty}(G,M) \\
                    \downarrow  E_N   & & \downarrow \ez=E_M  \\
                     N \pl\pl &\overset{\pi}{\longrightarrow} & M
                     \end{array} \pl .
                     \end{equation}
Here $M\subset L_{\infty}(G,M)$ is considered as operator-valued constant functions and, as seen in \ref{horm'}, the conditional expectation is given by integration.  Then for any $x\in M$,
 \[ \ez(\pi(x)) \lel \int_{G} \al_{g^{-1}}(x) d\mu(g)  \lel E_N(x) \pl \]
is exactly the conditional expectation form $M$ onto the fixpoint algebra $N$. Since $\ez$ is a unital complete positive $N$-bimodule map we see that
 \[ \ez  L_p^q(M\subset L_{\infty}(G,M))\to L_p^q(N\subset M) \pl \]
is completely contraction for all $1\le p,q\le \infty $. This implies that the inclusion $\pi: L_p^q(N\ssubset M) \subset L_p^q(M\ssubset L_{\infty}(G,M))$ is a completely isometric embedding (see \cite{JP} for details). Note that the conditions $\la$-$\Gamma\E$ and $\la$-CLSI pass to subsystems:

\begin{prop}\label{transf} Let $S_t:L_{\infty}(G)\to L_{\infty}(G)$ be an ergodic, right invariant semigroup and $T_t:M\to M$ be the transferred semigroup defined as above. Then
 \begin{enumerate}
 \item[i)] $\norm{T_t-E_N:L_2(M)\to L_2(M)}{}\le \norm{S_t-\ez:L_2(G)\to L_2(G)}{}$ for all $t>0$ and hence the spectral gap for $T_t$ (with respect to $E_N$) is not less than the spectral gap of $S_t$.
 \item[ii)] $(T_t)$ satisfies $\la$-$\Gamma\E$ (resp. $\la$-FLSI, $\la$-CLSI) if $(S_t)$ does.
 \end{enumerate}
 \end{prop}

Combining \eqref{sfc} with Proposition \eqref{decay} and Proposition \eqref{lowest}, we obtain the following application of transference:

\begin{theorem}\label{Ho2} Let $S_t:L_{\infty}(G)\to L_{\infty}(G)$ be an ergodic right invariant semigroup with kernel function $k_t$. Let $\si$ be the spectral gap of $S_t$ and suppose $\sup_g |k_t(g)|\kl c t^{-m/2}$ holds for some $c,m>0$ and $0\le t\le 1$. Then the transferred semigroup $T_t:M\to M$ and its generator $A$ with spectral gap $\la_{\min}\gl \si$ satisfy:
\begin{enumerate}
\item[i)] $ \|T_t:L_1(M)\to L_1^{\infty}(N\ssubset M)\|_{cb}\kl c t^{-m/2}$ for $0\le t\le 1$, and
 \[ \|T_t(id-E):L_1(M)\to L_1^{\infty}(N\ssubset M)\|_{cb}\kl   \begin{cases} 2ct^{-m/2} &0\le t\le 1 \pl ,\\
                   c(m,\la_{\min})e^{-\la_{\min} t} &  1\le t<\infty  \pl .\end{cases}\]
\item[ii)] For every function $F$ satisfying condition (I)$+$(QM) or (I)+($\Delta_2$), the generator $\Phi_F(A)$ satisfies $\Gamma\E$ and hence CLSI.
\end{enumerate}
\end{theorem}

Now we combine Theorem \ref{Ho2} with the kernel estimates for a H\"ormander systems on a Lie group. Let $G$ be a compact Lie group and $\mathfrak{g}$ be its Lie algebra (of right invariant vector fields). A  generating set $X=\{X_1,\cdots,X_r\}$ of $\mathfrak{g}$ is a right invariant H\"ormander systems on $G$. Indeed,
  \[ X(f) \lel \frac{d}{dt}f(\exp(tX)g)|_{t=0} \pl, \]
is right translation  invariant because the left and right translations commute. Then the sub-Laplacian $\displaystyle \Delta_X=\sum_{j=1}^r {X_j}^*{X_j}$ generates a right invariant semigroup $S_t=e^{-t\Delta_X}$
\begin{cor} \label{diagtrans}Let $X$ be a generating set of $\mathcal{G}$ and $S_t=e^{-t\Delta_X}:L_\infty(G) \to L_\infty(G)$ be the right invariant semigroup given by the sub-Laplacian $\Delta_X$. Then transferred semigroup $T_t:M\to M$ and its generator $A$ satisfy
 \begin{enumerate}
\item[i)] For every function $F$ satisfying condition $(I)+(QM)$ or $(I)+(\Delta_2)$, the generator $\Phi_F(A)$ satisfies $\Gamma\E$ and hence CLSI.
  \item[ii)] In particular, for all $0<\theta<1$ the generator $A^{\theta}$ satisfies $\la$-$\Gamma\E$ with ${\la(\theta,X)=c_0t_0^{-\theta}\theta^2(1-\theta)}$. Here $t_0=t_0(\Delta_X)$ is the return time of $\Delta_X$ and $c_0$ an absolute constant.
\end{enumerate}
\end{cor}

\subsection{Finite dimensional representation of Lie groups}Let $\mm_m$ be the $m\times m$ matrix algebra and $U_m$ be its unitary group.
A unitary representation $u:G\to U_m$ induces a representation $\hat{u}:\mathfrak{g}\to \mathfrak{u}_n$ of the corresponding Lie algebra, where $\mathfrak{u}_n=i(\mm_m)_{sa}$ is the Lie algebra of $U_m$ and $(\mm_m)_{sa}$ are the self-adjoint matrices in $\mm_m$. Let $X=\{X_1,...,X_r\}$ be a generating set of $\mathfrak{g}$ and $Y_1,...,Y_r\in (\mm_m)_{sa}$ be their images under $\hat{u}$. Indeed, for the exponential map $exp$, we have
 \[ u(\exp(tX_j)) \lel e^{itY_j} \]
and $Y_j\in \mm_{m}^{sa}$ is the corresponding generator for the one parameter unitary $u(\exp(tX_j))\subset M_m$. Let us consider the (self-adjoint) Lindblad generator is given by
 \[ \L(\rho) \lel \sum_{j=1}^r Y_j^2\rho+\rho y_j^2-2Y_j\rho Y_j \pl .\]
Then we have a concrete realization of Lemma \ref{comdiag}.

\begin{lemma}\label{dial} Let $\pi:\mm_m\to L_{\infty}(G,\mm_m)$ be given by
 \[ \pi(x) \lel u(g)^{-1}xu(g) \]
Then
 \[ \Delta_X(\pi(\rho)) \lel \pi(\L(\rho)) \pl, \pl X_j(\pi(x))\lel i \pi([Y_j,x])\pl. \]
\end{lemma}

\begin{proof} Let $x\in \mm_m$ and $h,k\in l_2^m$ being two vectors. We consider the scalar function
 \begin{align*}
  f(g) &= \lan h, \pi(g)^{-1}x\pi(g)(k)\ran
  \end{align*}
Then we have
  \begin{align*}
   X_j(f)(g) &= \frac{d}{dt}  f(\exp (tX)g) |_{t=0}
   \lel
   \frac{d}{dt} \lan h,u(g)^{-1} e^{-itY_j}x e^{itY_j} u(g)k\ran |_{t=0} \\
   &= i (h,u(g)^{-1}(xY_j-Y_jx)u(g) k) \pl .
   \end{align*}
Since $h,k$ are arbitrary, we deduce the second assertion. Note that $ {X_j}=-{X_j}^*$. Then \[{X_j}{X_j}\pi(x)
 \lel -\pi([Y_j,[Y_j,x]])
 \lel \pi( 2Y_jxY_j-Y_j^2x-xY_j^2) \pl .\]
and hence
 \[ \Delta_X(\pi(x))\lel \pi( \sum_j Y_j^2x+xY_j-2Y_jxY_j) \lel \pi(\L(x)) \pl .\]
%
This implies in particular that the semigroup $S_t=e^{-t\Delta_X}$ on $G$ satisfies
 \begin{align*} (S_t\ten id)\circ \pi \lel & \pi \circ e^{-t\L} \pl . \qedhere\end{align*}
 \qd

\begin{theorem}\label{Horm} Let $X=\{X_1,...,X_r\}$ be a generating set of $\mathfrak{g}$ and $u:G\to U_m$ be a unitary representation so that $\hat{u}(X_j)=Y_j$. Let
 \[ \L(x) \lel \sum_j Y_j^2x+Y_j^{2}-Y_jxY_j \pl. \]
Then  $A=\L^{\theta}$ satisfies $\la$-$\Gamma\E$ and hence $\la$-CLSI
with $\la(X,\theta)=c_0 t_0^{-\theta}\theta^2 (1-\theta)$ depending on $t_0=t_0(\Delta_X)$ and $\theta$.
\end{theorem}
We obtain the following corollary from the $cb$-version of Varopoulos' theorem \cite{JZh}.
\begin{cor} Let $G$ be a $d$-dimensional Lie group and $X$ be a generating set of $\mathfrak{g}$ using up to $l_X$-th iterated Lie bracket. Let $\L$ as above.  Suppose $S_t=e^{-tB}$ be a semigroup of completely positive self-adjoint trace preserving maps on $M_m$ such that
 \begin{enumerate}
  \item[i)] The fixed-point algebra $N_{\L}$ of $e^{-t\L}$ is contained in the fixed-point algebra $N_B$ for $e^{-tB}$;
   \item[ii)] $\lan x,\L^{\al} x\ran_{tr}\kl c (\lan x,B x\ran_{tr}+\norm{x}{2}^2 ) $ for some  $0<\al< \frac{dl_X}{2}$.
  \end{enumerate}
Then $B^\theta$ satisfies $\Gamma\E$ and hence CLSI.
 \end{cor}

\begin{proof} Denote $d_X=dl_X$. The cb-version of Varopoulos' theorem implies that
  \[ \|(I+\L)^{-\al/2}:L_2(M_m)  \to L_2^{q}(N_{\L}\subset M_m)\|_{cb}
  \kl c(q)  \]
holds for $\frac{1}{q}=\frac{1}{2}-\frac{\al}{d_X}$ provided $2\al< d_X$. By our assumption we have
 \[ \|(Id+\L)^{\al/2}(x)\|_2 \sim (\|x\|+ \|\L^{\al}(x)\|) \kl c (\|x\|+ \|B^{1/2}x\|_2)  \pl .\]
Using $N_{\L}\subset N_B$ we deduce that
 \[ \|(I+ B)^{-1/2}:L_2(M_m)\to L_2^q(N_B\subset M_m)\|_{cb} \kl c'(q)  \pl .\]
By $ii)\Rightarrow i)$ in Varopoulos' theorem we deduce that
 \[ \|e^{-tB}:L_\infty^1(N_B\subset M)_m\to  M_m\|_{cb} \kl c t^{-d_X/2\al}  \pl .\]
Thanks to the spectral gap for $B$, we may again use Theorem  \ref{erg} and deduce the assertion. \qd

\section{A density result}
In this section we show that on matrix algebra the set of self-adjoint generators with $\Gamma\E$ is dense. Let $(T_t=e^{-tA}:\mm_n\to \mm_n)$ be
 a semigroup of self-adjoint and unital completely positive maps. Using the Lindblad form, we may assume that
 \[ \L(x) \lel \sum_{k=1}^m a_k^2x+xa_k^2-2a_kxa_k\]
and then the corresponding derivation is given by
 \[ \delta:\mm_n \to \oplus_{k=1}^m\mm_n\pl, \pl \delta(x) \lel ([a_k,x])_{k=1}^m       \pl .\]
Therefore, the fixpoint algebra is given by
\[ N\pl=\pl\{ x | \delta(x)\lel 0\} \lel \{a_1,...,a_m\}'  \pl .\]
It is easy to check that
 \[ \Gamma_\L(x,y) \lel \sum_k [a_k,x^*][a_k,y] \pl .\]

\begin{lemma}\label{H1}Let $A_j=ia_j$. Then $X=\{A_i,...,A_r\}$ is a H\"ormander system of some compact connected Lie group $G$.
\end{lemma}

\begin{proof} Let $A_j=ia_j$ and $\mathfrak{k}\subset \mathfrak{u}_n$ the real Lie algebra generated by $A_1,...,A_r$. Let $G_1$ be the closed subalgebra generated by $e^{tA_j}$, i.e. the closure of the Lie algebra $e^{\mathfrak{k}}$. We may define the map $V={\rm span}\{A_1,...,A_r\}$ and the map  $\phi:G_1\times V\to \Mz_m^{sa}$. Note that $\delta_A(x)=[A,x]$ is a Schur-mutliplier (since $A$ is normal) and hence
 \[ e^{t\delta_A}
 \lel \sum_j \frac{t^j}{j!} \delta_A^j \]
also leaves $\mathfrak{k}$ invariant, certainly for elements $A\in\{A_1,...,A_r\}$ and then by induction also for elements in $\mathfrak{k}$. Thanks to the definition of $G_1$ we deduce that $\mathfrak{k}$ is $G_1$ invariant. Certainly the inner product given by the trace on $\mathbb{B}(S_2^m)$ is invariant with respect to conjugation $\si(g)(T)=gTg^*$  and leaves $\mathfrak{k}$ in variant. By differentiation the Killing form, the restriction of the trace to an invariant subspace,   is $G_1$ in variant. According to \cite[5.30]{Lie1}, we see that $\mathfrak{k}$ and moreover, $H=e^{\mathfrak{k}}$ is itself a connected Liegroup with respect to a modified topology, see \cite[5.20]{Lie1}. Certainly, $A_1,...,A_m$ is a H\"ormander system for any Lie group with Liealgebra $\mathfrak{k}$. Hence it suffices to show that $k$ is a compact Lie algebra. First we note that for iterated commutators $C=[A,B]$ of elements satisfying $A^*=-A$, $B^*=-B$, we still have $C^*=-C$. Thus $\mathfrak{k}_{\rz}$, and $\mathfrak{k}_{\cz}$ are invariant under $^*$ operation. According to \cite[Prop 1.56]{Liea}, we see that $\mathfrak{k}$ is reductive, and the real part of a complex Liealgebra. Let $\mathfrak{k}=\mathfrak{t}+[\mathfrak{k},\mathfrak{k}]$. Then $\mathfrak{k}_0=[\mathfrak{k},\mathfrak{k}]$ is semisimple. According to \cite[Theorem 1.42]{Liea}, we know that the Killing form is non-degenerate on $\mathfrak{k}_0$. According to \cite[Prop 4.27]{Liea}, we see that
$\mathfrak{k}_0$ is compact. According to \cite[Theorem 5.11]{Lie1}, we find that the simply connected Lie algebra $H$ is a product of the commutative algebra with finite dimensional Lie algebra and $H_0$, and hence $H$ is indeed compact. This argument follows \cite{One}. \qd

Our aim is now to find a suitable approximation of the form $B_{\epsilon}=\phi_{\epsilon,\si}(\L)$ which satisfy a $\Gamma \E$ estimate and are close to $\L$ in operator norm on $L_2(\mm_n,tr)$. We apply the technique from Section 3 and define for a fixed $\si>0$,
 \[ F_{\epsilon,\si}(t) \lel 1_{[\eps,1)}(t) t^{-2} + 1_{[1,\infty]}(t) t^{-\si} \pl .\]

\begin{lemma}\label{apcalc} Let $\eps>0$. Define
\[\phi_{\eps,\si}(\la) \lel (-\ln\eps)^{-1}\int_{\eps}^{\infty}(1-e^{-t\la})F_{\eps,\si}(t) \frac{dt}{t^2}\pl.\]
Then
 \begin{enumerate}
 \item $\|\L-\phi_{\eps,\si}(\L)\|\kl \frac{2\si^{-1}+\|\L\|^2}{|2\ln \eps|}$;
 \item If $c(\L)\Gamma_{I-E}\le \Gamma_{I-T_t}$ holds for $t\gl t_0\gl 1$, then
     \[ \frac{c(\L)}{(1+\si|\ln \eps|)t_0^{\si}}
 \Gamma_{I-E}
     \kl \Gamma_{\phi_{\eps,\si}(\L)}
      \pl .\]
 \end{enumerate}
\end{lemma}
\begin{proof} Using differentiation we have that
$x-\frac{x^2}{2}\le 1-e^{-x}\le x$. Define $\displaystyle \psi(\la)=\int_{\eps}^1 (1-e^{-\la t}) \frac{dt}{t^2}$. Then
 \begin{align*}
 |\ln \eps|\la -\frac{\la^2}{2} &\kl  \int_{\eps}^1 (\la t -\frac{\la^2t^2}{2})\frac{dt}{t^2}
      \kl \psi(\la) \kl \int_{\eps}^1 \la t  \pl \frac{dt}{t^2}
       \lel |\ln \eps|\la \pl .
      \end{align*}
Write $\displaystyle \tilde{\psi}(\la)=\int_1^{\infty} (1-e^{-\la t}) \frac{dt}{t^{1+\si}}$. Note that $0\le \tilde{\psi}(\la)\kl \si^{-1}$. Then we find
 \[ |\ln\eps| \la -\frac{\la^2}{2|\ln \eps|}
 \kl \phi_{\eps,\si}(\la) \kl
 \la+\frac{1}{\si |\ln \eps|} \pl .\]
By functional calculus, we deduce that
 \[ \|\L- \phi_{\eps,\si}(\L)\| \kl \frac{1}{\si|\ln \eps|}+ \frac{\|\L\|^2}{2|\ln \eps|}
  \kl \frac{2\si^{-1}+\|\L\|^2}{2|\ln \eps|} \pl .\]
For the second assertion, we observe that by linearity of $\Gamma_A$ in the variable $A$
 \[ \Gamma_{\phi_{\eps,\si}(\L)}
 \gl c(\L) |\ln\eps|^{-1}(\int_{t_0}^{\infty} \frac{dt}{t^{1+\si}}) \Gamma_{I-E}
 \gl \frac{c(\L)}{\si|\ln\eps|}  t_0^{-\si} \Gamma_{I-E} \pl .\]
This completes the proof of ii).\qd

\begin{rem}\label{tnot} {\rm a) An interesting choice is $\si=\frac{1}{\ln t_0}$. Then we find
 \[ \Gamma_{\phi_{\eps,\si}(\L)}\gl \frac{c(\L)}{|\ln t_0|} \frac{1}{|\ln \eps|}\Gamma_{I-E} \pl ,
 \]
and $\|\L-\phi_{\eps,\si}(\L)\|\kl \frac{2|\ln t_0|+\|\L\|^2}{2|\ln \eps|}$.
b) We can also slightly improve the lower estimate. Let $\beta>0$. The function $g(x)=1-e^{-x}$ is concave and hence $\frac{1-e^{-x}}{x}\gl (1-\frac{x}{2})$ implies
 \begin{align*}
  \int_{\eps}^{\eps^{\beta}} (1-e^{-\la t})\frac{dt}{t^2} &\gl
  \la \frac{1-e^{-\|\L\|\delta^{\beta_1}}}{\|\L\|\eps^{\beta}}
  \int_{\eps}^{\eps^{\beta}}\la t \frac{dt}{t^2} \gl (1-\frac{\|\L\|\eps^{\beta}}{2})(1-\beta)\la |\ln\eps|  \pl .
 \end{align*}
Thus assuming $|\ln \eps|\gl \frac{2}{\beta}|\ln \frac{\|\L\|}{2\beta}|$ implies
 \[  -2\beta \la \kl \phi_{\eps,\si}(\la)-\la\kl \frac{1}{\si|\ln\delta|} \]
 and hence for $\si=\frac{1}{\ln t_0}$, $t_0\gl 1$ we get
  \[ \|\phi_{\eps,\si}(\L)-\L\|\kl 2\beta \|\L\|+ \frac{1}{\ln t_0 |\ln\eps|} \pl .\]

c) Estimating the return time $t_0$ through the H\"omander system may not be very concrete. Nevertheless, we only need to know $\|T_{1/2}:L_1(\mm_n)\to L_1(N\ssubset \mm_n)\|$ and the spectral gap of $\L$ to control $t_0$.
 } \end{rem}

\begin{theorem}\label{predense} Let $L$ be the generator of a semigroup of unital completely positive and self-adjoint maps $T_t=e^{-tA}$ on $\Mz_m$. Then there exists a constant $\al(\L)$ such that for every $\eps>0$ there exists a generator  $B_{\eps}$, obtained from functional calculus of $\L$, such that
 \[ \|\L-B_{\eps}:L_2(\Mz_m)\to L_2(\Mz_m)\| \kl \eps \quad \mbox{and} \quad
  \eps \al(\L)  \Gamma_{I-E_N} \kl \Gamma_{B_{\eps}}
   \pl .\]
 Moreover, we have the estimate
  \[ \al(\L) \le \Big(2\ln t_0 (\ln t_0+\|\L\|^2) \Big)^{-1} ,\]
  where $t_0$ is the return time of $\L$.
\end{theorem}

\begin{proof} According to Lemma \ref{H1}, we know that $X=\{X_1,\cdots,X_r\}$ is a H\"ormander system. By construction, the corresponding Lie group representation  satisfies $\pi(X_k)=ia_k$. Thus the proof of Theorem \ref{Horm} applies and shows that there exists a $t_0\gl e$, depending on $a_1,...,a_r$, such that $\Gamma_{I-T_t}\gl \frac{1}{2}\Gamma_{I-E_N}$ for $t\gl t_0$. Since $\Delta_{X}$ is ergodic, and the Lie algebra of $G$ is generated by $ia_1,...,ia_r$, we see that $\al_g(x)=x$ if $e^{ita_j}xe^{-ita_j}=x$ for all $j=1,...,r$. Thus the fixed point algebra is indeed $N=\{a_1,...,a_r\}'$ the commutator and also the fixpoint algebra of $e^{-t\L}$. Now, we choose $\si=\frac{1}{\ln t_0}$ and deduce that for every $0<\eps_0<1$
 \[ \|\phi_{\eps_0,\si}(\L)-\L\| \kl \frac{2|\ln t_0|+\|\L\|^2}{2|\ln \eps_0|} \]
and $\Gamma_{\phi_{\eps_0,\si}(\L)}\gl \frac{1}{2e \ln t_0 |\ln\eps_0|}\Gamma_{I-E}$. Thus we may choose $0<\eps_0<1$ such that ${|\ln\eps_0|=\frac{|\ln t_0|+\|\L\|^2/2}{\eps}}$ and obtain
 \[ \Gamma_{\phi_{\eps_0,\si}(\L)}\gl \frac{\eps}{2e (\ln t_0) (\ln t_0+\|\L\|^2)}\Gamma_{I-E} \pl .\]
Thus $\al(\L)= \frac{1}{2e\ln t_0 (\ln t_0+\|\L\|^2)}$ does the job.
  \qd
\begin{rem}{\rm We can improve the dependence in $\|\L\|$ using \eqref{tnot} b). We need $4\beta\|\L\|\lel \eps$ and $|\ln \eps_0|\gl \frac{2|\ln t_0|}{\eps}$ and
\[ |\ln \eps_0|\gl \frac{2}{\beta}|\ln \frac{\|\L\|}{2\beta}| \lel \frac{8\|\L\|}{\eps} |\ln \frac{2\|\L\|^2}{\eps}|   \pl .\]
And then  we obtain the non-linear estimate
 \begin{align*}
  \Gamma_{B_{\eps}}& \gl      \frac{\eps}{2e (\ln t_0) (8\|\L\|+2\ln \|\L\|+|\ln \eps| +2\ln t_0)}  \Gamma_{I-E} \pl .
  \end{align*}
Note that $t_0=\frac{\ln c_0\|T_{}\|_{cb}}{\la_{\min}(\L)}$ only depends linearly on $1/\la_{\min}(\L)$. Hence, up to the cb-estimate, our estimate just depends on the minimal and maximal eigenvalue of $\L$.}\end{rem}

\begin{cor} The set of generators of unital completely positive self-adjoint semigroups on $\Mz_m$ satisfying $\Gamma \E$ is dense.
\end{cor}

\begin{rem} {\rm In \cite{Dan} it was shown that for every ergodic semigroup  of completely positive trace preserving maps there exists a entanglement-breaking time $t_{EB}$ such that $T_t$ is eventually entanglement-breaking for $t>t_{EB}$. A completely positive trace preserving map is called entanglement-breaking if its Choi matrix is a convex combination of tensor product density matrices. Our kernel estimate can be used to estimate this entanglement breaking time $t_{EB}$.
}
\end{rem}
\section{Geometric applications and deviation inequalities}
The aim of this section is to derive several concentration inequalities for semigroups satisfying FLSI in the non-ergodic and possibly infinite dimensional situation. The starting point is a version of Rieffel's quantum metric space. Let $T_t:M\to M$ be a semigroup of unital completely positive and self-adjoint maps and $A$ be the generator of $T_t$. As usual we will assume that $\A\subset {\rm dom}(A^{1/2})$ is a dense $^*$-algebra and invariant under $T_t$. On $M$ we define the Lipschitz norm via the gradient form,
\[ \|f\|_{Lip_{\Gamma}} \lel  \max\{\|\Gamma(f,f)\|^{\frac{1}{2}}, \|\Gamma(f^*,f^*)\|_{\infty}\}^{\frac{1}{2}} \pl , f\in \A\pl.\]
This induces a quantum metric on the state space by duality
 \[ \|\rho\|_{\Gamma^*}
 \lel \sup \{ |\tau(\rho f)| \pl | \pl E(f)=0\pl , \pl \|f\|_{Lip_{\Gamma}}\le 1\}\pl. \]
Usually such a Lipschitz norm is considered in the ergodic setting, where the fixpoint subalgebra $N=\cz 1$ and hence the conditional expectation is given by $E(f)=\tau(f)$. Since for states $\rho(1)=1$, one can assume the additional condition $E(f)=0$ when calculating the distance $d_{\Gamma}(\rho,\si)=\|\rho-\si\|_{\Gamma^*}$. This is crucial in the non-ergodic situation, see the last section of \cite{JRZ} for more detailed discussion. Let $\delta:{\rm dom}(A^{1/2})\to L_2(\hat{M})$ be the derivation which implements the gradient form
 \[ \Gamma(x,y) \lel E_M(\delta(x)^*\delta(y))\pl. \]
In the construction of a derivation in \cite{JRS} the following additional estimate was also proved.
 \begin{equation}\label{JRS}
  \|\delta(x)\|_{\hat{M}} \kl 2\sqrt{2} \max\{\|\Gamma(x,x)\|^{1/2},\|\Gamma(x^*,x^*)\|^{1/2}\}
  \lel 2\sqrt{2}\|x\|_{Lip_\Gamma} \pl .
  \end{equation}


\subsection{Wasserstein $2$-distance and transport inequalities}
In this part we review and extend the work of Carlen-Maas \cite{CM} which adapted Otto-Vilani's theory \cite{OV} to the non-ergodic self-adjoint setting.  Following \cite{CM} we use the symbol
\[ [\rho](x) \lel \int_0^1 \rho^sx\rho^{1-s} ds \]
for the multiplier operator, and
 \[ [\rho]^{-1}\lel  \int_0^{\infty} (\rho+t)^{-1}x(\rho+t)^{-1} dt \pl .\]
for the inverse. The need of the symmetric two-sided multiplication instead of multiplication with $\rho$ is major difference between the commutative and noncommutative setting. Let us recall a key formula which recovers the generator from the logarithm  as follows:
 \begin{align} \label{derr}
  A(\rho) &=    \delta^*\Big([\rho](\ln \rho-\ln(E(\rho))\Big) \end{align}
 Indeed, let us assume that $\rho$ and $x\in \mathcal{A}$ and write $\si=E(\rho)$. Using the operator integral $J_F$ for $F(x)=\ln(x)$, we deduce from $\delta(f(\si))=0$ that
 \begin{align*}
 &\langle x, \delta^*[\rho](\ln \rho-\ln(E(\rho)))\rangle \lel  \tau\Big(\delta(x^*)[\rho]\delta(\ln \rho ))\Big)
 -\tau\Big(\delta(x^*)[\rho](\delta(\ln \si))\Big) \\
 &= \tau\Big(\delta(x^*)[\rho]J_F^{\rho}(\delta(\rho)\Big) \lel
 \tau\Big(\delta(x^*)[\rho][\rho]^{-1}(\delta(\rho)\Big)\\
 &= \tau(\Gamma(x,\rho)) \lel
 \frac{1}{2} \Big(\tau(A(x^*)\rho)+\tau(x^*A(\rho))-\tau(A(x^*\rho))\Big) \lel
 \tau(x^*A(\rho)) \pl .
 \end{align*}
Here we used $A=A^*$ and $A(1)=0$. The expression $\ln \rho-\ln E(\rho)$ itself occurs by differentiating the relative entropy $D_N(\rho)=D(\rho||E(\rho))$. Consider $g(t)=\rho+t\beta$ with a self-adjoint $\beta$. Using the derivation formula \eqref{derv} for $F(x)=x\ln x$ with derivative $F'(x)=1+\ln x$, we deduce from the tracial property that
 \begin{align*}
 \frac{d}{dt}D_N(\rho+t\beta)|_{t=0}&=
 \frac{d}{dt}\tau(F(\rho+t\beta))-\frac{d}{dt}\tau(F(E(\rho+t\beta)))|_{t=0} \\&\lel \tau(F'(\rho)\beta)-\tau(F'(E(\rho))\ln E(\beta)) \\
 &\lel  \tau(\beta)+\tau((\ln \rho)\beta)-\tau(E(\beta))-\tau(E(\ln\rho)E(\beta))\\& \lel
 \tau((\ln \rho-\ln E(\rho))\beta) \pl .
 \end{align*}
This means the Radon-Nikodym derivative of $D_N$ with respect to the trace satisfies
 \begin{equation}\label{totder}
   \frac{ d D_N'(\rho)}{d\tau} \lel \ln \rho-\ln E(\rho)\pl.
   \end{equation}
In the following we will identify a normal state $\phi_\rho(x)=\tau(x\rho)$ of $M$ and its density operator $\rho$.
\begin{defi}\label{w2} Given a state $\rho\in M$, we define the weighted $L_2$-norm on $L_2(\hat{M})$ by the inner product
 \[ \langle \xi,\eta\rangle_{\rho}:= \langle \xi,[\rho]\eta\rangle_{L_2(\hat{M})}
 \lel  \int_0^1  \tau_{\hat{M}}(\xi^* \rho^{1-s}\eta \rho^s) ds \pl .\]
\end{defi}
If $\rho$ is invertible and $\mu 1 \le \rho \le \mu^{-1}1$, we have
 \[ \mu \langle \xi,\xi\rangle \kl
  \langle \xi,\xi\rangle_{\rho} \kl \mu^{-1} \langle \xi,\xi\rangle   \pl .\]
Hence for all invertible $\rho$, the weighted $L_2$ norm $\norm{{\atop} \pl }{\rho}$ is equivalent to the norm on  $L_2(\hat{M})$ norm. However, this change of metric  is crucial in introducing the following Riemannian metric. Recall that $\mathcal{H}_\Gamma=\mathcal{H}$ is the $W^*$-submodule of $L_\infty^c(M\subset \hat{M})$ generated by $\delta(\mathcal{A})\mathcal{A}$.

\begin{lemma}\label{weakaprox} Let $\rho$ be a normal state of $M$. For $z\in M$, define
 \[ \|z\|_{\Tan_{\rho}}
 \lel \inf\{  \|\xi\|_{\rho} \pl | \pl  \delta^*([\rho]\xi )=z\} \pl .\]
Here the infimum is taken over all $\xi\in \mathcal{H}$. Then there exists $a_n\in \mathcal{A}$ such that $\|\delta(a_n)\|_{\rho}\le \|z\|_{\Tan_{\rho}}$ and $z=\lim_n \delta^*([\rho]\delta(a_n))$ holds weakly.\end{lemma}
\begin{proof}[Proof of Lemma \ref{weakaprox}]  Observe that for $x\in\mathcal{A}$, $[\rho](x)$ belongs to the closure of $\mathcal{H}$.  We say that $\xi$ in $\mathcal{H}$ is divergence free if $\delta^*(\xi)=0$. Let $\xi_0$ be in the closure of $\mathcal{H}$ such that \[ \|z\|_{\Tan_\rho}^2= \|\xi_0\|_{\rho}^2\pl, \pl  \delta^*([\rho](\xi_0)=z\pl.\]
Write $\xi_{\eps} \lel \xi_0 + \eps [\rho]^{-1}(\xi)$. It
satisfies
 \[ \|\xi_0\|_\rho^2 \le \|\xi_0+\eps [\rho]^{-1}(\xi)\|_\rho^2
 \lel \|\xi_0\|^2_{\rho}+2\eps {\rm Re}\lan \xi_0,[\rho]^{-1}(\xi\ran)_\rho+\eps^2\|[\rho]^{-1}(\xi)\|_{\rho}^2    \]
and hence $\lan \xi_0,[\rho]^{-1}(\xi)\ran_{\rho}\lel 0$ for all divergence free $\xi$. Equivalently, we find
   \[ \tau(\xi_0^*\xi) \lel \tau(\xi_0^*[\rho][\rho]^{-1}(\xi)) \lel 0 \pl. \pl \]
 Let us rephrase this in terms of gradient form,
 \[ \langle a\ten b,c\ten d\rangle_{\Gamma}
 \lel \tau(b^*\Gamma(a,c)d) \pl .\]
An element $\xi$ is divergence free if and only if
 \[  \langle x\ten 1,\xi\rangle_{\Gamma}\lel 0 \]
for all $x$. Hence $\xi_0$ is orthogonal to the divergence free forms if and only if $\xi_0$ is in the closure of $\delta(x)$. In other words there exists a sequence $a_n\in \A$ such that $\xi_0\lel \lim_n \delta(a_n)$. This implies
  \[ \tau(b^*z)\lel  \tau(\delta(b^*)[\rho]\xi_0)
  \lel \lim_n \tau\Big(b^*\delta^*([\rho]\delta(a_n))\Big) \pl \]
for all $b\in \A$. That completes the proof.
\qd
\begin{rem} {\rm a) If $z$ is selfadjoint, we may use the fact that $\delta$ is $^*$-preserving to show that $\xi_0\in \hat{M}$ is also self-adjoint. Thus we may replace $a_n$ by their self-adjoint parts using the fact that $[\rho]$ preserves self-adjointness.\\
b) Since $A$ is self-adjoint, we know that the range of $A$ is dense in $(I-E)(L_2(M))$, the orthogonal complement of $L_2(N)$, and hence contained in the closure of  $\delta^*(\delta(\A)\A)\subset L_2(N)^{\perp}$. In fact, the $L_2$-closure of  $\delta^*(\delta(\A)\A)$ is exactly $(I-E)(L_2(M))$.
}
\end{rem}
In the following we denote by $\mathcal{H}_{\rho}$ the closure of $\delta(\mathcal{A})\mathcal{A}$ with respect to the $\norm{\cdot}{\rho}$ norm. $\mathcal{H}_{\rho}$ is viewed as the tangent space at the point $\rho$ and $\norm{\cdot}{Tan_\rho}$ gives a the Riemannian metric at $\rho$.

\begin{cor}\label{nnn} Let $\rho$ be a density operator of $M$. Then
 \[ \|x\|_{\Gamma^*}\kl 2\sqrt{2}\|x\|_{\Tan_{\rho}} \pl .\]  \end{cor}

\begin{proof} Let $a_n\in \mathcal{A}$ such that $\displaystyle \lim_{n\to \infty}\delta^*([\rho]a_n)=x$. We may assume that
 \[ \|\delta(a_n)\|_{\rho}\kl (1+\eps)\|x\|_{\Tan_{\rho}} \pl \]
for a given $\eps>0$.  Then we deduce that for $f\in \mathcal{A}$ we have
 \begin{align*}
 |\tau(f^*x)| &= \lim_n |\tau(f^*\delta^*([\rho]\delta(a_n)))|
 \lel   \lim_n \tau(\delta(f)^*[\rho]\delta(a_n))\\
   &\le \limsup_n  \|\delta(f)\|_{\rho} \|\delta(a_n)\|_{\rho} \kl  (1+\eps)\|\delta(f)\|_{\rho} \|x\|_{\Tan_{\rho}} \pl .
   \end{align*}
Furthermore, we deduce from the fact that the inclusion $M\subset \hat{M}$ is trace preserving that
 \begin{align*}
  \|\delta(f)\|_{\rho}^2
  &=\tau(\delta(f)^*[\rho]\delta(f))
  \lel \int_0^1 \tau(\delta(f)^*\rho^{1-s}\delta(f)\rho^s) ds  \\
 &\le \|\delta(f)\|_{\infty}^2 \int_0^1 \|\rho^{1-s}\|_{\frac{1}{1-s}}\|\rho^s\|_{\frac{1}{s}} ds \kl \|\delta(f)\|_{\infty}^2 \pl .
\end{align*}
Thus the estimate \eqref{JRS} implies the assertion, after sending $\eps$ to $0$.
\qd

Denote $S(M)$ as the set of normal states of $M$.
Let $F:S(M)\to \rz$ be a function defined (on a open dense subset) of the state space. We say that $F$ admits a \emph{gradient} $\grad_{\rho}F$ with respect to the tangent metric, if for every  $\rho$ there is an vector $\xi\in \mathcal{H}_{\rho}$ such that for every differentiable path $\rho:(-\eps,\eps)\to S(M)$
 \[ \rho'(0) \lel \delta^*([\rho]\xi_0) \quad \Longrightarrow \quad
  \frac{d}{dt}F(\rho(t))|_{t=0}
  \lel \langle \xi,\xi_0\rangle_{\rho} \pl .\]
and we write $\grad_{\rho}F=\xi$. Our control function is the relative entropy with respect to the fixpoint algebra  \[F(\rho)=D_N(\rho)=D(\rho||E(\rho))\pl .\] Let $\rho:(-\eps,\eps)\to S(M)$ be a path. Using the directional derivative of $D_N$ from \eqref{totder}, we find that at $\rho=\rho(0)$
  \begin{align*}
  F'(\rho(0))
  &= \tau\Big((\ln \rho-\ln E(\rho))\delta^*([\rho]\xi_0)\Big)
  \lel \tau\Big(\delta(\ln \rho-\ln E(\rho))[\rho]\xi_0\Big) \lel
 \lan\delta(\ln \rho),\xi_0\ran_{\rho} \pl .
  \end{align*}
By definition that means
 \begin{equation}
 {\rm grad}_{\rho}D_N \lel \delta(\ln\rho)  \pl .
 \end{equation}
Note that in \cite{CM} the inner product with the modified multiplication was exactly designed to satisfy this property. Moreover, we find the correspond tangent direction in the dual of the state space  is given by
 \[ \delta^*([\rho]\grad_{\rho}F)
 \lel \delta^*([\rho]\delta(\ln \rho))
 \lel \delta^*([\rho][\rho]^{-1}\delta(\rho))\lel A(\rho) \pl .\]
A curve $\gamma$ in the state space is said to follow the path of \emph{steepest descent} with respect to $F$ if
 \[ \frac{dF(\gamma)}{dt} \lel -\|\grad_{\gamma(t)}F(\gamma(t))\|_{\gamma(t)}^2
  \pl. \]
The right hand side $E(\rho)=\|\grad_{\rho}F\|_{{\rho}}^2$ is the energy function with respect to $F$. In our special case, $F=D_N$, we find \begin{align*}
 E(\rho) &= \|\grad_{\rho}D_N\|_{\rho}^2
 \lel \lan \delta(\ln \rho)\delta(\ln(\rho)\ran_{\rho} \lel  \lan[\rho]^{-1}\delta(\rho), \rho),[\rho][\rho]^{-1}\delta(\rho)\ran_{L_2(\hat{M})} \\
 &= \tau(\delta(\rho) [\rho]^{-1}\delta(\rho))  \lel \tau(\rho\delta^*\delta(\ln \rho)) \lel
  \tau(\rho A(\ln(\rho))
 \lel \tau(A(\rho)\ln\rho) \lel \I_A(\rho) \pl .
 \end{align*}
This means the sub-Riemannian metric is chosen so that the Fisher information for $F=D_N$ satisfies
 \[ \frac{d}{dt}D_N\circ T_t|_{t=0}=-\I_A  \pl. \]
We summarizes the above discussion as follows.
\begin{prop}\label{steep}
Suppose a differentiable curve $\gamma:(a,b)\to S(M)$ satisfies that
 \[ \gamma'(t)\lel -A(\rho(t)) \]
Then the curve $\gamma$ follows the path of steepest descent with respect to $D_N$. In particular, the semigroup path $\gamma(t)=T_t(\rho)$ is  a curve of steepest descent for $D_N$.
\end{prop}

\begin{lemma} Let $F(\rho)=D_N(\rho)$ and $E(\rho)=I_A(\rho)$. Let $\rho:[0,\infty)\to S(M)$ be a path of steepest descent with respect to $F$ and $\la>0$. Then
 \[ 2\la F(\rho(t))\kl E(\rho(t)) \quad \mbox{implies} \quad F(\rho(t))\kl e^{-2\la t}F(\rho(0)) \]
\end{lemma}

\begin{proof} According to the above discussion, we have
 \[\rho'(t)=-A(\rho(t))=-\delta^*\Big(\rho(t)F'(\rho(t))\Big)\pl, F'(t)=\langle \grad_{\rho(t)}F, -F'(\rho(t))\rangle_{\rho(t)}\lel -E(\rho(t))\pl .\]
Then our assumption implies that
 \[ F'(t)\kl -2\la F(t) \]
and hence $F'(t)\kl e^{-2\la t}F(t)$ by Gr\"onwall's Lemma.\qd

The standard Riemannian distance on $S(M)$ of our metric is given by
 \[ d_{A,2}(\rho,\si) \lel \inf\{ L(\gamma): \gamma(0)=\rho, \gamma(1)=\si\}\]
where the infimum runs over all piecewise smooth curve and the length function is defined by
 \[ L(\gamma) \lel \int_0^1 \|\gamma'(t)\|_{\gamma(t)} dt \pl .\]
Thanks to Lemma \ref{nnn}, and the definition, we still have the distance estimate
\[ \|\rho-\si\|_{\Gamma^*}\kl 2\sqrt{2} \pl d_{A,2}(\rho,\si) \pl .\]
The following result follows similarly form \cite{CM}[Theorem 8.7] using the path of steepest descent and  the relative entropy. Note that in \cite{CM} the generalized log-Sobolev inequality is defined with constant $2\la$.
\begin{theorem}\label{CM} The $\la$-FLSI inequality
 \[ \la D(\rho||E(\rho))  \kl \I_{A}(\rho) \]
implies
 \[ d_{A,2}(\rho,E(\rho))\kl 2\sqrt{\frac{D(\rho||E(\rho))}{\la}} \pl .\]
\end{theorem}
We have the following corollary using the $\Gamma$ Lipschitz distance.
\begin{cor} $\la$-FLSI implies
 \[ \|\rho_1-E(\rho_1)\|_{\Gamma^*}\kl 4 \sqrt{\frac{2D(\rho||E(\rho))}{\la}} \]
and
 \[  \|\rho_1-\rho_2\|_{\Gamma^*}
 \kl 4\sqrt{2}\Big(\sqrt{\frac{D(\rho_1||E(\rho_1))}{\la}}+\sqrt{\frac{D(\rho_2||E(\rho_2))}{\la}}\Big) \pl . \]
\end{cor}

\begin{proof} The first inequality is just a combination of Lemma \ref{nnn} and Theorem \ref{CM}. For the second inequality, we observe that $\|E(\rho)\|_{\Gamma^*}=0$, and hence the triangle inequality implies
 \begin{align*} \|\rho_1-\rho_2\|_{\Gamma^*}
 &\le  \|\rho_1-E(\rho_1)\|_{\Gamma^*}+\|\rho_2-E(\rho_2)\|_{\Gamma^*} \pl . \qedhere
 \end{align*}
\qd
\begin{rem}\label{geoT}{\rm Let $e$ be a projection. Then $\rho_e=\frac{e}{\tau(e)}$ is the normalized state which satisfies
 \[ D(\rho_e||E(\rho_e))\kl \tau(\rho_e\ln \rho_e)\kl -\ln \tau(e) \pl .\]
Assume  that there exists a self-adjoint $y$ such that
 \[ h\kl |\frac{\tau(e_1y)}{\tau(e_1)}-\frac{\tau(e_2y)}{\tau(e_2)}| \quad \mbox{and}\quad  \|\Gamma(y,y)\|\le 1 \pl.\]
Then we find the \emph{geometric} version of  Talagrand's inequality (see \cite{TA} and also \cite{JZ})
 \[ \tau(e_1)\tau(e_2) \kl e^{-\frac{\la h^2}{64}} \]
Indeed, this follows from the triangle inequality
 \[ h\kl \|\rho_{e_1}-\rho_{e_2}\|_{\Gamma^*}\kl
 4\sqrt{2}\la^{-\frac{1}{2}} (\sqrt{-\ln\tau(e_1)}+\sqrt{-\ln \tau(e_2)})
 \kl 8\la^{-\frac{1}{2}}\sqrt{-\ln\tau(e_1)-\ln\tau(e_1)} \pl .\]
The constant $64$ is probably not optimal in general.
}
\end{rem}

\subsection{Wasserstein $1$-distance and concentration inequalities}

In \cite{JZ} the commutative characterization of Wasserstein-Entropy estimates in terms of concentration inequalities was extended to the noncommutative setting. In the non-ergodic setting, we have the following result.

\begin{theorem}\label{characW} Let $(M,\tau)$ be a finite von Neumann algebra and $(T_t)$ be a self-adjoint semigroup of completely positive trace reducing maps. Let $N$ be the  fixed-point subalgebra. Then the following conditions are equivalent
\begin{enumerate}
\item[i)] There exists a constant $C_1>0$ such that for all $p\gl 2$
     \[ \|f\|_{L_{\infty}^p(N\subset M)}\kl C_1 \sqrt{p} \|f\|_{Lip_{\Gamma}}
     \pl ; \]
\item[ii)] There exists a constant $C_2>0$ such that for all normal states $\rho$
 \[ \|\rho\|_{\Gamma^*}\kl C_2 \sqrt{D(\rho||E(\rho))} \pl.\]

 \end{enumerate}
\end{theorem}
In the following wesay that $(T_t)$ or its generator $A$ \emph{satisfies $\la$-WA$_1$} if
 \[ \|\rho\|_{\Gamma^*}\kl 2\sqrt{2} \sqrt{\frac{D(\rho||E(\rho))}{\la}} \pl .\]
Note that the factor $2\sqrt{2}$ is chosen so that $\la$-FLSI implies $\la$-WA$_1$ (via $\la$-TA$_2$).

\begin{proof} Fix $\frac{1}{p}+\frac{1}{p'}=1$. Recall that the relative $p$-R\'{e}nyi entropy $D_p(\rho||\si) \lel p'\ln\|\si^{-1/2p'}\rho\si^{-1/2p'}\|_p$ is monotone over $p\in (1,\infty]$ and hence $\displaystyle D_N^p(\rho)=\inf_{\si\in N, \tau(\si)=1} D_p(\rho||\si)$ satisfies
  \[ D(\rho||E(\rho))\kl D_N^p(\rho)\kl p'\ln \|\rho\|_{L_1^p(N\subset M)}\kl \frac{p'}{e\eps} \|\rho\|_{L_1^p(N\subset M)}^{\eps} \pl .\]
for any $\eps>0$. Therefore, we deduce from ii) that
 \[ \|\rho\|_{\Gamma^*} \kl C\sqrt{\frac{p'}{\eps}} \|\rho\|_1^{1-\frac{\eps}{2}} \|\rho\|_{L_1^{p}}^{\frac{\eps}{2}}  \pl .\]
For $\rho=\rho^*\in L_1^p(N\subset M)$, we note the  elementary inequality
 \[ \|\rho^+\|_{L_1^p(N\subset M)} \kl \|\rho\|_{L_1^p(N\subset M)} \pl .\]
Thus for arbitrary $\rho=\rho_1+i\rho_2$, we deduce that
 \[ \max_{j,\pm} \|\rho_j^{\pm}\|_1^{1-\frac{\eps}{2}} \|\rho_j^{\pm}\|_{L_1^p(N\subset M)}^{\frac{\eps}{2}} \kl 4
  \|\rho\|_1^{1-\frac{\eps}{2}} \|\rho\|_{L_1^p(N\subset M)}^{\frac{\eps}{2}} \pl .\]
In other words condition i) implies that
 \[ \|\rho-E_N(\rho)\|_{\Gamma^*}\lel
 \|\rho\|_{\Gamma^*} \kl 4C \sqrt{\frac{p'}{\eps}}
 \|\rho\|_1^{1-\frac{\eps}{2}} \|\rho\|_{L_1^p(N\subset M)}^{\frac{\eps}{2}} \pl .\]
Now, we may use real interpolation theory \cite{BL} and duality deduce
 \begin{equation} \label{ep}
   \|f-E_N(f)\|_{[L_{\infty},L_{\infty}^{p'}(N\subset M)]_{\eps/2,\infty}}
  \kl4C \sqrt{\frac{p'}{\eps}}  \|f\|_{Lip_{\Gamma}} \pl .
  \end{equation}
 In particular, this is true for $\eps=\frac12$. Since $\tau(1)=1$, we deduce that  the inclusion
 \[ [L_{\infty}(M),L_{\infty}^{p'}(N\subset M)]_{1/4,\infty}
 \subset [L_{\infty}(M),L_{\infty}^{p'}(N\subset M)]_{1/2,1} \subset [L_{\infty}(M),L_{\infty}^{p'}(N\subset M)]_{1/2} \]
is of norm bounded by a universal constant $c_0$. Thanks to  \cite{JP}, we have
\[[L_{\infty}(M),L_{\infty}^{p'}(N\subset M)]_{1/2}=L_{\infty}^{2p'}(N\ssubset M)\pl.\] Let $q=2p'$, then we deduce
 \[ \|f\|_{L_{\infty}^q(N\subset M)} \kl 4c_0C \sqrt{2p'}
 \|f\|_{Lip_{\Gamma}} \lel c_1 C \sqrt{q}\|f\|_{\Gamma} \]
for a universal constant $c_1=4\sqrt{2}c_0$. This means $ii)$ implies $i)$ with constant $c_1C$.  Conversely, we deduce from i), and again \cite{JP} that for $\frac{1}{q}=\frac{\eps}{2s}$
 \[   L_{\infty}^{q}(N\subset M)
  \lel [L_{\infty}(M),L_{\infty}^{s}(N\subset M)]_{\eps/2}
  \subset [L_{\infty}(M),L_{\infty}^p(N\subset M)]_{\eps/2,\infty}  \]
This means i) implies that
 \[ \|f\|_{[L_{\infty}(M),L_{\infty}^s(N\subset M)]_{\eps/2,\infty}} \kl C \sqrt{q}\|f\|_{\Gamma}
 \lel \sqrt{2}C \sqrt{\frac{s}{\eps}} \|f\|_{\Gamma} \pl .\]
By duality we deduce for $s=p'$ that
 \[ \|\rho\|_{\Gamma^*}\kl \sqrt{2e}C
 \sqrt{\frac{p'}{e\eps}} \|\rho\|^{1-\frac{\eps}{2}}
 \|\rho\|_{L_1^p(N\subset M)}^{\frac{\eps}{2}} \pl .\]
Thus for a state $\rho$ such that $D(\rho||E(\rho))<\infty$ , we may choose $\eps=(\ln \|\rho\|_{L_1^p(N\subset M)})^{-1}$ and obtain  that
  \[ \|\rho\|_{\Gamma^*}\kl \sqrt{2e}C  \sqrt{p'\ln \|\rho\|_{L_1^p(N\subset M)}} \pl .\]
By sending $p\to 1$, we deduce ii). \qd

\begin{rem}\label{jz}
{\rm It was proved in \cite{JZ} that $\|\rho\|_{\Gamma^*}\kl C \sqrt{\Ent(\rho)}$ is equivalent to \[ \|f-E(f)\|_{L_p(M)}\kl C' \sqrt{p}\|f\|_{\Gamma} \pl .\]
In that sense the estimate with respect to $D_N$ is significantly stronger, because the inclusion ${L_{\infty}^p(N\subset M)\subset L_p(M)}$ is contractive  for all finite $M$.}
\end{rem}

\begin{lemma}\label{fcharac} For a positive density $\rho$
 \[ D(\rho||E(\rho))\lel \sup_\si \tau(\rho (\ln \si-\ln E(\si)))\]
where the supremum is taken over all strictly positive density $\si$.
\end{lemma}

\begin{proof}
Using the convexity of $F(\rho)=D(\rho||E(\rho))=\lim_{p\to 1} \frac{\|\rho\|_{L_1^p(N\subset M)}-1}{p-1}$, we know that
 \[ F(\rho)\gl F(\si)+ F'(\si)(\rho-\si) \]
In the previous section, we observed in \eqref{totder} that the total derivative is \[F'(\si)(\beta)=\tau\Big(\beta(\ln \si-\ln E_N(\si))\Big)\pl.\] Hence we obtain after cancellation that
 \[ F(\rho)\gl \tau(\rho(\ln \si-\ln E_N(\si)))\]
Obviously, we have equality for $\si=\rho$ in case of a state, and hence homogeneity implies the assertion. Note also that we may replace $\si$ by $\si+\delta 1$ to guarantee that the condition is well-defined. The extra scaling factor $\tau(\si)+\eps$ cancels thanks to the logarithm.
\qd

\begin{prop} Let $(T_t)$ be a semigroup as in Theorem \ref{characW}. The condition
\begin{enumerate}
\item[iii)] There exists a $c>0$ such that
 \[ E_N(e^{tf})\le e^{ct^2} \]
for all self-adjoint $f$ with $\Gamma(f,f)\le 1$ and $t>0$.
\end{enumerate}
implies $\la$-WA$_1$ for some $\la$. If in addition $N$ is contained in the center of $M$, then {\rm iii)} is equivalent to WA$_1$. \end{prop}

\begin{proof} Let us assume that iii) holds and that $f=f^*$ satisfies $\Gamma(f,f)\le 1$. We define $\rho=\frac{e^{tf}}{\tau(e^{tf})}$ and deduce that for every state $\psi$ (thanks to the cancelation of the scalar factor)
 \begin{align*}
 D(\psi||E(\psi))&\gl \tau(\psi (\ln \rho-\ln E(\rho))) \lel \tau(\psi (tf-\ln E(e^{tf})) \pl .
 \end{align*}
This implies
 \[ \tau(\psi f) \le \frac{D(\psi||E(\psi))}{t}+
 \frac{\tau(\psi \ln E(e^{tf}))}{t}
 \kl  \frac{D(\psi||E(\psi))}{t}+ct \pl .\]
 Now we may choose $t=\sqrt{\frac{D(\psi||E(\psi))}{c}}$ to deduce the condition ii) in Theorem \ref{characW} with constant $C=2\sqrt{c}$. For the converse, we assume that $N$ is in the center, $f=f^*$ and $\Gamma(f,f)\le 1$ and $E(f)=0$. Then we deduce from condition i) that
  \[ \tau(\si f^p)\le \tau(\si|f|^p)^{1/p} \lel \|\si^{1/2p}f\si^{1/2p}\|_p \kl C\sqrt{p} \pl \]
for all $\si \in N_+$, $\tau(\si)=1$.
$E(f)=0$ implies that the first order term in the exponential expansion vanishes and hence
 \begin{align*}
 \tau(\si E(e^{tf}))
 &\le 1+\sum_{k\gl 2} \frac{(Ct)^k \sqrt{k}^k}{k!} \kl
  1+ \sum_{k\gl 2} \frac{(Cet)^k}{k^{k/2}}
  \kl 1+\sum_{j=1}^{\infty} \frac{(KCet)^2}{j!} \pl .
 \end{align*}
Here we use that for $k=2j$ we have $(2j)^{\frac{2j}{2}}\gl 2^jj^j\gl j!$. A slightly more involved estimate works for $k=2j-1$, $j\gl 2$ and leads to the constant $K$.
\qd

Let us recall the definition of the Orlicz space $L_{\Phi}(M,\tau)$ of a Young function $\Phi$ by the Luxembourg norm
 \[ \|x\|_{L_{\Phi}}\lel \inf\{ \nu| \tau(\Phi(\frac{|x|}{\nu})\le 1\} \pl .\]
It is well-known that for the convex function $\Exp_2(t)=e^{t^2}-1$, we have
 \[ \|x\|_{L_{\Exp_2}} \sim \sup_{p\gl 2} \frac{\|x\|_p}{\sqrt{p}} \pl .\]

\begin{cor}\label{wap} Assume that the generator $A$ satisfies $\la$-WA$_1$. Then
 \[ \|f\|_{L_{Exp_2}}\kl K \la^{-2} \|f\|_{Lip_\Gamma} \pl \]
holds for some universal constant $K$.
\end{cor}

Indeed, we have two ways of proving this. By H\"older inequality we have a contraction
 \[ L_{\infty}^p(N\subset M) \subset L_p(M) \pl .\]
On the other hand, we note that $\la$-WA$_1$ implies the condition $\la$-WA$_1'$ (which still implies the geometric Talagrand inequality)
 \[ \|\rho\|_{\Gamma^*}\kl 2 \sqrt{\frac{\Ent(\rho)}{\la}} \pl .\]
Then Remark \ref{jz} also implies Corollary \ref{wap}.

\begin{rem}{\rm Similar concentration inequalities for a fixed state $\si$ can be found in \cite{DR}. They deduced an estimate for $\tau(\si e^{f})$ using the gradient norm of $\si^{1/2}f\si^{-1/2}$. Here we also need information for $\si^{-1}$, unless $N$ is central.
}\end{rem}

 In \cite{JRZ} the cb-version of  having  finite diameter was used for approximation by finite dimensional systems. It is shown that  the famous rotation algebras $A_{\theta}$ have finite $cb$-diameter. Let us recall that for the intrinsic metric ${\|f\|_{Lip_{\Gamma}}\simeq\|
 \delta(f)\|}$ one can define a natural operator space structure as intersection of a column and a row space in a Hilbert $C^*$-module, or cb-equivalently as a subspace $\delta({\rm dom}(A^{1/2}))\subset \hat{M}$. Thus it makes sense to say that $(\A,\|\pl\|_{Lip_{\Gamma}},M)$ has \emph{finite}  cb-diameter $D_{cb}$  if
 \[ \|I-E_N: (\A,\|\pl\|_{Lip_\Gamma})\to M\|_{cb} \kl D_{cb} \pl .\]

\begin{cor} Let $A$ be a generator of a self-adjoint semigroup on a finite von Neuamann algebra $M$. If $(\A,\|\pl\|_{Lip_{\Gamma}})$ as a quantum metric space has finite cb-diameter,  then $A$ satisfies WA$_1$ for $A\ten id_{\Mz_m}$ for all $m\in \nz$ and $A\ten id_{\tilde{M}}$ for any finite von Neuamm algebra $\tilde{M}$.
\end{cor}

\begin{proof} We just have to note that the inclusion
 \[ L_{\infty}(M\bar{\ten} \tilde{M})\subset L_\infty^{p}(N\bar{\ten}\tilde{M}\subset M\bar{\ten} \tilde{M}) \]
is a contraction. In particular, the norm from the Lipschitz functions is smaller than $ 2D_{cb}\sqrt{p}$. Then Theorem \ref{characW} implies the assertion.
\qd

\begin{rem}
{\rm We see  that both conditions $\la$-CLSI and $D_{cb}<\infty$  imply $\la$-WA$_1$ on all matrix levels, a property we will call $\la$-CWA$_1$. Note that according to Remark \ref{geoT},  $\la$-CWA$_1$ implies the geometric Talagrand inequality on all matrix levels, which we will call matrical Talagrand inequality.}
\end{rem}

Let $(\M,g)$ be a $d$-dimensional compact Riemannian manifold with sub-Laplacian $\Delta_X$ and sub-Riemannian (or Carnot-Caratheodory) metric $d_X$  induced  by a H\"ormander system $X$. This gives a corresponding gradient form
  \[ \Gamma_X(f,f) \lel \sum_{j=1}^k |{X_j}(f)|^2 \pl .\]
For matrix valued functions $f:\M \to \mm_m$, the natural operator space structure is given by
 \[ \|f\|_{\Mz_m(Lip_{\Gamma})}
 \lel \max \{ \|\sum_j |{X_j}(f)|^2\|^{1/2},
 \|\sum_j |{X_j}(f)^*|^2\|^{1/2} \} \pl .\]
Thanks to Voiculescu's inequality this is equivalent to
 \[ \|f\|_{\Mz_m(Lip_{\Gamma})} \sim \|\sum_j  g_j \ten {X_j}(f)\|  \]
where $g_j$ are freely independent  semicircular (or circular) random variables (\cite{Po}). For matrix valued functions it is therefore better to to use the free Dirac operator $D=\sum_j g_j \ten {X_j}$ and the Laplace-Beltrami operator
in contrast to the spin Dirac operator $D=\sum_j c_j\ten X_j$ which is more common in noncommutative geometry \cite{NCG}. Let us now consider a manifold $\M$ with finite diameter $\diam_X(\M)=\sup_{x,y} d_X(x,y)$ and a normalized volume form $\mu$.  Here $d_X$ is the Carnot-Caratheodory distance given by the H\"ormander system (see \cite{ABB}). Let $f:\M\to M$ be an $M$-valued Lipschitz function. Let $h,k\in L_2(M)$. Then $f_{h,k}(x)=(h,f(x)k)$ is a complex valued function and hence (following Connes' \cite{NCG})
 \begin{align*}
 &|(h,(f(x)-\ez_\M f)(k))|
  \lel |f_{h,k}(x)-\int_\M f_{h,k}(y)d\mu(y)| \kl \int_\M |f_{h,k}(x)-f_{h,k}(y)| d\mu(y) \\
 &\le \int (\sum_{j} |{X_j}f_{h,k}|^2)^{1/2} d(x,y) d\mu(y)  \\
 &\le \diam_X(\M)  \sup_z |(\sum_{j} |{X_j}f_{h,k}(z)|^2)^{1/2}|
 \lel \diam_X(\M) \sup_z (\sum  \lan h,{X_j}(f)(z)k\ran)^{1/2}
 \\&\kl\diam_X(\M) \|h\| \|k\|
 \|\sum_j |{X_j}(f)|^2\|^{1/2} \pl .\end{align*}
Actually, the inequality $|f(x)-f(y)|\le \|f\|_{Lip}d_X(x,y)$ follows directly from the definition of the distance using connecting path. Therefore we have shown the following easy fact:
\begin{lemma}\label{cbmetric}    Let $\Delta_X$ be the sub-Laplacian on $\M$ given by a H\"ormander system $X$. Then
 \[ D_{cb}(\Delta_X) \kl \diam_X(\M) \pl .\]
\end{lemma}

\begin{theorem}\label{HWA} Let $X$ be a H\"ormander system on a connected compact Riemannian manifold. Then $\Delta_X$ satisfies CWA$_1$.
\end{theorem}

\begin{proof} According to the Chow-Rashevskii theorem (see \cite[Theorem 3.29]{ABB} and \cite{Rash}), the Carnot-Caratheodory distance
$d:\M\times \M\to \rz$ is continuous with respect to the original topology of the Riemannian metric. Thus, by compactness $\diam_X(\M)$ is finite. Then Lemma \ref{cbmetric} implies the assertion.
\qd

\begin{cor} Let $\L$ be the generator of a self-adjoint semigroup in $\Mz_m$. Then $\L$ satisfies CWA$_1$
\end{cor}

\begin{proof} According to Lemma \ref{H1}, we find a connected compact Lie group $G$ and a generating set $X$ of $\mathfrak{g}$ such that the transference principle \ref{Ho2} applies. That is, via the co-representation $\pi(x)(g)=u(g)xu(g)^{-1}$, $e^{-t\L}$ is a sub-dynamical system of $e^{-t\Delta_X}\ten id_{\Mz_m}$. According to Theorem \ref{HWA}, we know that $\Delta_X$ has $\la(X)$-CWA$_1$ for some constant $\la(X)$, and hence $\L$ inherits this property (compare to \ref{transf}). \qd

\begin{rem}{\rm
We conjecture that on compact Riemannian manifolds the Laplace-Beltrami operator satisfies $\la$-CLSI. However, since $\Gamma\E$ fails in general new techniques will be needed to approach this problem.
}
\end{rem}

\section{Examples and counterexamples}

\subsection{Data processing inequality and free products}

For the sake of completeness let us show that if to generators $A_1$ and $A_2$ of selfadjoint semigroups satisfy $\la$-CLSI (in its von Neumann algebra version), then $A=A_1\ten id+id\ten A_2$ also satisfies $\la$-CLSI. Indeed, $\rho \in L_2(M_1\ten M_1)$ and $E_1$, $E_2$ the conditional expectation onto the fixpoint algebras. Then we deduce from the data processing inequality that
 \begin{align*}
 D(\rho||E_1\ten E_2(\rho))
 &= \tau(\rho\ln \rho)-\tau(\rho\ln E_1\ten E_2(\rho))\\
 &= \tau(\rho\ln \rho)-\tau(\rho\ln E_1\ten id(\rho))
 +\tau(E_1\ten id(\rho)\ln E_1\ten E_2(\rho))\\
 &=D(\rho||E_1\ten id(\rho))+D(E_1\ten id(\rho)||E_1\ten E_2(\rho))\\
 &\le D(\rho||E_1\ten id(\rho))+D(\rho||id\ten E_2(\rho))\\
 &\le \I_{A_1\ten id}(\rho)+\I_{id\ten A_2}(\rho)
 \lel \I_{A}(\rho) \pl .
 \end{align*}
Following the lead of \cite{JZ}, we have performed a quite extensive research on stability of different conditions with respect to free products. The details would add more length this paper, but let us record some facts:
\begin{enumerate}
\item[i)] Assume that $A_j$ and $B_j$ have the same fixpoint algebras. Then $\Gamma_{A_j}\le \Gamma_{B_j}$ shows the generators $A(n)$ and $B(n)$ of the free amalgamated products $e^{-tA(n)}=\ast_{j=1}^n e^{-tA_j}$ and $e^{-tB(n)}=\ast_{j=1}^n e^{-tA_j}$ still satisfy $\Gamma_{A_n}\le \Gamma_{B_n}$.
\item[ii)] For $A_j=(id-E)$ we find for $A(n)$ the generator of the block-length semigroup. For fixed $\nen$, we have
     \[ \Gamma_{id-E}\kl n \Gamma_{A(n)}\pl. \]
\item[iii)] Thus $\la \Gamma_{id-E_j}\le \Gamma_{A_j}$ implies that for a $n$-fold free product, we have
     \[ \frac{\la}{n}\Gamma_{I-E}\kl \Gamma_{A(n)} \pl .\]
Since this holds with amalgamation, we deduce that $\frac{\la}{n}$-$\Gamma\E$ for $n$-fold free products of semigroups satisfying $\la$-$\Gamma\E$.
\item[iv)] The free group in $m$ generators satisfies  $\frac{1}{3m}$-$\Gamma\E$ and $\frac{1}{3m}$-CLSI.
\end{enumerate}

\begin{conc} Although beyond the scope of this paper, let us note that infinite  free and infinite tensor products of complete graphs, or more generally subordinated sublaplacians on compact manifolds give new infinite dimensional examples of von Neumann algebras satisfying Talagrand's Wasserstein estimate TA$_2$ and CLSI, but not necessarily $\Gamma\E$, which remains true for finite free products.
\end{conc}

Let us briefly sketch some arguments for a reader familiar with free probability, see \cite{VDN}. Let $N\subset M_j$ be finite von Neumann algebras with trace preserving conditional expectation $E$. According to \cite{Boc} a family $T_t^j:M_j\to M_j$ which leaves $N$ invariant can be extended to the free product with amalgamation $\M=\ast^j_{N}M_j$ via
 \[ T_t(a_1\cdots a_{m})\lel T_t^{i_1}(a_1)\cdots T_t^{i_m}(a_m)  \]
provided $a_j\in M_{i_j}$ and $i_1\neq i_2\neq \cdots \neq i_m$. We refer to \cite{VDN} for the definition and general facts on amalgamated free products.  In the following we assume that each semigroup  $T_t^j=e^{-tA_j}$ consists of selfadjoint trace preserving completely positive maps.


\begin{proof}[Proof of i)] For simplicity of notation we will assume that all the algebras are the same, and all the generators $A=A_j$ and $B=B_j$ are the same. Our first task is to identify the module for free product.
Let $\delta_{A}$ be the derivation. Then we observe that $\delta_{A_j}(xb)=x\delta_{A_j}(b)$ holds for $x\in N$. 
For a word $\om=a_1\cdots a_m$ so that $a_j\in A_{i_j}$ we define the vectors $\xi^l=(e_{i_1},...,e_{i_l})\in \ell_2(\nz^l)$ and
 \[ v(\om)
 \lel \sum_{l=1}^{m} \xi^l \ten
 u_N(a_1\ten \cdots a_{l-1})\ten \delta_{A}(a_{l})a_{l+2}\cdots a_m \pl .\]
To explain the cancellation let us consider $\om=b_1^*b_2^*$ $b_k\in A_{l_k}$,
and $\om'=a_1a_2a_3$, $a_k\in A_{r_k}$. We use the notation $a^{\circ}$ for the mean $0$ part.  We have the following decomposition in mean $0$ words.
 \begin{align*}
  \om^*\om' &= b_2b_1a_1a_2a_3
  \lel b_2(b_1a_1)^{\circ}a_2a_3+b_2E(b_1a_2)a_2a_3 \\
  &=  b_2(b_1a_1)^{\circ}a_2a_3+ (b_2E(b_1a_2)a_2)^{\circ}a_3 + E(b_2b_2a_1a_2)a_3 \pl .
   \end{align*}
Of course for $l_1\neq r_1$ the two additional terms vanish. Only if $l_1=r_2$ and $l_2=r_2$ we really find $3$ terms. This implies
 \begin{align*}
 A(\om^*\om')
 &=  A(b_2(b_1a_1)^{\circ}a_2a_3)+
 A(b_2E(b_1a_1)a_2)^{\circ}a_3) + A( E(b_2b_2a_1a_2)a_3) \\
 &= A(b_2)(b_1a_1)^{\circ}a_2a_3+b_2A(b_1a_1)a_2a_3 + b_2(b_1a_1)^{\circ}A(a_2)a_3+ b_2(b_1a_1)^{\circ}a_2A(a_3) \\
 &+ A(b_2)E(b_1a_2)a_2a_3+ b_2E(b_1a_2)A(a_2)a_3\\
 & +A(b_2E(b_1a_1)a_2)a_3+ b_2E(b_1a_1)a_2)^{\circ}A(a_3)+
  E(b_2b_2a_1a_2)A(a_3) \pl .
 \end{align*}
We recall  that $2\Gamma(\om,\om') \lel A(\om^*)\om'+\om^*A(\om')-A(\om^*\om ')$ and hence compare this with
 \begin{align*}
 &A(\om^*)\om'+\om^*A(\om')
 \lel  A(b_2)b_1a_1a_2a_3+b_2A(b_1)a_1a_2a_3+b_2b_1A(a_1)a_2a_3+ b_2b_1a_1A(a_2) \\
 &=A(b_2)(b_1a_1)^{\circ}a_2a_3 + A(b_2)E(b_1a_1)a_2a_2a_3
  +b_2(b_1a_1)^{\circ}A(a_2)a_3+ b_2E(b_1a_1)A(a_2)a_3\\
 &\pll +  b_2(b_1a_1)^{\circ}a_2A(a_3)+b_2E(b_1a_1)a_2A(a_3)
  + b_2(A(b_1)a_1a_2a_3+b_2b_1A(a_1))a_2a_3
 \end{align*}
We first observe that the $A(a_3)$ terms cancel, because they can not interact with anything from $b$. If $l_1\neq r_1$, we certainly find $0$. If $l_1=r_1$ and
$l_2\neq r_2$, the the terms for $A(b_2)$ and $A(a_2)$ can not interact and we find
 \[ b_2\Gamma(b_1^*,a_1)a_2a_3 \]
Finally, if $l_1=r_1$ and $r_2=l_2$ we get the additional term
  \[ A(b_2)E(b_1a_2)a_2a_3+b_2E(b_1a_1)A(a_2)a_3-A((b_2E(b_1a_1)a_2)^{\circ})a_3
\lel \Gamma(b_2^*,E(b_1a_2)a_2)a_3 \pl .\]
In full generality, we have to use an inductive procedure and obtain
 \[ \Gamma(b_1\cdots b_m,a_1,...,a_n)\lel
 v(b_1\cdots b_m)^*v(a_1\cdots a_n) \pl .\]
This implies for $x=\sum_{i_1,...,i_k} \om(i_1,...,i_k)$ with $\om(i_1,...,i_k)$ in the linear span $A^{\circ}_{i_1}\cdots A_{i_k}^{\circ}$ that
 \begin{align*}
  \Gamma_B(x,x)&= \!\sum_{\om,\om'}
  \sum_k \delta_{\si_k(\om),\si_k(\om')} (v^B_k(\om),v^B_k(\om'))
 \le \sum_{\om,\om'} \sum_k \delta_{\si_k(\om),\si_k(\om')} (v^A_k(\om),v^A_k(\om')) = \Gamma_A(x,x) \p .
  \end{align*}
Here we used  that for fixed and some elements $\xi_l$
 \[ \sum_{l,l'} \xi_l^*\Gamma_A(\al_l,E_N(\beta_l^*\beta_{l'})\al_{l'})\xi_{l'}
 \kl \sum_{l,l'} \xi_l^* \Gamma_B(\al_l,E_N(\beta_l^*\beta_{l'})\al_{l'})\xi_{l'} \pl .\]
The additional inner product can be obtained by writing $E(x^*y)=\sum_j w_j(x)^*w_j(y)$, see \cite{JD}, and hence
 \[ \Gamma_A(\al,E(x^*y)\beta)
  \lel \sum_j \Gamma_A(w_j(x)\al,w_j(y)\beta) \kl
  \sum_j \Gamma_B(w_j(x)\al,w_j(y)\beta)
  \lel   \Gamma_B(\al,E(x^*y)\beta) \pl \] follows from our assumption. \qd

A particularly interesting case is given by $T_t=e^{-t(I-E_N)}$. The corresponding free product gives the so-called blocklength
 \[ T_t(a_1\cdots a_n) \lel e^{-tn}a_1\cdots a_n \]
for (free) products of mean $0$ terms.


\begin{proof}[Proof of ii)] For $I-E$ the gradient form
 \[ 2\Gamma_{I-E}(x,y) \lel (x-E(x))^*(y-E(y)) + E((x-E(x))^*(y-E(y)) \]
splits into two forms $v_1(x)=(x-E(x))$ and $v_2(x)=u_N(x-E(x))$. Therefore, we may use our argument from above and find two orthogonal forms
 \[ v^1(a_1\cdots a_m)
 \lel \sum_{k=1}^m e_{i_1,...,i_k} u_N(a_1\cdots a_{k-1})a_{k}\cdots a_k \]
and
 \[ v^2(a_1\cdots a_m) \lel
  \sum_{k=1}^m e_{i_1,...,i_k} u_N(a_1\cdots a_{k})a_{k+1}\cdots a_m \pl .\]
We may apply the contraction $id\ten E_N$ to $v^2$ and deduce from the mean $0$ property of the products that
 \begin{equation}\label{EEE}
  E_N(x^*x) \lel |id\ten E_N(v^2(x))|^2 \kl |v^2(x)|^2 \pl .
  \end{equation}
Moreover, let $P_j$ be the projection onto words starting with $i_1=j$. Then we see that for mean $0$ worlds
 \begin{align*}
  x^*x &\lel
 \sum_{j,k} P_j(x)^*P_k(x^*)\kl n \sum_{j} |P_j(x)|^2 &\le n  (\sum_{l\gl 1} \delta_{\si_l(\om),\si_l\om')} (v_l^2(\om),v_l^2(\om'))
 \kl n |v^1(x)|^2 \pl ,
 \end{align*}
because the term $l=1$ exactly corresponds to $i_1=i_1'=j$. Therefore, we find that
 \begin{align*}
  \frac{x^*x+E(x^*x)}{2}
 &\kl  \frac{n}{2} |v^1(x)|^2+ \frac{1}{2}|v^2(x)|^2 \kl
 \frac{n}{2}(|v^1(x)|^2+|v^2(x)|^2) \lel n \Gamma_{(I-E)\ast(I-E)} \pl .
 \end{align*}
By taking words $a_j$  of length $1$ and $E(a_j^*a_j)$ very small, we see that $n$ is indeed optimal.\qd

\begin{theorem}\label{free} Let $A_j$ be generators such that
 \[ \Gamma_{I-E}\kl \Gamma_{A_j} \pl \]
holds for $j=1,...,n$. Then the free product $A^n$ of $\ast_{j=1}^n T_t^j$ satisfies
 \[ \Gamma_{I-E} \kl n \Gamma_{A^n} \pl .\]
\end{theorem}

\begin{cor} Let $l(g_{i_1}^{k_1}\cdots g_{i_m}^{k_m})=\sum_j k_j$ be the word length on the free group with $n$ generators and $A(\la(w))=l(w)\la(w)$ the corresponding generator. Then
 \[ \Gamma_{I-E}\kl 3n \Gamma_{A} \pl .\]
\end{cor}

\begin{proof} We refer to $\Gamma_{|\Delta|^{1/2}}\gl \frac{1}{3}\Gamma_{I-E}$ to section 7.3.
\qd

\subsection{Graphs}

Let $(V,E)$ be a graph with finite vertex set and $w:E\to \rz_{+}$ be a symmetric weight function. Indeed, we may assume that $A=E(\delta^*\delta)$ is given by an inner derivation $\delta(x)=[\xi,x]$ with $\xi$ selfadjoint. This implies that
 \[ \Gamma(f,f)(x) \lel \sum_y \|\xi_{xy}\|^2|f(x)-f(y)|^2 \pl. \]
always given by a set of weights $w_{xy}=\|\pi(e_x)\xi \pi(e_y)\|_2^2$. Thus we find $\delta(f)(x)=(\sqrt{w_{yx}}(f(y)-f(x)))_{y}$.
Let us assume we use the normalized Haar measure $\mu$ for $\ell_{\infty}(V)$. Then we deduce
 \[ (\delta(f_1),\delta(f_2))_{\mu}\lel
 \frac{1}{|V|} \sum_{x} \sum_{y} 2w_{yx} f_1(x)^*f_2(x)-2 \frac{1}{|V|}\sum_x \sum_y w_{yx}f_1(y)^*f_2(x) \pl .\]
This means we have a one to one relations ship between the weights
 \[ w_{xy} \lel \frac{A_{xy}}{2} \]
for $x\neq y$ and the selfadoint weighted Laplacian with diagonal entries $A_{xx}=2\sum_{y\neq x} w_{yx}$. For the ergodic sitation we see that $\Gamma_{I-E}$ has entries $w^{I-E}_{xy}=\frac{1}{2|V|}$.

\begin{conc} For an ergodic (i.e. irreducible) graph Laplacian the condition $\la$-$\Gamma\E$ is equivalent to
 \[ w_{xy}\gl \frac{1}{2|V|} \quad \mbox{ for all } x\neq y \pl .\]
The weights $w_{xy}(\theta)$ for the approximating sequence given by $A^{\theta}$ are strictly positive.
\end{conc}

\begin{proof} For an erodic graph Laplacian we have a spectral gap and $\|T_1:L_1\to L_{\infty}\|\le c_1<\infty$. Thanks to Saloff Coste's argument (see \eqref{sfc}), we find  $\|T_t-E:L_1\to L_{\infty}\|\kl c' e^{-\la_{\min} t}$ for $t\gl 2$. Thus $A^{\theta}$ satisfies $\la(\theta)$-$\Gamma\E$ and hence $w_{xy}(\theta)\gl \frac{\la(\theta)}{2|V|}$ is strictly positive. \qd

\begin{rem} Of course, we expect CLSI for every finite graph. It it also not clear how expander graphs fit into this picture since they are quite opposite to complete graphs, see \cite{BobT} for more information.
\end{rem}

\subsection{Fourier multiplier and discrete groups}
We will use group von Neumann algebras and Fourier-multipliers. This means we consider a discrete group and $T_t(\la(g))=e^{-t\psi(g)}\la(g)$ for a conditional negative function $\psi$, see \cite{BrO},\cite{Boz} for more information. For Fourier multiplies the  gradient form is given by the Gromov distance
 \[ 2K_{\psi}(g,h) \lel \psi(g)+\psi(h)-\psi(g^{-1}h) \pl ,\]
and
\[ \Gamma_{\psi}(\la(g),\la(h))\lel K(g,h) \la(g^{-1}h) \pl .\]
It is very easy to see that for two generators $\psi$ and $\tilde{\psi}$ the relation
 \[  \Gamma_{\psi} \gl_{cp} \la \Gamma_{\tilde{\psi}} \]
is equivalent to
 \[ K_{\psi} \gl \la K_{\tilde{\psi}} \pl \]
in the usual sense of matrices.

Let us now consider a discrete  group $G$ with the normalized trace on $L(G)$ and the conditional expectation onto $\cz$ in $L(G)$ given by the trace.  Then $id-E$ is a Fourier multiplier and
 \[ K_{I-E}(g,h) \lel (1-\delta_{g,1})(1-\delta_{h,1}) (\frac{1}{2}+\frac{\delta_{g,h}}{2}) \pl .\]
It therefore suffices to consider the matrix on $G\setminus \{1\}$.  Let us now consider the specific example $G=\zz$ and $\psi(k)=|k|$ given by the Poisson semigroup. Then
 \[ K(k,j) \lel \frac{|k|+|j|-|k-j|}{2}
  \lel \begin{cases} 0& \mbox{if } k<0<j \mbox{ or } j<0<k \\
                     \min(|j|,|k|) & \mbox{ else. }
                     \end{cases} \pl .\]
Let us consider the matrix $B(j,k)=\min(j,k)$ and $\al_j$ be a finite sequence. Then we see that
 \[ (\al,B(\al))\lel \sum_{j,k} \bar{\al}_j\al_k \min(k,j)
 \lel \sum_{j,k} \bar{\al}_j\al_k \sum_{1\le l\le \min(j,k)}
 \lel \sum_l |\sum_{j\gl l}\al_j|^2  \pl .\]
Using $|a-b|^2\kl 2a^2+2b^2$, we deduce that
 \[ \sum_l |\al_l|^2 \lel \sum_l |\sum_{j\gl l}\al_j-\sum_{j>l}a_j|^2
 \kl 2\sum_l (|\sum_{j\gl l}a_j|^2+ |\sum_{j> l}a_j|^2)
 \kl 4 (\al,B(\al)) \pl .\]
This means $4B\gl 1_{\nz}$, and hence
 \[ 2K_{\psi} \lel 2(B \ten 1_{\ell_2^2}) \gl \frac{1}{2}1_{\zz\setminus 0} \pl .\]
On the other hand, let $\eiz$ the matrix with all entries $1$. Then  we certainly  have
 \[ B\gl \eiz_{\nz}     \]
and therefore
  \[ K_{\psi} \lel B\ten 1_{\ell_2^2} \gl \frac{1}{2} B\ten \kla \begin{array}{cc} 1&1\\ 1&1  \end{array}\mer \lel \frac{1}{2} \eiz_{\zz\setminus \{0\}} \pl .\]
Combining these estimates, we deduce that
  \[ 3K_{\psi} \gl K_{I-E} \pl .\]

\begin{conc} The Poisson group on $L(\zz)=L_{\infty}(\mathbb{T})$ satisfies $\frac{1}{3}$-$\Gamma\E$, and hence $\frac{1}{3}$-CLSI. In particular, the Fourier multiplier associated to $\psi(k_1,...,k_n)=\sum_{j=1}^n |k_j|$ on $\zz^n$ satisfies $\frac{1}{3}$-CLSI, but not $\Gamma\E$ for $n\gl 2$. The free product $L(\ff_n)$  with the word length function still satisfies $\frac{1}{3n}$-$\Gamma\E$.
\end{conc}

\begin{proof} Let $n=2$ and $\al_{jl}=\eps_j\eps_l$ defined for $1\le l,j\le 2M$, so that $\sum_j \eps_j=0$. Then
 \[ (\al,K_{\psi_2}\al)
 \lel (\eps,K_{\psi}\eps)(\eps,\eiz(\eps))+
 (\eps,\eiz(\eps))(\eps,K_{\psi}(\eps)) \lel 0 \pl .\]
On the other hand $(\al,id(\al))=M^2$. The last fact follows from 7.1.iii). \qd

\subsection{Non-additivity of $\I_{id-E_N}$}

We seen above that the data processing inequality
allows for tensorization of CLSI. This is no longer true for the symmetrized divergence of Kullback-Leibler, see \cite{KL}.

\begin{exam} Let $N_j\subset M_j$ be von Neumann subalgebas. The inequality
\begin{equation}\label{rr}
 \I_{id-(E_{N_1}\ten N_2})\le \I_{id-E_{N_1\ten M_2}}
 +\I_{id- E_{M_1\ten N_2}} \pl  .
 \end{equation}
is not valid in general.
\end{exam}

We note that \eqref{rr} is equivalent to
 \[ \tau((E_1\ten id)(x) \ln x)+\tau((id\ten E_2)(x) \ln x)
 \kl \tau(x\ln x)+\tau((E_1\ten E_2(x))\ln x) \pl .\]
Let $N_1=N_2=\cz$, and
$M_1=M_2=\ell_{\infty}(\{1,2,3\})$.  The conditional expectation is given by row and column average.
 Therefore we have decide  whether $\tau((x+1-E_1(x)-E_2(x))\ln(x))$ is always positive for a state $x$. Let $\delta>0$ and
 \[ [x_{ij}] \lel \left[\begin{array}{ccc} \delta &\al &\al \\
 \al& \gamma &\gamma \\
 \al&\gamma &\gamma
 \end{array}\right] \]
where $\al=3/8$ and $\gamma=15/8-\frac{\delta}{4}$. Then the $11$-entry of $1+x_{11}-E_{N_1}(x)-E_{N_2}(x)$ is given by
 \[ 1+\delta-2/3(\delta+\al+\al) \lel \frac{1}{2} +\frac{\delta}{3} \]
Note that $\lim_{\delta\to 0}\gamma=15/8$ stays away from $0$ and hence $1/2(\ln\delta)$ is going to $-\infty$. Thus for $\delta \to 0$ the expression converges to $-\infty$.

\subsection{Failure of Rothaus Lemma for matrix valued functions}

Let  $N=\Mz_n\ten 1\subset \Mz_n\ten \Mz_n$ and we work with the normalized trace. We use the notation $D_N(\rho)=D(\rho||E_N(\rho))$ for the asymmetry measure \cite{Marv,GJLR2}.

\begin{prop} For $n\gl 2$ there exists no constant such that
\begin{equation}\label{bb} D_N(|x|^2) \kl C \tau(xA(x))+D\|x-E_N(x)\|^2 \pl.
\end{equation}
Moreover, there are no constants  $C,D$ such that
\begin{equation}\label{bbb}
  D_N(|x|^2)\kl C D_N(|x-E(x)|^2)+D\|x-E_N(x)\|^2 \pl ,
 \end{equation}
holds for selfadjoint $x$.
\end{prop}

Let us start with the non-selfadjoint element
 \[ y \lel \frac{n}{\sqrt{n-1}} \sum_{j=2}^n |11\ran  \lan jj|  \pl .\]
The corresponding conditional expectation $E=E_N$ satisfies
 \[ E(|y|^2) \lel \frac{n^2}{n-1}
 E(\sum_{j,k} |jj\ran\lan kk|)
 \lel \frac{n}{n-1} \sum_{j=2}^n |j\ran\lan j|
 \lel \frac{n}{n-1} 1_{n-1} \pl .\]
Since $y$ has rank one, we get
 \[ D_N(|y|^2) \lel \tau(|y|^2\ln |y|^2)-\tau(E(|y|^2)\ln E(|y^2|)
 \lel \frac{1}{n^2} n^2\ln n^2 - \ln \frac{n}{n-1}
 \lel 2\ln n-\ln \frac{n}{n-1} \pl .\]
Now we modify this element by considering
 \[ x \lel \al (|1\ran\lan 1| \ten 1) + y \pl \]
by adding an element in $\Mz_n\ten 1$. Thus $x-E(x)=y$. We have to calculate $D_N(|x|^2)$. Let us denote by $f=|1\ran\lan 1|\ten 1$ the projection. First we observe that
  \[ x^*x \lel \al^2f+ \al y +y^*\al + y^*y \] and hence
  \[ E(x^*x)\lel \al^2 |1\ran\lan 1|+\frac{n}{n-1}1_{n-1} \pl ,\]
 where $1_{n-1}=\sum_{j=2}^n |j\ran\lan j|$ has rank $n-1$. This implies
  \[ \tau(E(|x|^2)\ln E(|x|^2)) \lel \frac{\al^2}{n}\ln \al^2+\ln \frac{n}{n-1} \pl .\]
In order to calculate the entropy for $|x|^2$, we write we decompose $f=|11\ran\lan 11|+1_{n-1}\ten |1\ran\lan 1$ the second projection, call it $g$ is orthogonal to the support of $y$,  and hence $x^2$ is equivalent to
 \[ \kla \begin{array}{ccc} \al^2 &n\al&0 \\
                            n\al & n^2 &0\\
                            0&0& \al^2 g
                            \end{array}\mer \pl .\]
The upper corner is of rank $1$ with size $n^2\al^2$ and hence
\[ \tau(|x|^2\ln |x|^2)
  \lel \frac{n^2+\al^2}{n^2}\ln(n^2+\al^2)+\frac{\al^2(n-1)}{n^2}\ln(\al^2) \pl .\]
This yields
 \begin{align*}
 D_N(|x|^2) &=
 \frac{n^2+\al^2}{n^2}\ln(n^2+\al^2)+\frac{\al^2}{n}\ln(\al^2)-\frac{\al^2}{n^2}\ln(\al^2)
 -\frac{\al^2}{n}\ln(\al^2)-\ln\frac{n}{n-1} \\
 &= \ln(n^2+\al^2)+\frac{\al^2}{n^2}\ln(1+\frac{n^2}{\al^2})-\ln\frac{n}{n-1}
 \end{align*}
In order to contradict \eqref{bb} and \eqref{bbb}, we observe that $\tau(xA(x))=\tau(xA(y))=\tau(yA(y))$. Hence the right hand side in \eqref{bb} and \eqref{bbb} is bounded, but the left hand side converges to $+\infty$ for $\al\to \infty$, as long as $n\gl 2$. For selfadjoint $x$ see below. \qed

We will now address cb-hypercontractivity (in the sense of \cite{BK}) at $p=2$ by considering the selfadjoint element
 \[ z \lel \kla \begin{array}{cc} 0&x\\x^*&0 \end{array}\mer  \pl .\]
We have $\tau(|z|^2\ln|z|^2)=\tau(|x|^2\ln |x|^2)$. For the conditional expectation, we find
 \[ E(|z|^2) \lel
 \kla \begin{array}{cc} E(xx^*)&0\\
  0 & E(x^*x)
  \end{array} \mer
  \lel
  \kla \begin{array}{cc} \al^2 f + E(n^2|11\ran\lan 11|)&0\\
  0 & \al^2f+\frac{n}{n-1}1_{n-1}
  \end{array}  \mer \pl .\]
This gives (with the normalized trace)
 \[ \tau(|z|^2\ln |z|^2)
 \lel \frac{\al^2}{2n}\ln \al^2+\frac{1}{2}\ln\frac{n}{n-1}+\frac{\al^2+n}{2n}\ln(\al^2+n) \pl .\]
The new part here is
 \begin{align*}
 D_N(xx^*)
 &= \frac{n^2+\al^2}{n^2}\ln(n^2+\al^2)+\frac{\al^2(n-1)}{n^2}\ln(\al^2)   -\frac{\al^2+n}{n}\ln(\al^2+n)\\
 &= \ln(\frac{n^2+\al^2}{n+\a^2})+\frac{\al^2}{n^2}\ln \frac{\al^2+n^2}{\al^2}
 -\frac{\al^2}{n}\ln(\frac{\al^2+n}{\al^2})\pl. \end{align*}
 Following our previous calculation we find that
\begin{align*}
D_N(|z|^2) &=
 \frac{1}{2}[\ln(n^2+\al^2)+\frac{\al^2}{n^2}\ln(1+\frac{n^2}{\al^2})-\ln\frac{n}{n-1}]\\
 & \pll
 +\frac{1}{2}[\ln(\frac{n^2+\al^2}{n+\a^2})+\frac{\al^2}{n^2}\ln \frac{\al^2+n^2}{\al^2}
 -\frac{\al^2}{n}\ln(\frac{\al^2+n}{\al^2})] \\
 &= \frac{1}{2}\ln(n^2+\al^2) + \frac{1}{2} \ln(\frac{n^2+\al^2}{n+\a^2})+\frac{\al^2}{2n^2}\ln(\frac{n^2+\al^2}{n+\a^2}) \\
  &\pll + \frac{\al^2}{n^2}\ln(1+\frac{n^2}{\al^2})-\frac{1}{2}\ln\frac{n}{n-1}-
 \frac{\al^2}{2n}\ln(\frac{\al^2+n}{\al^2}) \pl .
\end{align*}
In order to keep the last term in check, we choose $\al_n^2=n$ and then we find
 \begin{align*}
 D_N(|z|^2) &=\frac{1}{2}\ln n + \frac{1}{2}\ln(n+1)+\frac{1}{2}\ln(n+1)-\frac{1}{2}\ln 2+ \frac{1}{2n}[\ln(n+1)-\ln 2]\\
 & +\frac{1}{n}\ln(n+1)-\frac{1}{2}\ln\frac{n}{n-1}-\frac{1}{2}\ln 2  \\
 &=\frac{1}{2}\ln n + (1+\frac{3}{2n})\ln(n+1)-(1+\frac{1}{2n})\ln 2 - \frac{1}{2}\ln\frac{n}{n-1} \pl.
 \end{align*}
Note that the $\log n$ term is the optimal rate for entropy as $n\to \infty$, and hence the example is rather extreme. Following the work of \cite{BK}, we may formulate this observation as follows.

\begin{prop} Let $(A_n)$ be sequence of generators on $M_n$ such that $\sup_n \|A_n:L_2(M_n)\to L_2(M_n)\|<\infty$. Then the 2-cb-hyper-contractivity constant also converges to $\infty$.
\end{prop}

For example we may choose $A_n=Id-\tau_n$ which has norm $\le 2$. In fact, we only have to control the behaviour of $A_n$ on some version of the maximally entangled state.

\begin{proof} We recall \cite{BK} that the cb-hpercontractitivity constant $\la_{2}^{cb}$ is the best constant such that
\[  D_{\Mz_n}(|x|^2)\kl 4\la_2^{cb} \tau((id\ten A_n)(|x|)|x|) \lel
4\la_2^{cb} \mathcal{E}(x) \]
allows an estimate with respect to the energy. However, using the derivation calculus, we deduce the well-known estimate (see Lindsay Davies \cite{DaLi})
 \[ \mathcal{E}(|x|)\kl \mathcal{E}(x)
 \lel \mathcal{E}\left (\kla \begin{array}{cc} 0&y\\
 y^*& 0 \end{array}\mer\right)
 \kl \|A_n\|^2  \|y\|_{L_2(M_{n^2})} \kl \|A_n\| \pl .\]
Thus with our choice of $\al_n=\sqrt{n}$, we obtain the lower bound $\ln n\kl C \la_2^{cb}(A_n)\|A_n\|^2$, and hence not both of them can be bounded. \qd

\begin{rem} {\rm A counterexample of similar nature was constructed in \cite{BaRo} which also show that $S_2(S_p)$ is not uniformly convex. Since our example covers a different regime (and in particular the tracial case), we think they are of independent interest.}
\end{rem}

\subsection{Complete Wasserstein $1$-distance}

The complete analogues of the `triangle' inequality for the Wasserstein $1$ distance from \cite{Ma1,Ma2} can be formulated as follows:

\begin{prop} Let $(M_1,\Gamma_{A_1})$ satisfy $C_1^{-1}$-CWA$_1$, and
$(M_2,\Gamma_{A_2})$ satisfy $C_2^{-1}$-CWA$_1$, then $(M_1\ten M_2,\Gamma_{A_1\ten 1+1\ten A_2})$ satisfies $(C_1+C_2)^{-1}$-CWA$_1$.
\end{prop}

\begin{proof} Let $f$ be such that
 \[ \|\Gamma_{A_1\ten 1}(f,f)+\Gamma_{1\ten A_2}(f,f)\| \kl 1 \pl \]
and $E_{N_1\ten N_2}(f)=0$. In particular,
 \[ \|\Gamma_{A_1}(f-E_{N_1\ten M_2}(f),f-E_{N_1\ten M_2}(f))\|\le 1 \]
and  by Kadison's inequality
 \[ \|\Gamma_{1\ten A_1}(E_{N_1\ten M_2}(f),E_{N_1\ten M_2}(f))\|\le  \|E_{N_1}\Gamma_{1\ten A_2}(f,f)\|\kl
  1  \pl. \]
Let $t>0$. By the complete version, we know that
 \[ |\tau(\rho (f-E_{N_1\ten M_2}(f))|\kl  \frac{D(\rho||E_{N_1\ten M_2}(\rho))}{t}+tC_1 \]
and
 \begin{align*}  |\tau(\rho (E_{N_1\ten M_2}(f)-E_{N_1\ten N_2}(f))|
 &=  |\tau(E_{N_1\ten M_2}(\rho)E_{N_1\ten M_2}(f)-E_{N_1\ten N_2}(f))| \\
 &\le
   \frac{D(E_{N_1\ten M_2}(\rho)||E_{N_1\ten N_2}(\rho))}{t}+tC_2
   \end{align*}
Therefore  the triangle and proof of the data processing inequality implies
 \[ |\tau(\rho f)|\kl \frac{D(\rho||E_{N_1\ten N_2}(\rho))}{t} + (C_1+C_2)t \pl .\]
Taking the infimum over $t$ implies the assertion.\qd

This implies that for a tensor products with $\la$-CWA$_1$ $T_t^n=(e^{-tA})^{\ten_n}$, we find $\frac{\la}{n}$-CAW$_1$, which in turn is enough to imply Talagrand's inequality for matrix valued function on $\{-1,1\}^n$ and $[-1,1]^n$, see \cite{L3} for details in the scalar case.

\subsection{Major Open Problems}

\begin{prob} 1) Does every compact Riemanin manifold satisfy CLSI?

2) Does every generator of a selfadjoint semigroup on a matrix algebra satisfy CLSI?

3) Is CLSI stable under free products?
\end{prob}




\newcommand{\etalchar}[1]{$^{#1}$}






\end{document}